\documentclass[oneside,reqno,11pt]{amsart}

\usepackage[top=1in, bottom=1.25in, left=1.25in, right=1.25in]{geometry}

\usepackage{amsmath}
\usepackage{amsfonts}
\usepackage{amssymb}
\usepackage{graphicx}
\usepackage{hyperref}
\usepackage{mathrsfs}
\usepackage{pb-diagram}
\usepackage{epstopdf}
\usepackage{amsmath,amsfonts,amsthm,enumerate,amscd,latexsym,curves}
\usepackage{bbm}
\usepackage[mathscr]{eucal}
\usepackage{epsfig,epsf,epic}
\usepackage{caption}
\usepackage{subcaption}
\usepackage{multirow}
\usepackage{color}
\usepackage[all]{xy}
\usepackage[cjk]{kotex}
\usepackage[utf8]{inputenc}
\usepackage{mathabx}
\usepackage[sc]{mathpazo}
\usepackage{tikz-cd}

\usepackage[normalem]{ulem}
\usepackage{fancybox}

\newtheorem{thm}{Theorem}[section]
\newtheorem{lemma}[thm]{Lemma}

\newtheorem{cor}[thm]{Corollary}
\newtheorem{prop}[thm]{Proposition}
\newtheorem{defn}[thm]{Definition}

\newtheorem{remark}[thm]{Remark}

\numberwithin{equation}{section}

\newcommand{\id}{{\rm id}}

\newcommand{\bL}{\mathbb{L}}

\newcommand{\Z}{\mathbb{Z}}

\newcommand{\C}{\mathbb{C}}

\newcommand{\AI}{A_\infty}

\newcommand{\WT}[1]{\widetilde{#1}}

\newcommand{\Hom}{\mathrm{Hom}}

\newcommand{\Jac}{\mathrm{Jac}}

\newcommand{\HG}{{\widehat{G}}}

\allowdisplaybreaks[1]

\begin{document}

\author[Hansol Hong]{Hansol Hong}
\address{Department of Mathematics \\ Yonsei University \\ 50 Yonsei-Ro \\ Seodaemun-Gu \\ Seoul 03722 \\ Korea} 
\email{hansolhong@yonsei.ac.kr}

\author[Hyeongjun Jin]{Hyeongjun Jin}
\address{Department of Mathematics \\ Yonsei University \\ 50 Yonsei-Ro \\ Seodaemun-Gu \\ Seoul 03722 \\ Korea} 
\email{jhj941211@gmail.com, hj\_jin@yonsei.ac.kr}

\author[Sangwook Lee]{Sangwook Lee}
\address{Department of Mathematics and Integrative Institute for Basic Science \\ Soongsil University \\ 369 Sangdo-ro \\ Dongjak-Gu \\ Seoul 06978 \\ Korea} 
\email{sangwook@ssu.ac.kr}

\title[Orbifold Kodaira-Spencer maps and mirror symmetry for punctured surfaces]{Orbifold Kodaira-Spencer maps and closed-string mirror symmetry for punctured Riemann surfaces}
\begin{abstract}
When a Weinstein manifold admits an action of a finite abelian group, we propose its mirror construction following the equivariant TQFT-type construction, and obtain as a mirror the orbifolding of the mirror of the quotient with respect to the induced dual group action. As an application, we construct an orbifold Landau-Ginzburg mirror of a punctured Riemann surface given as an abelian cover of the pair-of-pants, and prove its closed-string mirror symmetry using the closed-open map twisted by the dual group action.
 \end{abstract}
%\subjclass{...}
%\keywords{.......}
\maketitle
\tableofcontents

%%%%%%%%%%%%%%%%%%%%%%%%%%%%%%%%%%%%%%%%%%%%%%%%%%%%%%%%%%%%%%%%%%

\section{Introduction}
Originating from predictions in string theory, mirror symmetry is a collection of  phenomena which involve various kinds of dualities between symplectic and complex geometries. Since the discovery of striking coincidence between the counts of rational curves in a projective quintic 3-fold and the period integrals in the mirror quintic \cite{Candelas}, the study of mirror symmetry has expanded significantly to  include mirror pairs of more general shapes. We are particularly interested in cases where the (complex) mirror is given as a complex manifold equipped with a holomorphic function known as a potential whose singularities encode complex geometric information. Such a mirror is referred to as a Landau-Ginzburg model, and our focus is its orbifold generalization in the presence of finite group symmetry.%In this paper we will focus on non-CY manifolds which are mostly noncompact.
% Among many versions of mirror symmetry, in this paper we investigate the relation between rings: one is the {\em symplectic cohomology} of a symplectic manifold and the other is the {\em Koszul cohomology} of mirror singularity. 

It has been one of the long-standing questions in mirror symmetry to find an effective equivariant framework when the symplectic manifold $X$ admits a symmetry group $G$. If $G$ is finite, the situation is relatively simple but still interesting, and there are several known approaches that relate mirror symmetry of $X$ and that of its quotient. At least when $G$ is abelian, it seems natural to consider the action of the dual group $\hat{G}$ on the other side of the mirror, and more interestingly, it is observed in broad generality that the mirror of $X$ should be obtained as the $\hat{G}$-quotient of the mirror of the quotient $X/G$ (whenever it defines a sensible geometric space). See for e.g., \cite{BH}.

The construction of a local mirror given in \cite{CHL2} is one instance that manifests the role of the dual group action on the mirror side. When a symplectic manifold $X$ admits an action of a  finite group $G$, we can study a local mirror of $X$ arising from an immersed Lagrangian $\bL$ in $X/G$. The local mirror $Y$ (with a potential $W:Y \to\Bbbk$)  for $X/G$ is given as the (weak) Maurer-Cartan deformation space of $\bL$, where the meaning of `deformation' will be explained shortly. %consisting of all (weak) bounding cochains in the Lagrangian Floer complex $CF(\bL,\bL)$. 

Recall that the Lagrangian Floer theory studies the intersection problem between Lagrangians quantum-corrected by the counts of pseudo-holomorphic disks with  boundaries on the Lagrangians. In particular, one can consider the intersection problem concerning a single Lagrangian $\bL$ with itself, and such counts  produce an $A_\infty$-algebra $CF(\bL,\bL)$ in this case. The (weak) Maurer-Cartan deformation space encodes the algebraic deformation of $CF(\bL,\bL)$ as an $A_\infty$-algebra. More specifically, 
%can be regarded as a way of studying Lagrangian intersection problem by `categorification'. Roughly speaking, if we consider Lagrangian submanifolds as objects and intersection points between them as morphisms, then the counting of pseudo-holomorphic polygons with Lagrangian boundary conditions gives rise to an $\AI$-category. Let $CF(L,L')$ be a morphism space of the category.
%For a single Lagrangian $\bL$, 
it consists of (weak) bounding cochains that are elements $b$ of $CF^1 (\bL,\bL)$ satisfying a certain equation called Maurer-Cartan equation up to a curvature $W(b)$ which will serve as a potential. (See \ref{subsec:BmLF} for more details.)
%
%Weak bounding cochains encode the information about algxebraic deformation of $CF(\bL,\bL)$ as an $A_\infty$-algebra. 

Now, the upstair mirror symmetry for $X$ can be recovered by orbifolding the mirror of $X/G$ with respect to a natural action of $\hat{G} = \hom (G,U(1))$ on $Y$.
% roughly defined as follows.
%We let $\chi \in \hat{G}$ act on a bounding cochain $b$ by $\chi \cdot b = \chi(g) b$ when $b$ is an image of a cycle in the intersection $L \cap g\cdot L$ in $X$ for $L$ projecting to $\bL$ under the quotient map.
%If $\check{X}$ is a (full) mirror of $X/G$, the mirror of $X$ should be given as $\check{X} / \hat{G}$
This point of view is often useful, 
%in studying mirror symmetry of $X$ itself, 
especially in the case it is easier to organize geometric information of the quotient $X/G$ for constructing its mirror space. In summary, one has

\begin{equation*}
\xymatrix{ X \ar@{<.>}[rr]^-{\textnormal{mirror}}  \ar[d]_{G\textnormal{-quotient}}&& Y/\hat{G}(\stackrel{\bar{W}}{\to} \mathbb{C}) \\  X/G \ar@{<.>}[rr]_{\textnormal{mirror}}  && Y(\stackrel{W}{\to} \mathbb{C})  \ar[u]_{\hat{G}\textnormal{-quotient}} }.
\end{equation*}

In this paper, we consider one of the simplest situations of the above kind, a Weinstein manifold $X$ with a free action of a finite abelian group $G$ on it, and explore the analogous questions about closed-string mirror symmetry. Although requiring freeness of the action seems a little restrictive, our discussion still includes reasonable amount of examples in dimension 2 which is our main focus, as open surfaces usually have bigger fundamental groups compared with compact ones. Any finite index subgroup of $\pi_1$ in this case is associated with a covering space $X$ which itself is a Weinstein manifold (by pulling back the structure from downstairs) equipped with the action of the finite deck transformation group. 

The version of mirror symmetry which we want to focus on is {\em closed-string mirror symmetry} which relates the symplectic cohomology of $X$ and the Koszul cohomology of $W$. The symplectic cohomology $SC^*(X)$ is the noncompact analogue of the Hamiltonian Floer cohomology. On chain-level, it is generated by Hamiltonian orbits where one uses a family of Hamiltonians which show a controlled behavior at each cylindrical end of $X$. It has a differential that counts pseudo-holomorphic cylinders between orbits, and is also equipped with the pair-of-pants product.

The main question is to find an appropriate framework to relate Floer theory of $X$ and that of $X/G$. Here is our key idea: we choose paths coherently from inputs to an output on each contributing pseudo-holomorphic map, similarly to a general formulation of $G$-equivariant TQFT appearing in \cite[2.6]{DBMS}. 
These paths efficiently encode lifting information of holomorphic curves in $X/G$ to $X$. One can recover Floer theory of $X$ by assigning a suitable weight to the curve counting downstairs, where the weight can be simply read off from some combinatorial data of these paths. More precisely, the weight is determined by the intersection numbers between the paths and the codimension 1 loci $\Theta$ in $X/G$ away from which the $G$-principal bundle $X \to X/G$ trivializes. We remark that the choice of $\Theta$ is not unique.

\vspace{0.3cm}

Having this as our geometric setup, we aim to understand the closed-string mirror symmetry of $X$ by passing to that of $X/G$ throughout the following steps.

\begin{enumerate}
\item[\bf (a)] {\bf $SH^*(X)$ via $SH^*(X/G)$ with $\hat{G}$-action.}
We first describe $SH^*(X)$ in terms of an enlargement of $SH^*(X/G)$ by the $\HG$-action. More precisely, we consider a `$\HG$-twisted' complex \[SC^*_{\hat{G}} (X/G):=SC^*(X/G) \otimes \Bbbk[\hat{G}]\] 
on which we keep track of lifting information about periodic orbits and pseudo-holomorphic curves in $X/G$.  Our argument mainly relies on basic covering theory here. However, the correspondence between $SC^*(X)$ and $SC_\HG^*(X/G)$ is not as simple as one might expect, since a single  liftable orbit in $X/G$ corresponds to a nontrivial linear combination of orbits in $SC^*(X)$ and vice versa. In addition, the weighted count mentioned above defines a differential on $SC^*_{\hat{G}} (X/G)$ and a product on its cohomology, but they are \emph{different} from algebraically induced ones just by taking semi-direct product. We emphasize that we always consider operations given by $\Theta$-weighted counts of pseudo-holomorphic curves. Now we have the following:

\begin{prop}[Proposition \ref{prop:isomupdownsc}]\label{prop:scupanddown} There exists an isomorphism of chain complexes \[SC^*(X) \to \left(SC^*_{\hat{G}} (X/G)\right)^{\hat{G}}\] such that
\begin{enumerate}
    \item it only depends on the choice of base points for  the covering $(X,x_0) \to (X/G , \bar{x}_0)$,
    \item it intertwines products on cohomology.
\end{enumerate}

%{\color{red} \begin{enumerate}
%\item[(i)] After fixing base points for the covering $(X,x_0) \to (X/G , \bar{x}_0)$, one has a natural $\hat{G}$-action on $SC^\ast (X)$ commuting with the obvious geometric $G$-action on it. 
%\item[(ii)] Taking the  average with respect to $G$ action (further twisted by the $\hat{G}$-action) gives rise to a natural isomorphism $SC^\ast(X) \to \left(SC^\ast_{\hat{G}} (X/G)\right)^{\hat{G}}$.
%%There exists a natural isomorphism $SC^\ast(X) \to \left(SC^\ast_{\hat{G}} (X/G)\right)^{\hat{G}}$ which takes the (twisted) average with respect to some natural $\hat{G}$-action on $SC^\ast(X)$. 
%In particular, the $\chi$-eigenspace is mapped into $SC^\ast (X/G) \otimes \chi$ for each $\chi \in \hat{G}$.
%\end{enumerate}}
\end{prop}

%This allows us to interprete the mirror symmetry for $X$ with that of $X/G$. 

\item[\bf (b)] {\bf  Construction of the mirror of $X/G$.}
We use the construction of \cite{CHL} to find a local mirror of $X/G$. This involves a choice of `good enough' Lagrangian $\bL$, which is unfortunately not canonically given, but rather experimental.\footnote{Heuristically, one may argue using Koszul duality \cite{Hongmc} that the Maurer-Cartan deformation of a Lagrangian skeleton encodes the mirror geometry, but it is generally a nonsense for too singular skeletons.}
Nevertheless, once cleverly chosen, $\bL$ provides not only a Landau-Ginzburg (LG for short) mirror but also a closed-string $B$-model invariant canonically associated to this LG model. Let us denote by $W:Y  \to \Bbbk$ the Lagrangian Floer potential of $\bL$.  

We will work with the closed-string $B$-invariant $CF((\bL,b),(\bL,b))$ with $b$ viewed as a variable varying over the weak Maurer-Cartan space $Y:=MC_{weak} (\bL)$. (See \ref{subsec:BmLF} for its precise definition.) It is essentially a part of the Hochschild cohomology of the $A_\infty$-algebra $CF(\bL,\bL)$ up to some modifications, but is computationally much simpler than the Hochschild cohomology itself, as it can be formulated nearly in terms of operations on the original Lagrangian Floer complex (in a similar spirit to \cite{Smith:2023aa}). 
%We denote by $W:Y  \to \Bbbk$ the Floer potential of $\bL$. The Landau-Ginzburg model $(Y,W)$ will serve as a mirror of $X/G$.
When $X/G$ is a pair-of-pants and $\bL$ a suitably chosen immersed circle (appearing in Figure \ref{fig:seilag}),  the cohomology $HF((\bL,b),(\bL,b))$ will be shown to be isomorphic as a ring to a certain singularity invariant of the potential $W=xyz$, called the \emph{Koszul cohomology} $Kos(W)$ (see \ref{subsec:orbLGB}).

\item[\bf (c)] {\bf Closed-string mirror symmetry for $X/G$.}
The downstair mirror symmetry for $X/G$ will follow from the natural ring homomorphism
\begin{equation}\label{eqn:KSdown}
KS: SH^*(X/G) \to HF((\bL,b),(\bL,b)),
\end{equation}
whose underlying chain-level map $\mathfrak{ks}:SC^*(X/G) \to CF((\bL,b),(\bL,b))$ is nothing but a closed-open map, counting punctured pseudo-holomorphic disks. 
We refer to this map as the Kodaira-Spencer map, following the terminology of  \cite{FOOO10b}. 
%which is nothing but a component of the closed-open map
%$$\mathfrak{ks}: SC^\ast(X/G) \to CF((\bL,b),(\bL,b))$$
%on the chain-level, 

In fact, the map $KS$ is a noncompact analogue of the ring homomorphism from the quantum cohomology of a toric manifold to the Jacobian ring $\Jac(W)$ of its LG mirror $W$ appearing in \cite{FOOO10b}. Our situation is significantly different from \cite{FOOO10b} in two aspects due to the fact that $W$ has nonisolated singularities in general: (i) our closed-string $B$-invariant $Kos(W)$ is infinite dimensional over $\Bbbk$ (as is $SH^*(X/G)$), and (ii) it has higher degree components while $\Jac(W)$ sits as the degree-$0$ component of $Kos(W)$.
%For the pair-of pants $X/G$, 
\begin{prop}[Corollary \ref{cor:basecase}]\label{prop:popisomintro}
If $X/G$ is a pair-of-pants, the map \[KS :SH^*(X/G) \to HF((\bL,b),(\bL,b))\]
%(as well as its underlying chain-level map) 
is a $\HG$-equivariant ring isomomorphism.
%
%{\color{red} \vspace{0.3cm}
%NOTATIONS	:     $\qquad \mathfrak{ks} : SC \to CF, \qquad  KS:SH \to HF$}
\end{prop}
In this case, by composing  $HF((\bL,b),(\bL,b)) \cong Kos(W)$ in {\bf (b)}, we have
\begin{equation*}
 SH^*(X/G) \stackrel{\cong}{\to} Kos(W)
\end{equation*}
which will nicely match closed-string $A$-model invariant of $X/G$ and $B$-model invariant of the mirror potential $W=xyz$. 
%We note that this theorem is the most important technical foundation of all subsequent results in this paper concerning punctured Riemann surfaces, and that this kind of Kodaira-Spencer type maps from symplectic cohomology has not been computed rigorously in any literature as far as authors know.

\item[\bf (d)] {\bf Closed-string mirror symmetry for $X$.}
We finally lift the argument to the equivariant setting (with help of $\hat{G}$-equivariance of $KS$ \eqref{eqn:KSdown}) to study the mirror symmetry between $X$ and the orbifold LG model $\bar{W}: Y/\hat{G} \to \Bbbk$. It will be obtained as the $\hat{G}$-invariant part of the map between 
\begin{equation}\label{eqn:equivcpx}
\mathfrak{ks}_{\hat{G}}: SC^*_{\hat{G}} (X/G ) \to  CF_{\hat{G}} ((\bL,b),(\bL,b))
\end{equation}
which turns out to be $\hat{G}$-equivariant as well.
The $A_\infty$-operations on \[CF_{\hat{G}} ((\bL,b),(\bL,b)):=CF ((\bL,b),(\bL,b)) \otimes \Bbbk[\hat{G}]\] count holomorphic disks on $\bL$ as usual, but the count is additionally twisted by weights coming from the intersection between $\Theta$ and the chosen paths from inputs to the output. Again, we note that the $\AI$-operations are not the algebraically-modified ones for semi-direct product. An analogous twist should also apply to the count of punctured disks for $\mathfrak{ks}_{\hat{G}}$, and this is crucial for it to become a ring homomorphism. Note that $SC^*_{\hat{G}} (X/G)$ and $CF_{\hat{G}} ((\bL,b),(\bL,b))$ contains $SC^*(X/G)$ and $CF((\bL,b),(\bL,b))$ respectively, corresponding to the identity element of $\hat{G}$. Extra factors they have can be viewed as the analogue of twisted sectors in the orbifold cohomology $H_{orb}^{\ast}$, and we may still call them twisted sectors for this reason.
We shall see that the $\HG$-invariant subalgebra $HF_{\hat{G}} ((\bL,b),(\bL,b))^\HG$ computes a certain singularity invariant of $\bar{W}$ (as known as the {\em orbifold Koszul algebra}), denoted by $Kos (W,\hat{G})$ (Definition \ref{def:orblgB}). 
%, when $X$ is an abelian cover of the pair-of-pants.
\footnote{For a special choice of $\Theta$, $CF_{\hat{G}} ((\bL,b),(\bL,b))$ can be made identical to the orbifold closed-string $B$-model appearing in \cite{CLe}, but is only quasi-isomorphic to that in general.}

\end{enumerate}

We want to give a more explicit account of the mirror symmetry for punctured Riemann surfaces $X$ given as an abelian cover of the pair-of-pants $P=X/G$ where $G$ is the deck transformation group.
Most of known results on its mirror symmetry, such as \cite{Lee} or \cite{CHL_gl}, use a pair-of-pants decomposition and local-to-global construction of the mirror space.  The resulting mirror can be described as a LG model on a toric Calabi-Yau, glued up from the local piece isomorphic to $W:\mathbb{C}^3 \to \mathbb{C}$ for $W (x,y,z) = xyz$. 
% The other direction of mirror symmetry was investigated in, for e.g., \cite{CLL}, \cite{Abouzaid:2021aa}, etc. 
Alternatively, we construct the mirror of $X$ by orbifolding that of $P$ via the $\hat{G}$-action, and use the corresponding variation of the Kodaira-Spencer map to prove the closed-string mirror symmetry for $X$ as summarized in \eqref{eqn:diaequivsh}.

\begin{equation}\label{eqn:diaequivsh}
\xymatrix{ SH^*(X) \ar[rr]^{KS_{\hat{G}} \hspace{1cm}}  && Kos( \bar{W}: \mathbb{C}^{3}/\hat{G} \to \mathbb{C} ) \\  SH^*(P) \ar@{.>}[u]^{\textnormal{Prop. \ref{prop:scupanddown}}} \ar[rr]_{KS \hspace{1cm}}  && Kos( W=xyz:\mathbb{C}^3 \to \mathbb{C} ) \ar@{.>}[u]_{\hat{G}\textnormal{-orbifolding}} }
\end{equation}

%The above construction provides an alternative approach for surfaces $X$ equipped with a finite abelian symmetry whose associated quotient is a pair-of-pants $P$. 
 We choose an immersed circle $\bL$ in $P$ with three transversal self-intersection points as depicted in Figure \ref{fig:seilag}, often referred to as the Seidel Lagrangian, whose (weak) Maurer-Cartan deformation successfully recovers the known mirror $W=xyz$ on $\mathbb{C}^3$.\footnote{The usage of $\mathbb{C}$-coefficients in the mirror construction will be justified in Lemma \ref{lem:seidelLagoverc}.} 
%We first show in \ref{subsec:popms} that the Kodaira-Spencer map becomes a ring isomorphism for $P$. 
The critical loci of $W$ is the union of three coordinate axes in $\mathbb{C}^3$, and the Kodaira-Spencer map in Proposition \ref{prop:popisomintro} identifies nonconstant orbits appearing on each cylindrical end of $P$ with algebraic de Rham forms on the corresponding coordinate axis in $crit(W)$.
With this understanding of $P$, the rest of question for its abelian covers is to inspect twisted sectors arising from $\hat{G}$-quotient. Namely, the orbifold LG mirror is the global quotient of 
$$W : \mathbb{C}^3 \to \mathbb{C} \quad (x,y,z)\mapsto xyz$$ 
by the $\hat{G}$-action preserving $W$. As mentioned, the closed-string $B$-model invariant $Kos (W,\hat{G})(\cong HF_{\hat{G}} ((\bL,b),(\bL,b)))$ of $W/\hat{G} : [\mathbb{C}^3/\hat{G}] \to \mathbb{C}$ has extra components associated with fixed loci of the action as well as the untwisted sector concerning $W$ on $\mathbb{C}^3$ itself. 
%(the latter can be also thought of as being fixed by $1$). 

\begin{thm}[Corollary \ref{cor:final}]\label{thm:mainksrs}
Let $G$ be a finite abelian group.
For a principal $G$-bundle $X \to P$ of the pair-of-pants $P$, we have an isomorphism
\begin{equation}\label{eqn:mainisom}
 SH^*(X) \stackrel{\cong}{\to} Kos(W,\hat{G}) 
\end{equation}
obtained by a sequence of quasi-isomorphisms 
$$ SC^*(X) \to SC^*_{\hat{G}} (P)^{\hat{G}} \to CF_{\hat{G}} ( (\bL,b),(\bL,b))^{\hat{G}} \to K^\ast (W,\hat{G})$$
where $ K^\ast (W,\hat{G})$ is a cochain complex underlying $Kos(W,\hat{G})$.  In particular, if we transfer the algebra structure from $CF_{\hat{G}} ( (\bL,b),(\bL,b))^{\hat{G}}$, then we get a new product structure on $Kos (W,\hat{G})$\footnote{This new product structure can also be intrinsically formulated using endomorphisms of the twisted diagonal matrix factorization. See \cite{Leee}.}  that makes \eqref{eqn:mainisom} into a ring isomorphism.
\end{thm}

It seems that the closed-string mirror symmetry on this geometric setup has not been studied much, although it is not very hard to explicitly compute and match both sides of \eqref{thm:mainksrs}. (For example, $SH^*(X)$ for surfaces $X$ can be easily calculated using the result of \cite{GP20}.)
To authors' knowledge, this is the first place where the closed-string mirror symmetry for noncompact symplectic manifolds is shown by this Kodaira-Spencer type approach, while we have a few preceding results in the compact case \cite{FOOO-T} and \cite{ACHL}. In fact, the latter deals with $\mathbb{P}^1_{a,b,c}$ that is an orbifold compactification of the pair-of-pants, and its quantum cohomology can be viewed as a deformation of $SH^*(P)$ by the \emph{Borman-Sheridan class} $\mathcal{BS}$ \cite{Ton}, \cite{BSV}. One natural question along this line is to see how the image of $\mathcal{BS}$ under $\mathfrak{ks}$ deforms the LG mirror of $P$.
Also, the recent work of \cite{Jeffs:2023aa} proves the closed-string mirror symmetry for nodal curves. It would be interesting to find an extension of our construction to part of these nodal cases.

\vspace{0.5cm}

We finally give an intuitive description of the geometry of our orbifold LG mirror in Theorem \ref{thm:mainksrs} in the following two illustrative examples. (See \ref{subsec:exexex} for rigorous treatments.) We shall see that the general shape of the mirror for abelian covers of $P$ is roughly a mixture of these two types.

\subsubsection*{\bf Sphere with four punctures}
Let us first take a four punctured sphere $X$ (Figure \ref{fig:4punc} (a)) whose quotient by the rotation action of $G=\mathbb{Z}/2$ gives the pair-of-pants $P$. The mirror orbifold LG model is given as the potential $W$ equipped with the action of $\hat{G}=\mathbb{Z}/2$ generated by $(x,y,z) \mapsto (-x,-y, z)$. In Figure \ref{fig:4punc} (b), we heuristically describe the twisted sectors of $[crit(W) / \hat{G}]$. Its untwisted sector corresponds to the original critical loci $crit(W)$, but $[crit(W) / \hat{G}]$ carries an additional component, a coordinate axis isomorphic to $\mathbb{C}$, since the $z$-axis is fixed by the nontrivial element in $\hat{G}$. This accounts for the fact that $X$ has one more cylindrical end compared to $P$. We emphasize however that the additional component in $[crit(W) / \hat{G}]$ does not precisely correspond a single cyclindrical end on $A$-side, but rather a certain eigenspace in $SC^*(X)$ with respect to the $G$-action.
%in accordance with Proposition \ref{prop:scupanddown}.

\begin{figure}[h]
\includegraphics[scale=0.5]{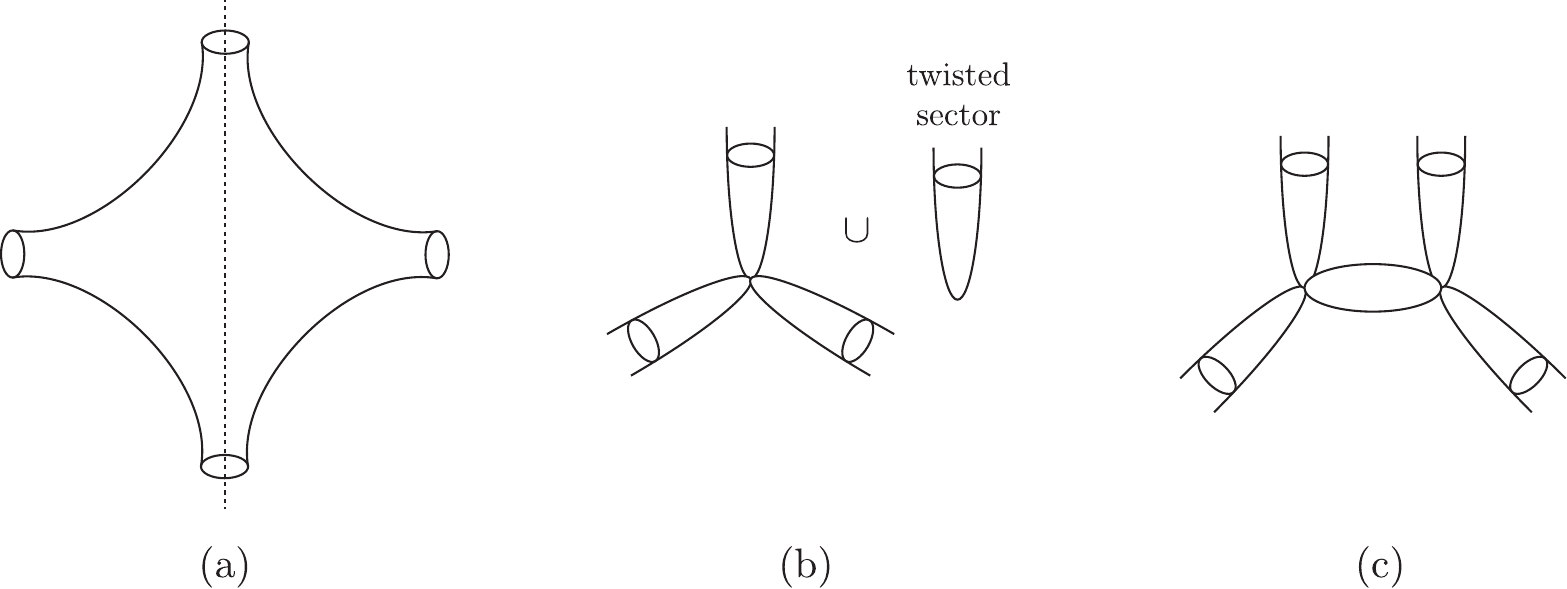}
\caption{Mirrors of the four-punctured sphere}
\label{fig:4punc}
\end{figure}

One can find a similar feature in its toric Calabi-Yau mirror. Through pair-of-pants decomposition, one obtains a mirror LG model $\tilde{W}$ on $O_{\mathbb{P}^1}  \oplus O_{\mathbb{P}^1} (-2)$ which is a resolution of the quotient $\mathbb{C}^3 / \hat{G}$. The critical loci of $W$ are its codimesion 2 toric strata as depicted in Figure \ref{fig:4punc} (c). The only difference from the previous picture of the orbifold mirror is the $\mathbb{P}^1$ component in $crit(\tilde{W})$, which does not actually affect as this $\mathbb{P}^1$ can only support constant (holomorphic) functions.
 
\subsubsection*{\bf Torus with three punctures}
We next look at the example of a tri-punctured elliptic curve $X$ (Figure \ref{fig:3puncell} (a)), which can be explicitly written as a hypersurface of $(\mathbb{C}^\ast)^2$ defined by $z_1 + z_2 + \frac{1}{z_1 z_2} =1$. It admits an action of $G= \mathbb{Z}/3$ by multiplying the third of unity diagonally to $z_1$ and $z_2$, and the quotient is again the pair-of-pants $P$. $X$ is also well-known as a generic fiber of the mirror $\mathbb{P}^2$, and all the three vanishing cycles project to the Seidel Lagrangian $\bL$ in $P$. (See \cite[Section 4]{AKO} or \cite[Section 6]{CHL} for more details.\footnote{More precisely, \cite{AKO} deals with the mirror of a general del Pezzo surfaces, and our case corresponds to $k=0$ therein. \cite{CHL} considers an orbifold compactifcation of ours.})

\begin{figure}[h]
\includegraphics[scale=0.5]{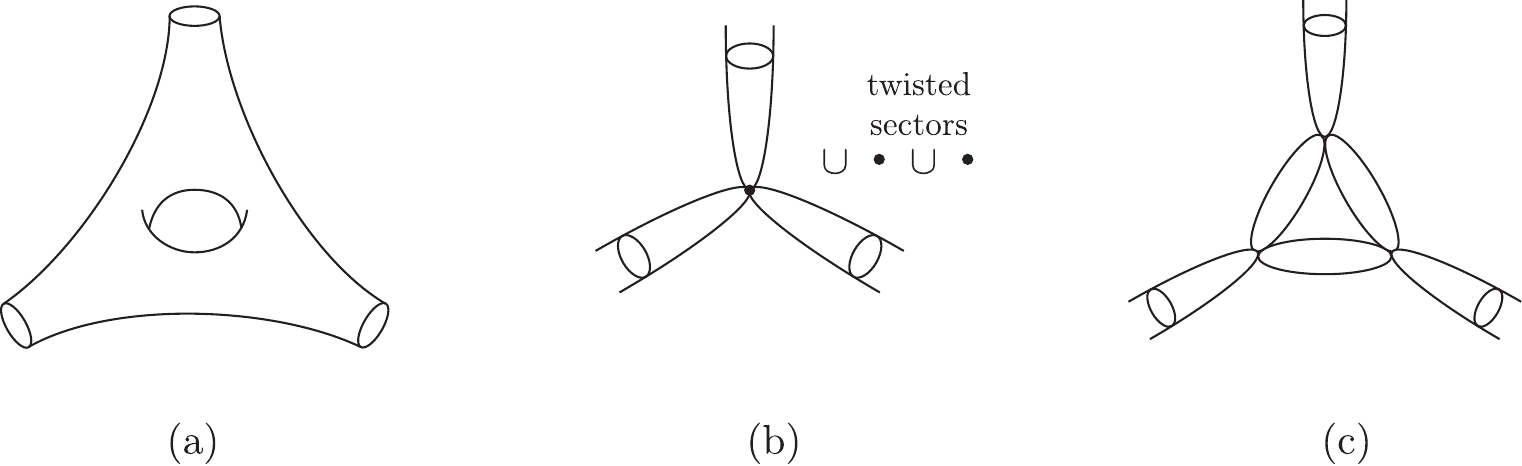}
\caption{Mirrors of the tri-punctured torus}
\label{fig:3puncell}
\end{figure}

Through our equivariant construction, its orbifold LG mirror is $[\mathbb{C}^3/\hat{G}]$ where the action is generated by $\chi \cdot (x,y,z) = (e^{2\pi i /3} x, e^{2\pi i /3} y, e^{2\pi i /3} z)$. The twisted sectors of $crit(W) /\hat{G}$ in this case does not have an extra $\mathbb{C}$-component which is consistent with the fact that $X$ only has three cylindrical ends. Instead it contains, in addition to $crit(W)$ (the untwisted sector), two isolated points supported over $(0,0,0)$ as twisted sectors associated with $\chi$ and $\chi^2$ in $\hat{G}$ (both of which fix $(0,0,0)$). 

On the other hand, the gluing construction produces the LG model on $K_{\mathbb{P}^2}$ as a mirror, whose associated critical loci is given as in (c) of Figure \ref{fig:3puncell}. Observe that it has (degenerate) elliptic curve as a part of the codimension 2 toric strata as well as three legs each isomorphic to $\mathbb{C}$. We believe that degree $1$ cocycles of this elliptic curve defined in a suitable sense (presumably in some cohomology group carrying a mixed Hodge
structure) should correspond to the two isolated points in the twisted sector in the orbifold mirror.

\vspace{0.5cm}

Extension of the construction in this paper to the case of a nonabelian symmetry $G$ on $X$ primarily requires a replacement of the equivariant Floer complex $CF_{\hat{G}} (\bL,\bL)$. 
%
%Indeed, one can identify this complex with the Floer complex of the union of Lagrangian branes supported over the same Lagrangian $\bL$ with various local systems (flat line bundles) coming from characters $\chi \in \hat{G}$. More precisely, a character $\chi$ determines a flat line bundle over $\bL$ which is nothing but the diagonal quotient of $\tilde{\bL} \times \mathbb{C}$. Here, $\tilde{\bL}$ is a chosen lifting of $\bL$ in $X$. By definition, the corresponding parallel transport jumps precisely by $\chi(g)$ at the self intersection point $x$ of $\bL$ if $x$ is the image of both $\tilde{x}$ and $g \cdot \tilde{x}$ for some $\tilde{x} \in \tilde{\bL}$. Along this line, it seems unavoidable to include higher rank local systems associated irreducible representations of $G$ when it is nonabelian. We expect that these lead to a new type of the orbifolded mirror LG model of $X$ when correctly performed, and leave for a future investigation.
%
Observe that $\hat{G} = \hom (G, U(1))$ can be thought of as the set of 1-dimensional (hence irreducible) representations of $G$. In fact, one can reformulate $CF_{\hat{G}} (\bL,\bL)$ in terms of flat line bundles on $\bL$ induced from these representations. Along this line, it seems unavoidable to include higher rank local systems associated higher dimensional irreducible representations of $G$ when $G$ is nonabelian. We expect that these lead to a new type of the orbifolded mirror LG model of $X$, and leave for a future investigation.

\vspace{0.3cm}

\begin{center}
{\bf Acknowledgement}
\end{center}
We would like to thank Cheol-Hyun Cho, Dahye Cho, Jungsoo Kang, Myeonggi Kwon, Jae Hee Lee, Kyoung-Seog Lee, Weiwei Wu for their valuable comments. We are grateful to the anonymous referee for a careful reading and many suggestions
that improved the quality of exposition.
The work of the first named author is partially funded by the Korea government (MSIT) (No. 2020R1C1C1A01008261 and  2020R1A5A1016126). The work of the third named author is supported by the Korea government(MSIT)(No. 202117221032) and Basic Science Research Program funded by the Ministry of Education (2021R1A6A1A10044154).

\section{Preliminaries}
In this section, we briefly review basic materials in symplectic cohomology and Lagrangian Floer theory relevant to our construction below, mainly to fix notations and clarify our geometric setup.
We will follow \cite{abo, RS} for the former, and \cite{FOOO} for the latter. 

Let $X$ be a Liouville manifold with an exact symplectic form $\omega= d\lambda$.
Then $X$ can be decomposed into a compact domain $X^{in}$ with contact boundary $(Y=\partial X^{in},\lambda|_Y)$ and its complement $X\setminus X^{in}$ symplectomorphic to the symplectization
\begin{equation}\label{cylindrical end}
(Y \times [1,+\infty),\>d(R \, \lambda|_Y)).
\end{equation}
where $R$ is the coordinate on $[1,+\infty)$ and $\lambda|_Y$ serves as a contact form on $Y$.
A component of $X\setminus X^{in}$ is called a cylindrical end.
In particular, when $X$ is of dimension $2$, each cylindrical end in $X$ is simply
\[C_1=(S^1\times [1,+\infty),\>  dR\wedge d\theta)\]
up to symplectomorphism.

\subsection{Symplectic cohomology}
Choose a generic Hamiltonian $H:X \to \mathbb{R}$ such that $H|_{X^{in}}:X^{in}\to \mathbb{R}$ is a $C^2$-small Morse-Smale function.
%{\color{red}DO WE NEED THIS? which is identically zero on the boundary, that is, but strictly negative in the interior of $X^{in}$}.
We assume that the slope of $H$ is linear in $R$ when $R$ is large enough with a small slope that the only 1-periodic orbits are critical points of $H|_{X^{in}}$. 
%in the direction of the Liouville vector field is a small enough number $\epsilon>0$ to guarantee that there is only constant 1-periodic orbits of $H$ in $X^{in}$.
%On the cylindrical end $Y\times(1,+\infty)$, we let $H$ be conical perturbed linear.
%Explicitly, we shall use a model Hamiltonian
%\[H=[\epsilon(r-1)]^{1-\rho(r)}[rp(y,t)]^{\rho(r)}\]
%where $p(y,t)\approx1$ is an arbitrarily close to 1(or log $C^2$-small) perturbation term defined on $Y\times S^1$ that makes $wH$ to have "non-degenerate" 1-periodic orbits on the cylindrical ends for all generic $w \geq 1$.
%Here the interpolation exponent $\rho(r)$ is a non-decreasing smooth function defined as
%\[
%\rho(r)=\left\{\begin{array}{ll} 0 & 1 < r \leq 1+\epsilon, \\ 1 & r \geq 2-\epsilon, \\ \frac{\exp\left(\frac{2r-3}{(r-1-\epsilon)(2-\epsilon-r)}\right)}{1+\exp\left(\frac{2r-3}{(r-1-\epsilon)(2-\epsilon-r)}\right)} & 1+\epsilon<r<2-\epsilon, \end{array}\right.
%\]
%whence it smoothly increases the interior autonomous Hamiltonian $F$ to the conical perturbed linear Hamiltonian $rp(y,t)$.
%Note that 1-periodic oribts of $H$ inside $X^{in}$ are exactly the critical points of $F$.
We then consider the family $\{w H\}_{w \in \mathbb{Z}_{>0}}$ of Hamiltonians and their associated Hamiltonian Floer cohomologies.
In what follows, the same notation $wH$ will denote (by abuse of notation) its perturbation by a small time-dependent function supported near its 1-periodic orbits in order to break $S^1$-symmetry. 
%In the case of $\dim_\mathbb{R} X = 2$, 
We obtain exactly two nondegenerate orbits from each $S^1$-family of 1-periodic orbit of (original) $wH = wH(R)$ after the perturbation. 
See \cite[Lemma 2.1]{CFHW} for more details.

As usual, the Floer complex $CF(wH)$ is generated by 1-periodic Hamiltonian orbits over the coefficient field $\Bbbk$ (taken to be $\mathbb{C}$ in our application), and we write $\mathfrak{d}$ for its differential that counts pseudo-holomorphic cylinders asymptotic to two orbits (of degree difference 1) at $\pm \infty$.
Now the symplectic cochain complex is defined as
\[SC^*(X):= \bigoplus_{w=1}^\infty CF(wH)[\bold{q}]\]
equipped with the differential $\mu_1:SC^*(X)\to SC^{*+1}(X)$ given by
\[\mu_1(x+\bold{q}y)=(-1)^{|x|} \mathfrak{d} x+(-1)^{|y|}(\bold{q} \mathfrak{d} y+\mathfrak{K}y-y).\]
where $\bold{q}$ is a formal variable of degree $-1$ such that $\bold{q}^2=0$.
Here $\mathfrak{K}:CF(wH)\to CF((w+1)H)$ denotes the Floer continuation maps (which is not an isomorphism in this case). 
One can alternatively define $SC^* (X)$ as a direct limit 
$$ SC^*(X) :=  \lim_{w \to \infty} CF (wH)$$
of the family $\{CF(wH)\}_{w \geq 1}$ under this continuation map.

We can make $SC^* (X)$ into an $A_\infty$-algebra extending $\mu_1$, for e.g., by letting $\mu_k$ counts pseudo-holomorphic maps from spheres with input/output punctures aligned along the equator. We will only consider $\mu_1$ and $\mu_2$ later, and hence, in practice, it is enough to restrict ourselves to the case $k \leq 2$.
To construct a moduli of such maps, one needs to make extra choices such as
 cylindrical end parametrization near each puncture, and a geodesic from each input puncture to the output puncture with a specified point on it (called flavor and sprinkle), which is to be partially discussed in the proof of Lemma \ref{lem:equatorsh} below. More details can be found in \cite[Section 4]{RS}.
%denotes the direct sum of orientation lines over $\Bbbk$ that corresponds to periodic orbits of $wH$:
%\[CF (wH) = \bigoplus_{x:\textrm{periodic orbit}}|o_x|.\]
%Then, with the Floer continuation maps $\mathfrak{K}:CF(wH)\to CF((w+1)H)$, the symplectic cochain complex is defined as
%\[SC^\ast(X):= \bigoplus_{w=1}^\infty CF(wH)[\bold{q}]\]
%where $\bold{q}$ is of degree -1 such that $\bold{q}^2=0$ with the differential $\partial:SC^\ast(X)\to SC^{\ast+1}(X)$ given by
%\[\partial(x+\bold{q}y)=(-1)^{|x|}\partial x+(-1)^{|y|}(\bold{q}\partial y+\mathfrak{K}y-y)\]
%In fact, we will assign to each pseudo-holomorphic map for $\mu_k$ the following extra datum indicating a special alignment of the punctures. 
%It does not genuinely increase the size of the moduli, but will be used in the equivariant construction to record the lifting information of the holomorphic curves in the quotient.

\begin{lemma}\label{lem:equatorsh}
On each domain for pseudo-holomorphic maps for $\mu_k$, one can choose 	a path from each input puncture to the output puncture in a manner that is consistent with the gluing of domains.
\end{lemma}
%
%The alignment of punctures we choose in the proof already appeared in \cite[4.6]{RS}, and is by no means unique. 

\begin{proof}
Although the choice is by no means unique, the most convenient one is to use a flavor that is a chosen geodesic from each input to output which already sits a part of moduli data. Below, we briefly recall how to obtain these paths. Let us denote by $z_0$ the output puncture, and take an input puncture $z_j$ in the domain $\Sigma$ of a pseudo-holomorphic map for $\mu_k$. A flavor for this pair is a choice of biholomorphic map $\psi_j : \Sigma \to \mathbb{P}^1$ such that $\psi_j(z_0) = \infty$ and $\psi_j(z_j) = 0$. The preimage of the positive real line in $\mathbb{P}^1$ under $\psi_j$ is now a path from $z_j$ to $z_0$ in $\Sigma$. 

Note that the choice of a direction $\theta_0 \in \mathbb{R}\mathbb{P} (T_{z_0} \Sigma)$ uniquely determines this path, as it singles out $\psi_j: \Sigma \to \mathbb{P}^1$ (uniquely up to a positive scaling on $\mathbb{P}^1$) by requiring $\psi_\ast (\theta_0)$ to point the positive real direction. Once $\theta_0$ is fixed, one can also coherently choose $\theta_j \in \mathbb{R}\mathbb{P} (T_{z_j} \Sigma)$ so that $(\psi_j)_\ast (\theta_j)$ is real positive. $\theta_j$'s are usually called the asymptotic markers. 

In our case, we may always take $\Sigma$ to be $\mathbb{P}^1$ with punctures aligned along the equator, and fix asymptotic markers  towards the positive real direction of the equator. (This agrees with the choice made in  \cite[4.6]{RS}.)\footnote{If we allow general positions for punctures $z_i$ and set asymptotic markers vary around in $S^1$-family, then the count of Floer solutions in the corresponding moduli will produce an $L_\infty$-structure. See \cite{BAS}}
Finally, we remark that one needs to slightly modify these paths near the punctures by choosing `sticks' so that they are consistent with gluings. See \cite[Definition 2.2]{BAS} (which is originally due to Abouzaid-Seidel \cite[2.4]{AS_w} in the case of wrapped Floer theory) for more details. 
%
%
%
%======
%
%In the cylindrical presentation of the domain for $\mu_k$ operations, we first fix a line $L:t=\pi$ and set asymptotic markers for $z_0$ and $z_1$ to avoid this line. We then connect the two asymptotic markers by a line that is not touching $L$. For $\mu_1$, .the two ends (punctures) are located at $s=\pm\infty$, and the asymptotic markers(cylindrical-end parametrizations) can be chosen so that one may choose the $t=0$ line for connecting the ends. 
%%This becomes a genuine circle (the equator $\mathbb{R} P^1$) if we replace the cylinder model by $\mathbb{C}P^1$ with two punctures at $z_0=0$ and $z_1=\infty$ corresponding two negative and positive ends respectively.
%When formulating moduli for $\mu_k$ with $k\geq2$, we view the domain as $\C P^1$ with punctures $z_0=0$, $z_1=\infty$, and $z_2,\ldots,z_k$ located at the line $(0,\infty)\subset \R P^1$. Meanwhile, it is always possible to choose cylindrical-end parametrizations for the punctures so that asymptotic markers are all pointing upward, that is, the direction which is increasing $t$ in the cylinder set-up. We then choose path from asymptotic marker for each input puncture to the asymptotic marker to the output that is not touching $L$, so that contained in $[0,\pi]\times\R$ while not touching $t=0,\pi$ lines. This data works well with the breaking of cylinders when the bubbling occurs, since we have allowed compact homotopy that makes the lines reside in $[0,\pi]\times\R$ in the gluing process.
\end{proof}

%A direct consequence of the lemma is that 
Thus, the domain of each pseudo-holomorphic maps $u$ has a canonical path from each input (positive puncture) to the output (negative puncture). 
%More precisely, the path follows the $t=0$ line of the cylindrical (positive) end corresponding to $i$-th marked point, runs along the equator of the fixed domain $\mathbb{C}P^1$, and merges into the $t=0$ line in the cylindrical (negative) end for the $0$-th marked point. We make the path avoid positive punctures $z_{i+1}, z_{i+2}, \cdots, $ in the middle by letting it go slightly `above' these punctures (this makes sense as we fixed the domain $\mathbb{C}P^1$) until it reaches $z_0$.
Its image under $u$ is a path in $X$ that joins starting points of the Hamiltonian orbit incident to the $i$-th puncture and that of the output orbit.
We denote this path by $\gamma_{u,i}$. A similar argument (using flavours) can be used to consistently choose paths from inputs to the output for other types of modulis such as the moduli of disks for Lagrangian Floer complex (see \ref{subsec:hatgequivlft}) or that of punctured disks for closed-open maps (see \ref{subsec:constksmap}). We will not repeat the argument for other modulis.

In the equivariant setting (with the existence of the group action) below, these data will be used to twist the algebraic structures in the symplectic cohomology. Note that the homotopy type of the paths $\gamma_{u,i}$ is preserved under the degeneration of pseudo-holomorphic maps, since the choice of paths in Lemma \ref{lem:equatorsh} is consistent with gluings. Hence the twist still gives rise to a well-defined algebraic structure.

\subsection{Symplectic cohomology of pair-of-pants}\label{subsec:SHP}
We give an explicit description of $SH^*(P)$ when $P$ is a tri-puntured sphere, as known as the pair-of-pants. It can be neatly represented in terms of $H^*_{log} (M,D)$ due to the work of \cite{GP20}, where we think $M=\mathbb{C}P^1$ as a compactification of $P$ with $D= M \setminus P$ consisting of three distinct points, say $q_\alpha$, $q_\beta$ and $q_\gamma$. The description is particularly simple in this case since $D$ is a set of points.

Following \cite[3.2]{GP20}, $H^*_{log} (M,D)$ is generated by $H^* (P)$ and the cohomology of three circles appearing as the boundary of the real blowup of $M$ along $D$. In our case, these circles can also be identified with boundary circles of small disk neighborhoods of $q_\alpha$, $q_\beta$ and $q_\gamma$. We denote these circles by $S_\alpha$, $S_\alpha$, $S_\alpha$ respectively, and choose (perfect) Morse functions $h_\alpha,h_\beta$ and $h_\gamma$ on them. Then the chain-level $H^*_{log} (M,D)$ is defined by
\begin{equation}\label{eqn:logcohommorse}
 CM^*(h;P) \oplus t_\alpha CM (h_\alpha;S_\alpha)[t_\alpha] \oplus t_\beta CM (h_\beta;S_\beta)[t_\beta] \oplus t_\gamma CM^* (h_\gamma;S_\gamma)[t_\gamma]
\end{equation}
where $h=H$ (regraded as a generic Morse function on $P$), and $t_\alpha,t_\beta,t_\gamma$ are formal variables of degree $1$. Since we are using perfect Morse functions, the differential on \eqref{eqn:logcohommorse} automatically vanishes. Moreover, $CM (h_\alpha;S_\alpha)$ is generated by two critical points $e_\alpha$ and $f_\alpha$ with $|e_\alpha|=0$ and $|f_\alpha|=1$, and similar for $\beta$ and $\gamma$. We set $e,f_1,f_2$ for critical points of $h=H$ as shown in Figure \ref{fig:hforpop}. Recall that the actual Hamiltonian we use to define $SH^*(P)$ is its time-dependent perturbation.

\begin{figure}[h]
\includegraphics[scale=0.55]{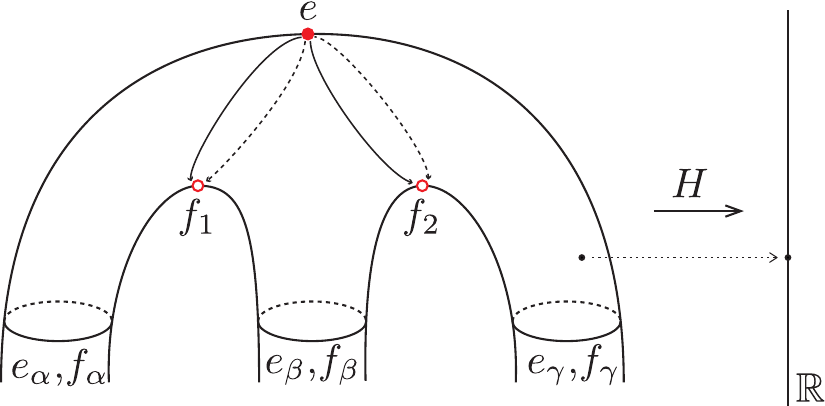}
\caption{Choice of the Hamiltonian $H$ on $P$ and its critical points in $P^{in}$}
\label{fig:hforpop}
\end{figure}

The product structure is given in the form of
$$ t^v x \cdot t^w y = t^{v+w} x \star y$$
where $x \star y$ between two critical points of one of $h_i$'s counts usual $Y$-shape flow lines composed with the restriction map $CM^* (h;P) \to CM (h_\alpha;S_\alpha)$ (and similarly for $\beta$ and $\gamma$) {\it cf}. Figure 4 in \cite{GP20}. In our case, one can deduce all the products from 
\begin{equation}\label{eqn:logcohommorse1}
\begin{array}{c}
 f_1 \star e_\alpha= f_\alpha, \quad  e_\alpha \star e_\alpha =e_\alpha, \quad e_\alpha \star f_\alpha = f_\alpha \\
 f_1 \star e_\beta= - f_2 \star e_\beta = f_\beta, \quad  e_\beta \star e_\beta =e_\beta, \quad e_\beta \star f_\beta = f_\beta \\
 f_2 \star e_\gamma= f_\gamma, \quad  e_\gamma \star e_\gamma =e_\gamma, \quad e_\gamma \star f_\gamma = f_\gamma 
\end{array}
\end{equation}
together with the fact that $e$ is the unit. (Note that $e_\alpha$ and $f_\alpha$ themselves do not belong to the cochain complex.) 
%{\color{red} better choose $\mathbb{Z}/3$-symmetric $f=H$?}

By \cite[Theorem 1.4, Example 1.1]{GP20}, the (low-energy) PSS map gives a ring isomorphism, which do not involve any complicated counting in our case. Under the PSS map, one can identify $e_\alpha t_\alpha^i$ and $f_\alpha t_\alpha^i$ with the two nondegenerate orbits near $q_\alpha$ with winding number exactly $i$, and similar for $\beta$ and $\gamma$. From now on, we will not strictly distinguish  an element of \eqref{eqn:logcohommorse} and its corresponding Hamiltonian orbit in $SC^*(P)$ unless there is a danger of confusion.

%Note that each component of the cylindrical end is symplectomorphic to the expanding cylinder $C_A=((0,\infty)\times S^1, Ae^rdr\wedge d\theta)$ by a reparametrization of $S^1$.
%Here $A>0$ is a constant assigned on the chosen component.
%Since such a symplectomorphism does not involve the radial coordinate, the quadratic Hamiltonian pulls back to the quadratic Hamiltonian again.
%Thus it suffices to confine our attention to the question that how the 1-periodic orbits of the quadratic function $H(r,\theta)=r^2$ in $C_A$ looks like.

$SC^*(P)$ is graded once we choose a volume form on $P$ and trivialize $TP(=T^\mathbb{C} P)$ accordingly. In our case, we will find it more convenient to use the section of $\Omega_{P}^{\otimes 3}$ to keep the symmetry. Namely, we choose a cubic volume form $\Omega$ (locally $\Omega= \varphi (z) dz^3$) on $S^2$ which has a double pole at each of points in $D= \{q_\alpha, q_\beta, q_\gamma\}$. For e.g., one may take $\frac{dz}{z^2(z-1)^2}$ for $q_\alpha = 0, q_\beta=1, q_\gamma=\infty$. 
It trivializes $TP^{\otimes 3}$, and for a Hamiltonican orbit $\gamma$, the linearized Hamiltonian flow $ TP^{\otimes 3} \to TP^{\otimes 3}$ gives a path of symplectic matrices, and hence the associated Conley-Zehnder index $CZ_{TP^{\otimes 3}} (\gamma)$. Our grading is then $|\gamma| := \frac{1}{3}CZ_{TP^{\otimes 3}} (\gamma)$. 

This type of grading appeared in, for e.g., McLean \cite[4.1]{Mc1} for Reeb orbits, and more detailed explanation (for Hamiltonian orbits) can be found in \cite[3.1]{DL}, where it is referred to as the Seidel-Mclean grading.
The upshot is the fractional grading on $SC^* (P)$ for which 
$$|t_\alpha| =|t_\beta|=|t_\gamma|= \frac{2}{3}, \quad |e_\alpha|=|e_\beta|=|e_\gamma|= 0, \quad |f_\alpha|=|f_\beta|=|f_\gamma|=1.$$
The component $CM^* (h;P)$ carries the usual $\mathbb{Z}$-grading, i.e., $|e| =0, |f_1|=|f_2|=1$. (The $A_\infty$-operation $\mu_k$ is of degree $2-k$ with respect to this grading.)  We will use an analogous fractional grading for Lagrangian Floer theory on $P$. Multiplying $3$ to this fraction degrees recovers the usual $\mathbb{Z}/2$-grading on $SC^* (P)$.

\subsection{Orbifold LG B-models}\label{subsec:orbLGB}
Our mirror LG-model $W:\Bbbk^n \to \Bbbk$ will carry an action of a finite abelian group $H$.
We introduce the construction of the closed-string $B$-model invariant attached to the associated orbifold setting.

Let $\theta_i$ and $\partial_{\theta_i}$ (for $i=1,\cdots,n$) be variables with $|\theta_i|=-1$, $|\partial_{\theta_i}|=1$ and
\begin{equation}\label{eq:cliffrel}
 \theta_i \theta_j=-\theta_j \theta_i,\;\; \partial_{\theta_i}\partial_{\theta_j}=-\partial_{\theta_j}\partial_{\theta_i},\;\; \partial_{\theta_i}\theta_j=-\theta_j\partial_{\theta_i}+\delta_{ij}
 \end{equation}
We recall an elementary definition in terms of above variables.
\begin{defn}
Let $R$ be a commutative ring. For a sequence $(r_1,\cdots,r_n) \subset R$, 
\[ K^*(r_1,\cdots,r_n):=(R\langle \theta_1,\cdots,\theta_n \rangle, \sum_i r_i \partial_{\theta_i})\]
is called the {\em Koszul complex} of $(r_1,\cdots,r_n)$. When $(r_1,\cdots,r_n)$ is given as the Jacobian ideal of some $W$, then we will write the corresponding Koszul complex as $(K^* (\partial W), d_W)$.
\end{defn}
The cochain complex $K^*(r_1,\cdots,r_n)$ is explicitly
\[\displaystyle \xymatrix{
 0 \ar[r] & R\cdot \theta_1\cdots\theta_n \ar[r]^-{\sum_i r_i\partial_{\theta_i}} 
 & \bigoplus\limits_{i_1<\cdots<i_{n-1}} R \cdot \theta_{i_1}\cdots\theta_{i_{n-1}} \ar[r] & \cdots \ar[r] & \bigoplus\limits_i R\cdot \theta_i \ar[r]^-{\sum_i r_i\partial_{\theta_i}} & R \ar[r] & 0.} \]
Let us denote $K^{-m}(r_1,\cdots,r_n)=\bigoplus\limits_{i_1<\cdots<i_m}R\cdot \theta_{i_1}\cdots\theta_{i_m}$.

The following notion is an important categorical invariant of singularities.
\begin{defn}
Let $R$ be a commutative ring. A {\em matrix factorization} of $W \in R$ is a $\Z/2$-graded projective $R$-module $P=P_0\oplus P_1$ together with a morphism $d=(d_0,d_1)$ of degree 1 such that
\[ d^2=W\cdot \id.\] 
Let $\phi: (P,d) \to (Q,d')$ be a morphism of degree $j\in \Z/2$.  Define
\[ D\phi:=d'\circ \phi -(-1)^{|\phi|}\phi \circ d\]
and define the composition of  morphisms in the usual sense. This defines a dg-category of matrix factorizations $(MF(W),D,\circ)$.
\end{defn}

\begin{defn}
Let $W\in R$ and let $H$ be a group which acts on $R$ leaving $W$ invariant. An {\em $H$-equivariant matrix factorization} of $W$ is a matrix factorization $(P,d)$ where $P$ is equipped with an $H$-action, and $d$ is $H$-equivariant. An {\em $H$-equivariant morphism} $\phi$ between two $H$-equivariant matrix factorizations $(P,d)$ and $(Q,d')$ is an $H$-equivariant morphism of $\Z/2$-graded $R$-modules $P$ and $Q$. Again, $H$-equivariant matrix factorizations form a dg-category $(MF_H(W),D,\circ)$.
\end{defn}

From now on we will consider $W\in R=\Bbbk[x_1,\cdots,x_n]$ and $H$ is a finite abelian group acting on $R$ leaving $W$ invariant. We call the pair $(W,H)$ an {\em orbifold LG B-model.} Furthermore we will assume that $H$ acts on $R$ diagonally, which means that for $h\in H$, 
\begin{equation}
h\cdot x_i = h_i x_i
\end{equation}
for some $h_i \in \Bbbk^*$. We introduce the following notation:
\[ x_i^h:=\begin{cases}
x_i & {\textrm{ if }} h\cdot x_i=x_i,\\
0 & {\textrm{ if }} h\cdot x_i \neq x_i.
\end{cases}
\]
%\begin{defn}
Let $S=\Bbbk[x_1,\cdots,x_n,x_1',\cdots,x_n']$. Given an orbifold LG model $(W,H)$.
  \[\nabla_i W:= \frac{W(x_1',\cdots,x_i',x_{i+1},\cdots,x_n)-W(x_1',\cdots,x_{i-1}',x_i,\cdots,x_n)}{x_i'-x_i}\in S\] 
and define an $(H\times H)$-equivariant matrix factorization of $W\boxminus W:=W(x_1',\cdots,x_n')-W(x_1,\cdots,x_n)$ by
\[ \Delta_W^{H \times H}:=\bigoplus_{h\in H}\Big( S\langle\theta_1,\cdots,\theta_n\rangle,\sum_{i=1}^n \big((x_i'-h\cdot x_i)\theta_i+\nabla_i W \cdot\partial_{\theta_i}|_{h\cdot x \to x}\big)\Big),\]
with $H$-action on variables $\{\theta_1,\cdots,\theta_n,\partial_{\theta_1},\cdots,\partial_{\theta_n}\}$ defined by
 \[ h\cdot\theta_i:=h_i^{-1} \theta_i, \;\; h\cdot\partial_{\theta_i}:=h_i \partial_{\theta_i}\]
when $h\cdot x_i=h_i x_i$. Also, for any $S$ (or $R$)-linear map $F$, $F|_{h\cdot x \to x}$ means that we replace all variables $\{x_1,\cdots,x_n\}$ in $F$ by $\{h\cdot x_1,\cdots,h\cdot x_n\}$ accordingly.
% The {\em orbifold Koszul algebra} of $(W,H)$ is defined by
% \begin{equation}\label{def:orbKos}
%     Kos(W,H):= \Hom_{MF_{H\times H}(W\boxminus W)}
%     (\Delta_W^{H \times H},\Delta_W^{H \times H}).
% \end{equation}

% \end{defn}
Denote each $h$-summand of $\Delta_W^{H\times H}$ by $\Delta_W^h$. We recall the following.
% , namely
% \begin{equation}\label{eq:deltah}
% \Delta_W^h := \Big( S\langle\theta_1,\cdots,\theta_n\rangle,\sum_{i=1}^n \big((x_i'-hx_i)\theta_i+\nabla_i W|_{h\cdot x \to x} \cdot\partial_{i}\big)\Big).
% \end{equation}

\begin{lemma}[\cite{CLe}]\label{lem:equivmfHH}
We have a quasi-isomorphism
\[ hom_{MF_{H \times H}(W\boxminus W)}(\Delta_W^{H\times H}, \Delta_W^{H\times H})\cong 
\big(\bigoplus_{h\in H} hom_{MF(W\boxminus W)}(\Delta_W^1,\Delta_W^h)\big)^H.\]
\end{lemma}

Let us define following (ordered) subsets of $\{1,\cdots,n\}$:
\[ I_h:= \{i\mid hx_i \neq x_i\},\quad  I^h:=I_h^c.\]
Recall that if $W$ has only isolated singularity, then $\Hom_{MF_{H \times H}(W\boxminus W)}(\Delta_W^{H\times H}, \Delta_W^{H\times H})$ is isomorphic to the orbifold Jacobian algebra $\Jac(W,H)$ which is given by following:
\[ \Jac(W,H)= \big( \bigoplus_{h\in H} \Jac(W^h) \cdot \xi_h \big)^H\]
%where $W^h$ is a coset of $W$ in $R/(x_i: i\in I_h)$ and $\xi_h$ is a formal generator of degree $d_h:=|I_h|$.
where 
\begin{itemize}
\item $W^h=W(x_1^h,\cdots,x_n^h)\in k[x_1^h,\cdots,x_n^h] $, 
\item $\Jac(W^h)=k[x_1^h,\cdots,x_n^h]/\partial W^h$ where $\partial W^h:=(\partial_{x_i}W^h:i\in I^h)$ is the Jacobian ideal of $W^h$,
\item $\xi_h$ is a formal generator of degree $d_h:=|I_h|$. 
\end{itemize}
For general $W$ which may have nonisolated singular locus, we still have the following result.

\begin{prop}\label{prop:mfkos}
We have a module isomorphism
\begin{equation}\label{eq:mfkos}
    \Hom_{MF_{H \times H}(W\boxminus W)}(\Delta_W^{H\times H}, \Delta_W^{H\times H}) \cong \big( \bigoplus_{h\in H} H^*\big(K^*(\partial W^h )\big) \cdot \theta_{I_h}  \big)^H.
\end{equation}

\end{prop}

\begin{proof}
The matrix factorization $\Delta_W^h$ is constructed by the resolution of a shifted MCM module $S/(x_1'-h\cdot x_1,\cdots,x_n'-h\cdot x_n)[-n]$ over the hypersurface ring $S/(W\boxminus W)$. Hence we have
\begin{align}
hom_{MF(W\boxminus W)}(\Delta_W^1,\Delta_W^h) & \simeq hom_S\big(\Delta_W^1,S/(x_1'-h\cdot x_1,\cdots,x_n'-h\cdot x_n)[-n] \big)\nonumber\\
&\simeq hom_S\big(\Delta_W^1[n],S/(x_1'-h\cdot x_1,\cdots,x_n'-h\cdot x_n) \big)\nonumber\\
 &\simeq (\Delta_W^1)^\vee[-n] \otimes_S \big(S/(x_1'-h\cdot x_1,\cdots,x_n'-h\cdot x_n)\big). \label{homcomplex}
 \end{align}
By self-duality of Koszul matrix factorizations, $(\Delta_W^1)^\vee[-n]$ is quasi-isomorphic to $\Delta_W^1$. Hence we have
\[ \eqref{homcomplex} \simeq \big( R\langle\theta_1,\cdots,\theta_n\rangle,d|_{h\cdot x \to x'} \big).\]
The latter is the following $\Z/2$-graded double complex: 

\[\xymatrixcolsep{0.8pc}\xymatrix{
& && & & \vdots& & \vdots & \vdots & \\
& &&& 0 \ar[r] & R\ar[u]_{d_{vert}} \ar[r]^{d_{hor}} & \cdots \ar[r]^-{d_{hor}} & \displaystyle \bigoplus_{1\leq i_1<\cdots<i_{n-1}\leq n}R\cdot \theta_{\{i_1,\cdots,i_{n-1}\}}\ar[u]_{d_{vert}} \ar[r]^-{d_{hor}} & R\cdot \theta_{\{1,\cdots,n\}} \ar[u]_{d_{vert}}\ar[r]  & 0 \\
&& 0\ar[rr] && R \ar[r]^-{d_{hor}}\ar[u] & \displaystyle\bigoplus_{i=1}^n R\cdot \theta_i \ar[u]_{d_{vert}}\ar[r]^-{d_{hor}}&\cdots \ar[r]^-{d_{hor}} & R\cdot \theta_{\{1,\cdots, n\}} \ar[u]_-{d_{vert}}\ar[r] & 0 \ar[u]& \\
0\ar[rr] && R \ar[u]\ar[rr]^-{d_{hor}} && \displaystyle\bigoplus_{i=1}^n R\cdot \theta_i \ar[r]^-{d_{hor}} \ar[u]_-{d_{vert}} & \cdots \ar[r]^-{d_{hor}} & R\cdot \theta_{\{1,\cdots,n\}}\ar[r] &0\ar[u]& &\\
&& \vdots \ar[u]_-{d_{vert}} && \vdots \ar[u]_-{d_{vert}} & & \vdots\ar[u]_-{d_{vert}} & & &
}
\]
with
\[d_{hor}:=\sum_{i=1}^n (h\cdot x_i-x_i)\theta_i,\quad d_{vert}:=\sum_{i=1}^n \nabla_i W\cdot\partial_{\theta_i} |_{h\cdot x \to x'}.\]
% As above, $\nabla_i W|_{h\cdot x \to x'}$ means that we replace variables $\{x_1',\cdots,x_n'\}$ in $\nabla_i W$ by $\{h\cdot x_1,\cdots,h\cdot x_n\}$.
We compute its cohomology by spectral sequence from vertical filtration. The first page is the cohomology with respect to $d_{hor}$. The $i$th row is 
\[\xymatrix{
0 \ar[r] & R\ar[r]^-{d_{hor}} &\bigoplus_{i=1}^n R\cdot \theta_i \ar[r]^-{d_{hor}} & \cdots \ar[r]^-{d_{hor}} &R\cdot \theta_{\{1,\cdots,n\}} \ar[r] & 0,}
\] 
and it is quasi-isomorphic to the following complex:
\[\xymatrix{
0 \ar[r] &  R^h\cdot \theta_{I_h} \ar[r]^-0 & \displaystyle  \bigoplus_{i \in I^h} R^h\cdot\theta_i \theta_{I_h}\ar[r]^-0
&\displaystyle \bigoplus_{\stackrel{i_1<i_2}{ i_1,i_2\in I^h}}R^h \cdot \theta_{i_1}\theta_{i_2}\theta_{I_h}}\]
\[\xymatrixcolsep{1pc}\xymatrix{
\; \ar[r]^-0 & \cdots \ar[r]^-0 &\displaystyle \bigoplus_{\stackrel{i_1<i_2<\cdots<i_{|I^h|-1}}{i_1,\cdots,i_{|I^h|-1}\in I^h}} R^h\cdot \theta_{i_1}\cdots \theta_{i_{|I^h|-1}}\theta_{I_h} \ar[r]^-0 & R^h \theta_{\{1,\cdots,n\}} \ar[r] & 0
}\]
where $R^h:=R/(x_i: i\in I_h)$. Therefore, the induced differential $d_1$ on $E_1$ is given by
\[ \sum_{i\in I^h} \nabla_i W(x,h\cdot x)\partial_{\theta_i} = \sum_{i\in I^h} (\partial_i W^h) \partial_{\theta_i},\]
hence each column is isomorphic to the (degree-)shifted Koszul complex of the Jacobian ideal $\partial W^h$. Here the column is referred to be shifted in the sense that $\theta_{I_h}$ is multiplied to the original complex.

It remains to show that $d_i=0$ on $E_i$ for all $i\geq 2$. Let $\displaystyle \sum_{J: I_h \subset J} f_J \theta_J$ be a cocycle of $(E_1,d_1)$. Then we have
\[ \sum_{\stackrel{J: I_h \subset J}{i\in J}}\nabla_i W(x,h\cdot x) f_J \frac{\partial \theta_J}{\partial \theta_i} = \sum_{j\in I_h}(h\cdot x_j-x_j)\theta_j \big( \sum_K g_K \theta_K \big)\]
for some $\sum_K g_K\theta_K$, so 
\[d_2 (\sum f_J \theta_J)=\sum_i\sum_{K} \nabla_i W(x,h\cdot x)g_K \frac{\partial\theta_K}{\partial\theta_i}.\] 
But each summand $g_K \theta_K$ is given by $K= J'-{j}$ for some $J'$ and $j \in I_h$, so 
$d_2=0$ on $E_2$. It implies that $d_{\geq 2}=0$.
%Fix a summand $\nabla_i W(x,hx) f_J \frac{\partial \theta_J}{\partial \theta_i}$ of LHS.
\end{proof}
% Since $\Hom_{MF_{H \times H}(W\boxminus W)}(\Delta_W^{H\times H}, \Delta_W^{H\times H})$ is an algebra, we can equip $\big( \bigoplus_{h\in H} H^\ast\big(K^\ast(\partial W^h )\big) \cdot \theta_{I_h}  \big)^H$ with the product structure. 

\begin{defn}\label{def:orblgB}
Given an orbifold LG model $(W,H)$,
\[ Kos(W,H):=\big( \bigoplus_{h\in H} H^*\big(K^*(\partial W^h )\big) \cdot \theta_{I_h}  \big)^H\]
is called its {\em orbifold Koszul algebra} whose product structure is given by the isomorphism \eqref{eq:mfkos}. (When $H=1$, we simply write $Kos (W)$ which agrees with the cohomology of $(K^* (\partial W),d_W)$.)
\end{defn}

\subsection{Finite group action on Lagrangian Floer theory and semidirect product}\label{subsec:BmLF}

 We briefly recall the construction of semidirect product $\AI$-algebras following \cite{CLe}. Its (weak) Maurer-Cartan deformation gives rise to the mirror orbifold LG model to which we will apply the construction in \ref{subsec:orbLGB}. Provided that the Lagrangian Floer complex is well-defined as an $A_\infty$-algebra, the construction below is entirely an algebraic procedure on this $A_\infty$-algebra. We will provide a detailed explanation on the technical aspects of defining the Lagrangian Floer complex in our geometric setup later in Section \ref{sec:KSRS} (see \ref{subsec:constksmap}).

 Consider a symplectic manifold $X$ (convex at infinity if noncompact) on which a finite abelian group $G$ acts. 
 In cases in which we are interested, we always have a (possibly immersed) Lagrangian submanifold $\bL \subset X/G $ with the following properties: there is an embedded Lagrangian lift $\bL_1\subset X$ of $\bL$, and if we let $\bL_g:=g\cdot \bL_1$ for $g\in G$, we have $\bL_{g_1} \neq \bL_{g_2}$ whenever $g_1 \neq g_2$.

After fixing the lift $\bL_1$, we have a natural identification of Floer complexes:
\[ CF(\bL,\bL) \cong \bigoplus_{g\in G}CF(\bL_1,\bL_g).\]
From now on we will freely use this identification for $CF(\bL,\bL)$.

Let $\HG=\Hom(G,U(1))$ be the character group of $G$. Following \cite[Chapter 7]{CHL2}, we have a natural $\HG$-action $\rho$ on $CF(\bL,\bL)$ as follows: for $v_g \in CF(\bL_1,\bL_g)$ and $\chi\in \HG$,
\[ \rho(\chi)(v_g):=\chi(g^{-1})v_g.\]
We recall the notion of semi-direct products.
\begin{defn}[\cite{Seidel-g2}]\label{def:Cho-LeeAinf}
Given an $\AI$-algebra $\mathcal{A}$ and $H$-action $\rho$ on $\mathcal{A}$,  the {\em semi-direct product $\AI$-algebra} structure on $ \mathcal{A}\otimes \Bbbk[H]$ is defined as
% follows. As a module
%\[ \mathcal{A}\rtimes \HG= \mathcal{A}\otimes k[\HG]\] and
\begin{align}
\begin{split}
 &m_k(w_1 \otimes \chi_1, \cdots, w_k \otimes \chi_k)\label{eq:interioraction}\\
 :=&m_k\big(\rho(\chi_2 \cdots \chi_{k})(w_1), \rho(\chi_3 \cdots \chi_{k})(w_2), \cdots, \rho(\chi_k)(w_{k-1}), w_k\big) \otimes \chi_1 \cdots \chi_k.
 \end{split}
\end{align}
\end{defn}

\begin{remark}
Observe that for $v_{g_1}\in CF(\bL_1,\bL_{g_1}),\cdots,v_{g_k}\in CF(\bL_1,\bL_{g_k})$, we can rewrite $m_k$ on $CF(\bL,\bL)\otimes \Bbbk[\HG]$ by
\begin{align*}
\begin{split}
 &m_k(v_{g_1} \otimes \chi_1, \cdots, v_{g_k} \otimes \chi_k)\\
=&m_k\big(v_{g_1},\chi_2(g_1^{-1})v_{g_2},\cdots,\chi_k(g_1^{-1}\cdots g_{k-1}^{-1})v_{g_k} \big) \otimes \chi_1 \cdots \chi_k.
 \end{split}
\end{align*}
This formulation is useful for the interpretation of semi-direct product $\AI$-structure in terms of ``path-decorated" discs which will appear later.
\end{remark}
%\[ \eqref{eq:interioraction}= m_k\big(w_1,\rho(\chi_1\chi_2)^{-1}(w_2),\cdots,\rho(\chi_1\cdots\chi_k)^{-1}(w_k)\big)\otimes \chi_1\cdots\chi_k\]
%by equivariance of $m_k$ with respect to $\HG$-action $\rho$.

%Let
%\[ \WT{\bL}:= \bigoplus_{g\in G}\bL_g.\]
%On $X$, we have an $\AI$-algebra $ CF(\WT{\bL},\WT{\bL})$.

For notational simplicity, we write by
\begin{equation}\label{eqn:hatGcoeffi}
 CF(\bL,\bL;\hat{G}):= CF(\bL,\bL) \otimes \Bbbk[\HG] 
\end{equation}
the corresponding module with the $A_\infty$-structure in Definition \ref{def:Cho-LeeAinf} from now on.

\begin{lemma}[\cite{Seidel-g2,CLe}]\label{lem:downtoup}
For $\WT{\bL}:= \bigoplus_{g\in G}\bL_g \subset X$, we have an $\AI$-isomorphism
\[ \Phi:  CF(\bL,\bL;\hat{G}) \to CF(\WT{\bL},\WT{\bL})\]
where $\Phi_1(v \otimes \chi) = \sum_{g\in G} \chi(g^{-1})(g\cdot v)$ and $\Phi_{\geq 2}=0$.
\end{lemma}

Next, we consider the deformation of $\bL$ by weak bounding cochains. Recall that $b \in CF^1(\bL,\bL)$ is called a (weak) bounding cochain if it satisfies
\[ m_0+m_1(b)+m_2(b,b)+\cdots = W(b)\cdot e_\bL.\]
for $W(b)\in \Bbbk$ where $e_\bL$ is the unit (which can be specified in the Morse model we will use in the later applications). $W$ is called the (Floer) potential. We denote the set of all bounding cochains by $MC_{weak}(\bL)$. 
Throughout, we work with an assumption that there exist elements $X_1,\cdots,X_n \in CF^1(\bL,\bL)$ such that any $\Bbbk$-linear combination $b$ of $X_1,\cdots,X_n$ satisfies weak Maurer-Cartan equation. Hence $MC_{weak}(\bL)$ contains their linear span isomorphic to $\Bbbk^n$. 
Define
\begin{equation}\label{eqn:cflblb}
 CF ((\bL,b),(\bL,b)):= CF(\bL,\bL)\otimes_\Bbbk R,
\end{equation}
i.e., we extend the coordinate to obtain an $\AI$-algebra over $R=\Bbbk [x_1,\cdots, x_n]$. This comes with some technical subtlety which will be addressed shortly. One can check that $b=x_1X_1+\cdots+x_nX_n$ is a weak bounding cochain on $ CF ((\bL,b),(\bL,b))$ where $A_\infty$-operations $m_k^{b,\cdots,b}$ on $ CF ((\bL,b),(\bL,b))$ are $b$-deformed $m_k$ on $CF (\bL,\bL)$ with coefficients of $b$ taken from $R$. 
%We will use the notation $CF ((\bL,b),(\bL,b))$ mostly to denote $CF_R(\bL,\bL)$ in what follows. 

  In full generality, $R$ should be the quotient of $\Lambda \{ x_1,\cdots, x_n\}$ by $f_j$'s appearing as coefficients of nonunit generators in $CF(\bL,\bL)$ in the (multi-)linear expansion of 
$$m_0+m_1(b)+m_2(b,b)+\cdots = \tilde{W}(b) \cdot e_\bL + \sum_j f_j (x_1,\cdots, x_n) Y_j$$
using the rule $m_k (x_{i_1} X_{i_1},\cdots, x_{i_k} X_{i_k} ) =x_{i_k} \cdots x_{i_1}$. Here $\Lambda \{ x_1,\cdots, x_n\}$ consists of noncommutative power series in $x_i$'s with valuations of coefficients bounded below. Notice that variables a priori do not commute. However, for our purpose, it suffices to work with a polynomial ring over $\mathbb{C}$ since (i) there is no convergence issues due to some strong finiteness conditions (see Lemma \ref{lem:seidelLagoverc}) and (ii) $f_j$'s are given as commutators. Therefore any $\mathbb{C}$-linear combination of $X_i$'s serves as a weak bounding cochian in our main application.

\begin{remark}\label{rmk:lblbhoch}
There is an obvious structural similarity between $CF_R(\bL,\bL)$ and (a certain component of) the Hochschild cochain complex $CH^* (CF(\bL,\bL ),CF(\bL,\bL ))$. In fact, if one identifies an element $x_{i_k} \cdots x_{i_1} Y \in CF_R(\bL,\bL)$ with a map $CF(\bL,\bL)^{\otimes k} \to CF(\bL,\bL)$ that sends the generator $X_{i_1} \otimes \cdots \otimes X_{i_k}$ to $Y$ and all the others to zero, then differentials and products on both sides match (up to a coherent change of signs). The only difference is that $x_i$'s do not commute on Hochschild side. 
\end{remark}

Let us now consider the $\HG$-equivariant twist of the above construction.

\begin{lemma}
Let $b=x_1X_1+\cdots+x_nX_n$ be a weak bounding cochain on $CF ((\bL,b),(\bL,b))$. 
Then $b\otimes 1$ is a weak bounding cochain on $CF ((\bL,b),(\bL,b)) \otimes \Bbbk[\HG]$.
\end{lemma}
Similarly to \eqref{eqn:hatGcoeffi}, we denote 
$$CF ((\bL,b),(\bL,b);\hat{G}):=CF ((\bL,b),(\bL,b)) \otimes \Bbbk[\HG]$$
in what follows.
It is helpful to spell out $m_k^{b\otimes 1,\ldots,b\otimes1}$ operations explicitly. 
\begin{align}
& m_k^{b\otimes 1,\ldots,b\otimes1}(a_1\otimes \chi_1,\cdots,a_k\otimes \chi_k)\label{twistedainfty}\\
=&m_k^{\rho(\chi_1\cdots\chi_k)b,\rho(\chi_2\cdots\chi_k)b,\cdots,\rho(\chi_k)b,b}\big(\rho(\chi_2\cdots\chi_k)a_1,\rho(\chi_3\cdots\chi_k)a_2,\cdots,a_k\big)\otimes \chi_1\cdots\chi_k,\nonumber
\end{align}
where $\rho(\chi)b=\sum_i x_i \cdot \rho(\chi)(X_i).$ 
%By definition of boundary deformation 

Recall that we employed $\HG$-action $\rho$ on $CF(\bL,\bL)$ to define $CF ((\bL,b),(\bL,b);\hat{G})$. Now, on $CF ((\bL,b),(\bL,b);\hat{G})$, we consider an extended $\HG$-action as follows:
\begin{equation}\label{eq:extendedaction} 
\chi\cdot(r(x) v\otimes \eta):= r(\chi\cdot x)\rho(\chi)v\otimes \eta.
\end{equation}

Let $S=\Bbbk[x_1,\cdots,x_k,x_1',\cdots,x_k']$, $b'=x_1' X_1+\cdots + x_n' X_n$, and consider $CF((\bL,b),(\bL,b'))$ which is defined analogously to \eqref{eqn:cflblb}. As a module
\begin{equation}\label{eq:diagonalMF}
CF((\bL,b),(\bL,b'))=CF(\bL,\bL)\otimes_{\Bbbk} S
\end{equation} 
and it has a curved differential $m_1^{b,b'}$ deformed by $b$ and $b'$ whose coefficients are taken from $S$. 
It is a (curved) $\AI$-bimodule, but we may regard it as a matrix factorization of $W \boxminus W:=W(x_1',\cdots,x_n')-W(x_1,\cdots,x_n)$. 

\begin{prop}\label{prop:floerkos}
We have the followings.
\begin{itemize}
\item $m_k^{b\otimes 1,\ldots,b\otimes1}$ is equivariant with respect to the $\HG$-action in \eqref{eq:extendedaction} so $m_k^{b\otimes 1,\ldots,b\otimes1}$ defines a well-defined $\AI$-structure on $(CF ((\bL,b),(\bL,b);\hat{G}))^{\HG}$.

\item If $CF((\bL,b),(\bL,b'))$ in \eqref{eq:diagonalMF} is isomorphic to $\Delta_W^1$, then $(CF ((\bL,b),(\bL,b);\hat{G}))^{\HG}$ is $\AI$-isomorphic to $hom_{MF_{\HG\times \HG}(W\boxminus W)}(\Delta_W^{\HG\times \HG},\Delta_W^{\HG\times \HG})$, so we also have
\[ H^*((CF ((\bL,b),(\bL,b);\hat{G}))^{\HG}) \cong Kos(W,\HG).\]
\end{itemize}

\end{prop}

The above results are shown in \cite{CLe} when $W$ has an isolated singularity. Their proofs only involve algebraic formalism, and the same proof works for nonisolated cases as well.

\subsection{The Seidel Lagrangian in the pair-of-pants $P$}\label{subsec:popFloer}
%Examples of our main interest are punctured Riemann surfaces whose abelian quotient is isomorphic to the pair-of-pants $P$. 
Let us look at the example of the pair-of-pants $P$ again.
We apply the construction in \ref{subsec:BmLF} to a certain immersed Lagrangian $\bL$. We take $\bL$ to be the immersed circle with three transversal self-intersection points $X,Y,Z$ depicted in Figure \ref{fig:seilag}, which carries a nontrivial spin structure. The nontrivial spin structure (marked as {\color{red}$\circ$} in Figure \ref{fig:seilag}) can be represented by fixing a generic point on $\bL$ so that the sign of a pseudo-holomorphic curve changes each time its boundary passes through the point.
This immersed circle was first introduced in \cite{Seidel-g2}, and is often called the \emph{Seidel Lagrangian} for this reason. Sheridan \cite{Sheridan11} used $\bL$ and its higher dimensional analogues to study homological mirror symmetry of pair-of-pants.

\begin{figure}[h]
\includegraphics[scale=0.5]{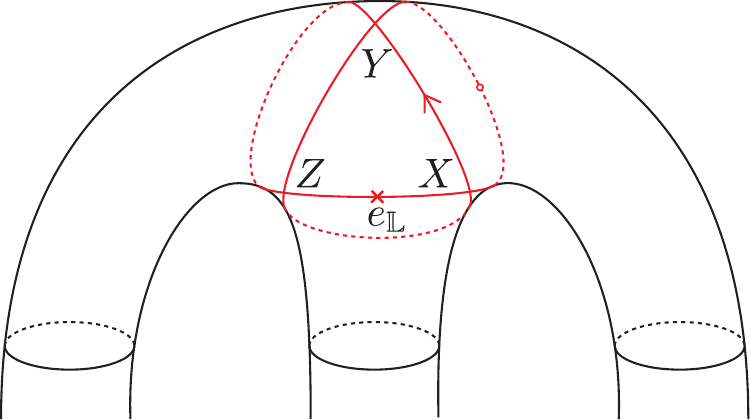}
\caption{The Seidel Lagrangian $\bL$}
\label{fig:seilag}
\end{figure}

The Maurer-Cartan formalism on $CF(\bL,\bL)$ gives rise to a LG model as follows.
Recall that $CF(\bL,\bL)$ takes two generators from the cohomology of $S^1$ which is the domain of the immersion, and two generators from each self-intersection points of complementary degrees.
Denote by $e_\bL$ and $f_\bL$ the unit and the point class in $CF(\bL,\bL)$ from $H^* (S^1)$, and write $X,\bar{X}$, $Y,\bar{Y}$, $Z,\bar{Z}$ for immersed generators supported at the self-intersection points $X$, $Y$, $Z$ respectively with degrees
$|X|=|Y|=|Z| \equiv 1,  |\bar{X}|=|\bar{Y}|=|\bar{Z}| \equiv 0 \mod 2.$
In practice, we prefer a fractional grading induced by the volume form $\Omega$ in parallel to \ref{subsec:SHP}, which results in
$$ |X|=|Y|=|Z|= \frac{1}{3}, \quad |\bar{X}|=|\bar{Y}|=|\bar{Z}|= \frac{2}{3}, \quad |e_\bL| = 0, \quad |f_\bL|=1$$
as already computed in \cite[Section 10]{Seidel-g2}. We take $1/3$ of the `index' therein as our fractional degree in order to keep the degree of $m_k$ to be $2-k$ as usual. (Our $m_k$ corresponds to the $0$-th order term $m_0^k$ in \cite{Seidel-g2}).
 
One can show that $b=xX+yY+zZ$ solves the weak Maurer-Cartan equation
$$ m_1(b) + m_2(b,b) + \cdots =W(b) \cdot e_\bL$$
with $W(b)=W(x,y,z)=xyz$. This is due to the symmetry of $\bL$ with respect to the reflection that swaps the front and back of $P$ (when drawn as in Figure \ref{fig:seilag}). As a result, one has $m_2(X,Y) = - m_2(Y,X)$ and similar for other pairs of degree-1 generators. See \cite[Section 10]{Seidel-g2} or \cite[Theorem 7.5]{CHL} for details.

Technically $x,y,z$ should be taken from the positive valuation part of the coefficient ring to control the convergence. For our purpose, $x,y,z$ can be taken from $\mathbb{C}$ as we will explain in Lemma \ref{lem:seidelLagoverc}. We refer readers to \cite{Seidel-g2} or \cite{CHL} for details on the $A$-infinity structure on $CF(\bL,\bL)$. Indeed, computations here are much simpler than ones in the references which consider the (orbifold-)compactification of $P$.
 
We justify the usage of coefficients $x,y,z$ over $\mathbb{C}$ in the following lemma.

\begin{lemma}\label{lem:seidelLagoverc}
For the Seidel Lagrangian $\bL$ in the pair-of-pants $P$, the Floer complex $CF((\bL,b),(\bL,b))$ (with its $A_\infty$-operations)
can be defined over $\mathbb{C}$. In particular, $b= xX + yY + zZ$ can be taken over $\mathbb{C}$-coefficients accordingly (i.e., $x,y,z$ can be arbitrary complex numbers).
\end{lemma}

\begin{proof}
The only issue is the convergence for $m_1^{b,b}$ and $m_2^{b,b,b}$, as they are given in the form of
\begin{equation}\label{eqn:bdeformedmk}
m_k^{b,\cdots,b} (Z_1,\cdots,Z_k) = \sum m_{k' \geq k} (b,\cdots,b,Z_1,b,\cdots,b, Z_k, b, \cdots, b)
\end{equation}
where we allow to insert arbitrarily many $b$'s. However, in our particular situation, \eqref{eqn:bdeformedmk} is always a finite sum. To see this, one can use aforementioned fractional grading.
When $k=1,2$, \eqref{eqn:bdeformedmk} must be a finite sum, since inserting $b$ each time increases the total degree of input by $1/3$ whereas the degree of the operation drops by $1$ for the additional number of  inputs. 
%
%QUESTION: If we're to consider the full $A_\infty$-operation on $CF(\bL,\bL)$, then the argument above does not cover some cases, say putting lots of $p$'s. Can we use ``divisor axiom" to handle this?
\end{proof}

Formally, one can think of $x,y,z$ as the dual of $X,Y,Z$ where the latter is viewed as the element of the bar construction $B (CF(\bL,\bL)) =  \oplus_k CF(\bL,\bL)^{\otimes k}[1]$. Thus their gradings are naturally given as
$$|x| = - (|X|-1) = \frac{2}{3}, \quad |y| = \frac{2}{3}, \quad |z| = \frac{2}{3}$$ 
from the fractional grading on $CF(\bL,\bL)$. In particular, $W$ has degree $2$.

%\begin{remark}\label{rem:ainftydga}
%To fit into dga sign convention, one needs to modify $m_1^{b,b}$ and $m_2^{b,b,b}$ as
%$$d(Z)= m_1^{b,b}(Z),\quad  Z_1 \cdot Z_2 = (-1)^{|Z_1|} m_2^{b,b,b} (Z_1,Z_2).$$
%\end{remark}

We now describe the ring structure of the cohomology of the extended Floer complex $CF((\bL,b),(\bL,b))$ defined in \eqref{eqn:cflblb}. First, the differential acts on generators as 
$$m_1^{b,b} (X) =  yz e_\bL, \quad m_1^{b,b} (Y) =  zx e_\bL, \quad m_1^{b,b} (Z) =  xy e_\bL,$$
$$m_1^{b,b} (e_\bL)=0, \quad m_1^{b,b} (f_\bL) = xy\bar{Z} + yz \bar{X} + zx \bar{Y}.$$
The most of computations of $m_1^{b,b}$ are simply a reinterpretation of \cite[Section 10]{Seidel-g2} for our boundary-deformed $A_\infty$-operations. One also finds from the same reference that
$$ m_2^{b,b,b} (X,Y) = \bar{Z} + ze_\bL, \quad m_2^{b,b,b} (Y,X) = -\bar{Z}.$$
Taking $m_1^{b,b}$ of the first equation leads to
$$ m_1^{b,b} (\bar{Z}) = m_2^{b,b,b} (yz e_\bL, Y) + m_2^{b,b,b} ( X, zx e_\bL) = z( x X- yY)$$
by the Leibniz rule. 
To see this more geometrically, there are actually nontrivial pearl trajectories with $4$ marked points involving two constant pearl components, which are incident to $Z, \bar{Z}$ and $X, \bar{X}$ respectively, and similar for the other.
Using the symmetric arguments, we obtain
$$ m_1^{b,b} (\bar{X}) = x(yY - zZ), \quad m_1^{b,b} (\bar{Y}) = y(zZ - xX), \quad m_1^{b,b} (\bar{Z}) = z(xX - yY).$$
Therefore generating cocycles of the cohomology as a $\mathbb{C} [x,y,z]$-module are 
\begin{equation}\label{eqn:addgen}
 e_\bL, \quad yY-zZ, \quad xX -yY 
 %\quad  zZ -xX
\end{equation}
whose module structure is determined by
\begin{equation}\label{eqn:subjectoll}
\begin{split}
xy\,e_\bL=yz\,e_\bL = zx\,e_\bL=z(&xX - yY) = x(yY - zZ) =0,\\
y (xX - yY) &= -y (yY - zZ) 
\end{split}
%= y(zZ - xX) =0.
\end{equation}
where the last relation is obtained from $y (zZ - xX) =0$ modulo coboundaries of $m_1^{b,b}$. 
%Notice that generators in \eqref{eqn:addgen} is not independent as the last three add up to zero.

Finally, products among these generators which do not involve $e_\bL$ vanish as the following computation shows.
\begin{align*}
m_2^{b,b,b} ( yY - zZ, xX - yY) &= m_2^{b,b,b} ( yY, xX - yY) - m_2^{b,b,b} (zZ, xX-yY) \\
&= xy \, m_2^{b,b,b} (Y,X) - zx \, m_2^{b,b,b} (Z,X) + yz \, m_2^{b,b,b} (Z,Y)\\
&= -xy\bar{Z} - zx \bar{Y} - yz\bar{X}= -m_1^{b,b}(f_\bL)=0\\
m_2^{b,b,b} ( xX - yY, xX - yY) &= -xy\left(m_2^{b,b,b}(X,Y)+m_2^{b,b,b}(Y,X)\right) = -xyz\,e_\bL=0.
\end{align*}
Here we used the extra fact that $m_2^{b,b,b} (X,X) = m_2^{b,b,b} (Y,Y) =0$ since its associated moduli is obviously empty.

Now suppose we have an abelian cover $X$ of the pair-of-pants $P$, and denote its deck transformation group by $G$. Applying the construction \ref{subsec:BmLF}, we have an action of the dual group $\hat{G}:=\hom(G,U(1))$ on the LG model. More precisely, $\hat{G}$ acts on $\mathbb{C}[x,y,z]$ by
$$ \chi \cdot x = \chi(g_X) x,\quad \chi \cdot y = \chi(g_Y) y,\quad \chi \cdot z = \chi(g_Z) z$$
where $g_X$ is determined by the condition $\WT{X} \in \WT{\bL} \cap g_X \cdot \WT{\bL}$ for some lifting $\WT{\bL}$ of $\bL$ and similar for $y$ and $z$. The action leaves $W$ invariant. 
 
We next investigate the closed-string B-model invariants of the LG mirror $W(x,y,z)=xyz$  and its orbifoldings with respect to the action of the dual group $\hat{G}$. 

\subsection{Orbifold Koszul cohomology of $(xyz,\HG)$}\label{subsec:orbkosxyz}
Let $G$ be a finite abelian group which is the deck transformation group for the covering $X$ of the pair-of-pants $P$. Then $\HG$ acts diagonally on $\mathbb{C}[x,y,z]$ and the polynomial $W=xyz$ is invariant under $\HG$-action. For $h\in \HG$ we have three cases.
\begin{enumerate}
\item $h=1$,
\item $h$ acts nontrivially on two variables,
\item $h$ acts nontrivially on three variables.
\end{enumerate}

 In case (1), the shifted Koszul complex $K^*(\partial W^h)=K^*(\partial W)$ is given by
 \[ \xymatrix{0 \ar[r] & \C[x,y,z]\cdot \theta_x\theta_y\theta_z \ar[rrr]^-{yz \partial_{\theta_x} + xz\partial_{\theta_y}+ xy\partial_{\theta_z}} & & &\C[x,y,z] \theta_x \theta_y \oplus \C[x,y,z]\theta_y\theta_z \oplus \C[x,y,z] \theta_x\theta_z}\]
 \[\xymatrix{\ar[rrr]^-{yz \partial_{\theta_x} + xz\partial_{\theta_y}+ xy\partial_{\theta_z}}& & & \C[x,y,z]\theta_x\oplus
 \C[x,y,z]\theta_y\oplus \C[x,y,z]\theta_z \ar[rrr]^-{yz \partial_{\theta_x} + xz\partial_{\theta_y}+ xy\partial_{\theta_z}}& & & \C[x,y,z] \ar[r] & 0.}\]
The shift is trivial since $I_h=\emptyset$ for $h=1$. The resulting cohomology is given as
\[ H^i(K^*(\partial W))=
\begin{cases}
    \frac{\langle x\theta_x-y\theta_y, y\theta_y-z\theta_z \rangle}{\langle yz\theta_y-xz\theta_x,xz\theta_z-xy\theta_y,yz\theta_z-xy\theta_x\rangle} & i=-1 \\
    \C[x,y,z]/\langle xy,yz,xz \rangle & i=0 \\
    0 & {\textrm{otherwise}}
\end{cases}
\]
(see \ref{subsec:popms} for detailed computations) and $H^i(K^* (\partial W))^{\HG}$ is the set of $\HG$-invariant elements.

 In case (2) we have $W^h=0$. Assume for example that $z$ is the $h$-invariant variable (we can handle other two cases similarly). Then the shifted Koszul complex $K^*(\partial W^h)\cdot \theta_x \theta_y$ is
 \[ \xymatrix{0 \ar[r] & \C[z]  \theta_z \theta_x \theta_y\ar[r]^0 & \C[z] \theta_x \theta_y \ar[r] & 0}\]
 and evidently
 \[ (H^{-1}(K^*(\partial W^h))\cdot \theta_x \theta_y)^\HG \cong (\C[z]  \theta_z \theta_x \theta_y)^\HG, \quad 
 (H^0(K^*(\partial W^h))\cdot \theta_x \theta_y)^\HG \cong (\C[z]  \theta_x \theta_y)^\HG.\]

 For the case (3), we do not have any element of Jacobian ideal because $Fix(h)=\{0\}$, and the shifted Koszul complex is 
 \[ \xymatrix{0 \ar[r] & \C\cdot \theta_x \theta_y \theta_z \ar[r]&0.}\]

% In each case the computation of cohomology is straightforward.
By the second statement of Proposition \ref{prop:floerkos} and the following, we can relate the Floer theory of $\bL$ and Koszul cohomology of $(xyz,\HG)$.
\begin{prop}\label{prop:CFKOS}
There is an isomorphism
\[\eta:CF((\bL,xX+yY+zZ),(\bL,x'X+y'Y+z'Z)) \cong \Delta_W^1.\]

%{\color{red} notation conflict: $\Theta$ is a graph in $P$ in most of places.} 
\end{prop}

\begin{proof}
Let $\{f_\bL, X,Y,Z, e_\bL, \bar{X},\bar{Y},\bar{Z}\}$ be an ordered basis of $CF(\bL,\bL)$ as a free module. $m_1^{xX+yY+zZ,x'X+y'Y+z'Z}$ can be computed explicitly using the same argument as in \ref{subsec:popFloer}, and in the matrix form with respect to the above basis, it is given as
\[ m_1^{xX+yY+zZ,x'X+y'Y+z'Z}= 
\left(\begin{array}{cccccccc}
0 & 0 & 0 & 0 & 0 & x'-x & y'-y & z'-z \\
0 & 0 & 0 & 0 & x'-x & 0 & -xy & z'x \\
0 & 0 & 0 & 0 & y'-y & xy & 0 & -y'z' \\
0 & 0 & 0 & 0 & z'-z & -z'x & y'z' & 0 \\
0 & y'z' & z'x & xy & 0 & 0 & 0 & 0 \\
y'z' & 0 & z'-z & y-y' & 0 & 0 & 0 & 0 \\
z'x & z-z' & 0 & x'-x & 0 & 0 & 0 & 0 \\
xy & y'-y & x-x' & 0 & 0 & 0 & 0 & 0
\end{array}\right).
\]
Then the isomorphism $\eta$ is given by
\begin{equation}\label{eq:eta}
\begin{array}{c}
\eta(X)=\theta_x, \; \eta(Y)=\theta_y,\; \eta(Z)=  \theta_z,\;
\eta(-f_\bL)= \theta_x \theta_y \theta_z,\\
 \eta(e_\bL)= 1,\;
\eta(-\bar{X}) =\theta_y \theta_z,\;
\eta(-\bar{Y}) = \theta_z \theta_x,\;
\eta(-\bar{Z}) = \theta_x\theta_y. 
\end{array}
\end{equation}

\qedhere 

\end{proof}
The following is straightforward from Proposition \ref{prop:floerkos}.
\begin{cor}\label{cor:cfkostau}
For $W=xyz$, there is a quasi-isomorphism
\[\tau: (CF(({\bL},b),({\bL},b);\HG))^{\HG} \cong \big(\bigoplus_{\chi\in \HG}K^*(\partial W^\chi)\cdot \theta_{I_\chi}\big)^\HG.\]
\end{cor}

Let $f_1,\cdots,f_k \in R$ and $p_1,\cdots,p_k\in \{f_\bL,X,Y,Z,e_\bL,\bar{X},\bar{Y},\bar{Z}\}$ such that $\alpha_\chi=(\sum f_i p_i)\otimes \chi$ is an element representing a class in $H^*(CF ((\bL,b),(\bL,b);\hat{G}))^\HG)$. Among $\{p_1,\cdots,p_k\}$, let $p_1,\cdots,p_l$ be cochains whose fractional gradings are maximal. Then we have
\begin{equation}\label{eqn:tautautau}
 \tau(\alpha_\chi)=\sum_{i=1}^l f_i|_{Fix(\chi)} \cdot\eta(p_i)\in (K^*(\partial W^\chi)\cdot \theta_{I_\chi})^\HG. 
\end{equation}
Here we consider $\eta(p_i)$ as an element of $\Bbbk\langle \theta_x,\theta_y,\theta_z \rangle$ via the  identification \eqref{eq:eta}.

\section{Equivariant constructions}\label{sec:equivconst}

Our main interest lies in the situation where a finite abelian group $G$ acts freely on a Liouville manifold $(X,\omega=d\theta)$ preserving relevant structures. Since the action is free, it can be also understood as a principal $G$-bundle $G \to X \to B$ on a Liouville manifold $B$ and the Liouville structure on $X$ is pulled back from $B$. We will see that the Floer invariants (symplectic cohomology and Lagrangian Floer cohomology) of $B$ naturally admits an action of the dual group $\hat{G}$. In this section, we aim to reconstruct Floer invariants of $X$ from those of $B$ by carrying out suitable $\hat{G}$-equvariant constructions.

%If $p:(\widetilde{X}, \tilde{\omega} = d \tilde{\lambda} ) \to (X, \omega = d \lambda)$ is a covering between two Liouville manifolds such that $p^\ast \lambda = \tilde{\lambda}$, it takes a cylindrical end in $\widetilde{X}$ to a cylindrical end in $X$.

%
%\begin{remark}
%The fact that $G$ is abelian is important, as our construction involves the dual group $\hat{G}$ significantly. However, a large part of the construction here can be reformulated avoiding $\hat{G}$, and has a chance to generalize for non-abelian $G$.  
%\end{remark}

We first discretize the data of the $G$-bundle $X \to B$ along a certain submanifold of $B$ so that the holomorphic curves for Floer invariants can effectively capture the equivariant information arising from $G$-quotient.

\subsection{Trivializing the principal $G$-bundle $X \to B$}\label{subsec:Thetagraph}

Given a principal $G$-bundle $X \to B$, we choose its particular local trivialization in the following way. Suppose there exists an embedded oriented (possibly noncompact) submanifold $\Theta$ of $B$ of codimension $1$ such that each component of $B \setminus \Theta$ is simply connected. We additionally require that $\Theta$ defines a cycle in the Borel-Moore (locally finite) homology $H^{BM}_{2n-1} (B;\mathbb{Z})$, and hence its Poincar\'e dual gives an element of $H^1 (B;\mathbb{Z})$.
%It is clear from the picture that when $X$ is a punctured Riemann surface, one can take a disjoint union of embedded arcs as $\Theta$, and components of $X \setminus \Theta$ are homeomorphic to disks. {
For instance, the union of cocores in 1-handles can serve as $\Theta$ for Weinstein manifolds since their complement is a union of simply connected subsets.

Once $\Theta$ (together with its orientation) is fixed, the principal $G$-bundle $X$ can be described as the gluing of the trivial bundles across each connected component of $\Theta$. Therefore one can assign \emph{gluing data} to $\Theta$, that is, an element $g_i \in G$ to a component $\Theta_i$ of $\Theta$. (Conversely, if the data $\{ (\Theta_i, g_i)\}$ are given, one can reconstruct the $G$-bundle up to isomorphism.)

More specifically, for each component $U$ of $B \setminus \Theta$, we fix an equivariant trivialization $p^{-1}(U) \cong U \times G$ where the $G$-action on the right hand side is given by the \emph{left} multiplication on $G$. Suppose $U$ and $V$ are two components of $B \setminus \Theta$ that share $\Theta_i$ as their common boundaries, and consider a small path $\gamma$ going across $\Theta_i$ in such a way that $\gamma'$ followed by the given orientation of $\Theta_i$ form the positive orientation of $B$.
Then the gluing of the two associated trivializations $U \times G$ and $V \times G$ (each equipped with the left $G$-action) along $\gamma$ is given by the \emph{right} multiplication of $g_i$ to the $G$-factor, preserving the left $G$-action.

\begin{remark}\label{rmk:Thetagraph1}
For a punctured Riemann surface, one can allow $\Theta$ to be a graph each of whose edge carries an element of $G$ subject to the triviality condition around each vertex. (We still require that its complement is a union of simply connected regions.) Namely if we take the cyclic product of group elements assigned to edges incident to the same vertex, then the resulting element should be the identity in $G$.
\end{remark}

Making use of the gluing data $\{(\Theta_i,g_i)\}$, one can naturally assign a group element to a path in $B$.

\begin{defn}\label{def:ggamma1}
Let $\gamma: [0,1] \to B$ be a smooth path in $X$. Suppose 
$\gamma$ transversally intersects $\Theta$ at $t=c_1,c_2\cdots, c_l$ with $c_1<\cdots<c_l$. If $\epsilon_{k} \in \{-1,1\}$ denotes the parity of the intersection $\gamma$ and $\Theta$ at $\gamma(c_k)$, then the $G$-labeling $g_\gamma$ of the path $\gamma$ is defined as the product
$$g_\gamma:=g_{1}^{\epsilon_{1}} g_{2}^{\epsilon_{2}} \cdots g_{l}^{\epsilon_{l}}.$$
\end{defn}

Note that if $\gamma_1$ and $\gamma_2$ are homotopic relative to endpoints, then their $G$-labelling agrees since $\Theta$ is a cycle. Hence one can define $g_\gamma$ for $\gamma$ in a nongeneric position to be the labelling of a small perturbation of $\gamma$ (relative to endpoints). When $B$ is a punctured Riemann surface and $\Theta$ is given as a graph, one additionally consider the homotopy of $\gamma$ across a vertex, which does not cause any ambiguity due to triviality condition around the vertex. Also, if $\gamma$ is homotopic to a concatenation of other two paths $\gamma_1$ and $\gamma_2$, then $g_\gamma = g_{\gamma_1} \, g_{ \gamma_2}$.

It is easy to see that for a loop $\gamma$ in $B$ and its lifting $\tilde{\gamma}$ in $X$, $\tilde{\gamma}(1) = \tilde{\gamma} (0)$ if and only if $g_\gamma=1$. This can be generalized for general paths in $B$ and their liftings.
Recall that we have fixed a trivialization of $p:X \to B$ over the complement $B \setminus \Theta$. This gives rise to a map  
$$l: p^{-1} (B \setminus \Theta) \to G$$
which is nothing but the trivialization $p^{-1} (B \setminus \Theta) \cong B \setminus \Theta \times G$ followed by the projection to $G$. Hence $l$ labels locally the sheets of the covering $X \to B$ by elements of $G$. Obviously, the labelling $l$ is $G$-equivariant,
$$l (g\cdot x) = g l(x)$$
where the $G$-action on the right hand side is given by left multipilication.

Since $g_\gamma$ encodes the product of jumps of sheets across $\Theta$ when traveling along $\gamma$, we have:

\begin{lemma}\label{lem:pathliftendpts}
Suppose $\gamma$ is a path in $X$, and consider its arbitrary lifting $\tilde{\gamma}$ in $\WT{X}$. Then
$$ g_\gamma = l(\tilde{\gamma}(0))^{-1}l(\tilde{\gamma}(1)).$$
\end{lemma}
As $l(\tilde{\gamma}(0)) g_\gamma = l(\tilde{\gamma}(1))$ suggests, the labelling $g_\gamma$ to the path $\gamma$ measures the jump of sheets between endpoints of any of its lifting. 

\subsection{$\hat{G}$-equivariant symplectic chain complex}\label{subsec:constequivsc}
Let us now assume $B$ is a Liouville manifold, and pullback the Liouville structure to the principal $G$-bundle $X$ via the quotient map. Since the $G$-action is free, $X$ is equipped with a $G$-invariant Liouville structure. We proceed similarly to construct a  $G$-invariant almost complex structure and a Hamiltonian on $X$. Since the action is free, the Floer datum for $X$ arising in this way is generic.
For instance, $SC^*(X)$ admits an induced $G$-action for that on $X$ due to this choice.

Due to $G$-invariance, the image of a $1$-periodic Hamiltonian orbit in $X$ maps to a Hamiltonian orbit in $B$ under the quotient map $X \to B$, and the same is true for $J$-holomorphic curves upstairs. On the other hand, covering theory gives us an obstruction to lift orbits and holomorphic curves in $B$ to $X$. The gluing data $\{(\Theta_i,g_i)\}$ chosen above can be used to effectively detect liftable objects in $B$, especially those related with Floer theory of $X$.

We first define the action on the dual group $\hat{G}:=\hom(G,U(1))$ (the group of characters on $G$) on the downstair cochain complex $SC^*(B)$. For a character $\chi \in \hat{G}$ and a path $\gamma$, we set its value on $\gamma$ to be
\begin{equation}\label{eqn:chiactonpath}
\chi (\gamma) := \chi (g_\gamma).
\end{equation}
If $\alpha:[0,1]\to B$ is a Hamiltonian orbit in $B$ viewed as a generator of $SC^* (B)$, then $\chi$ acts on $\alpha$ by $\chi \cdot \alpha := \chi(\alpha) \alpha$. By linear extension, we have a chain-level action on the symplectic cohomology of $B$. As previously remarked, a loop $\gamma$ lifts to a loop if and only if $g_\gamma = 1$, whence $\chi(\gamma) = 1$ for all $\chi\in\HG$. On the other hand, if $G$ is abelian, $\chi (g) = 1$ for all $\chi \in \HG$ if and only if $g = 1$. Therefore, if $G$ is abelian, the cochains in $SC^*(B)$ that are liftable to $SC^*(X)$ are precisely the cochains which are invariant under the $\hat{G}$-action.

Hereby we assume $G$ is abelian. More interesting part is the $\hat{G}$-twisting on the moduli space of holomorphic curves in $B$ relevant to the algebraic structure on $SC^* (B)$. In fact, we will enlarge $SC^* (B)$ to 
\begin{equation}\label{eqn:schatg}
SC^*_{\HG} (B):= SC^* (B) \otimes \Bbbk [ \hat{G} ],
\end{equation}
and construct a $\hat{G}$-\emph{twisted} algebraic structure on it to reflect the lifting information. 
%Readers are warned that $SC_{\hat{G}} (B)$ is irrelevant to the equivariant symplectic cohomology. After all, $B$ does not have a geometric $\HG$-action. 
More precisely, recall from Lemma \ref{lem:equatorsh} that the domain of each element $u$ in the moduli space of pseudo-holomorphic curves associated with algebraic operations $SC^*(B)$ carries a prescribed path $\gamma_{u,k}$ from (the starting point of) the $k$-th input orbit to the output. The differential $d_{\hat{G}}$ on $SC_\HG^{*} (B)$ is then given by
\begin{equation}\label{eqn:equivdiff}
 d_{\hat{G}} (x \otimes \chi) := \sum_{y:orbits} \sum_{u \in \mathcal{M}(y;x)} \chi(\gamma_{u, 1}) (y \otimes \chi) 
\end{equation}
where $\gamma_{u,1}$ is a path from $x$ to $y$ in this case. Similarly, the product on $SC_\HG^{*} (B)$ is defined as
\begin{equation}\label{eqn:equivprod}
(x_1 \otimes \chi_1) \cdot (x_2 \otimes \chi_2) = \sum_{y:orbits} \sum_{u \in \mathcal{M}(y;x_1,x_2)} \chi_1(\gamma_{u,1}) \chi_2 (\gamma_{u,2}) ( y \otimes \chi_1 \chi_2).
\end{equation}
With these operations, we have:
\begin{prop}
Equation \eqref{eqn:equivdiff} defines a cochain complex equipped with a product \eqref{eqn:equivprod} satisfying the Leibnitz rule (and associative on the cohomology) on $SC_\HG^* (B)$. Both of operations are $\hat{G}$-equivariant where the $\hat{G}$-action is defined by
\begin{equation}\label{eqn:hatGactionSC}
 \chi \cdot (x \otimes \chi') := (\chi \cdot x) \otimes \chi'
\end{equation}
viewing $x$ as a path (loop) in $B$. 
\end{prop}

\begin{proof}
Recall that the $g_\gamma$ for a path $\gamma$ in $B$ is invariant under a path-homotopy (relative to end point), and is consistent with concatenation of paths. Since we already have such structures well-established on $SC^* (B)$, it only remains to check if the $\hat{G}$-twist is coherent. Namely, we need to show that concatenated paths at the two ends of the cobordism we employed to prove algebraic compatibility for $SC^*(B)$ (such as $d^2 =0$ and the Liebnitz rule) are homotopic in order to achieve the analogous algebraic compatibility for $SC_\HG^* (B)$. This is straightforward from Lemma \ref{lem:equatorsh}. 
\end{proof}

\begin{remark}
We emphasize that although they have the same underlying vector space, $SC_\HG^*(B)$ is different from the algebraic semi-direct product of $SC^* (B)$ with respect to $\hat{G}$-action on it. Algebraic operations on $SC_\HG^*(B)$ are not completely determined by those in $SC^* (B)$ and the $\HG$-action, as one needs to keep track of finer homotopy data of individual holomorphic maps contributing to algebraic operations.
\end{remark}

Notice that we can actually define an $A_\infty$-structure on $SC_\HG^* (B)$ with $m_k$ given analogously to \eqref{eqn:equivdiff} and \eqref{eqn:equivprod}, making use of all the paths from inputs to the output. Then the same argument proves the $A_\infty$-relation among the operations $m_k^{\hat{G}}$, and that the resulting structure is $\hat{G}$-equivariant. We will not consider this $A_\infty$-structure in the paper. 

Let us look into the $\hat{G}$-action on $SC_\HG^* (B)$ more closely. As discussed, the nontriviality of the $\hat{G}$-action reveals the obstruction for a Hamiltonian orbit to lift to $X$. In fact, one has the following.

\begin{prop}\label{prop:isomupdownsc}
There is a $G$-equivariant chain-level isomorphism
$$SC^* (X) \to \left(SC_\HG^* (B) \right)^{\hat{G}}$$ 
intertwining differentials and products, where the $G$-action on the right side is given by $g \cdot (\alpha \otimes \chi) = \chi(g) \alpha \otimes \chi$. 
%the right hand side denotes the $\hat{G}$-invariant part of $SC^\ast (X;\HG)$.
\end{prop}

%The isomorphism is not canonical. It depends 

\begin{proof}
By definition of the $\hat{G}$-action \eqref{eqn:hatGactionSC}, the right hand $\left(SC_\HG^* (B) \right)^{\hat{G}}$ is generated by images of Hamiltonian orbits in $X$ coupled with characters in $\hat{G}$. Suppose $a$ is such a Hamiltonian orbit in $B$, and hence is liftable to some orbit in $X$. Since the $G$-action on $X$ is free, there are $|G|$-many lifts of $a$. We label these lifts by elements of $G$ using our labelling $l: p^{-1} (B \setminus \Theta) \to G$ of (local) sheets of the covering $X \to B$. We denote by $\tilde{a}_g$ the Hamiltonian orbit in $X$ which lifts $a$ such that $l (\tilde{a}_g (0)) = g$. (Generically, the starting point $a(0)$ is away from $\Theta$.) Note that any Hamiltonian orbit in $X$ is given in this form.

Now, with this notation, we define a chain-level map
\[\Psi:SC^*(X) \to \left(SC_\HG^* (B)\right)^{\hat{G}} \qquad
 \WT{a}_g \mapsto \sum_{\chi\in \HG} \frac{\chi(g) ( a  \otimes \chi) }{|\HG|}.\]
Since $a$ is a liftable loop, $\Psi$ lands on the $\hat{G}$-invariant part of $SC_\HG^*(B)$. Moreover, $\Psi$  is a bijection as its inverse (set-theoretical, at this moment) is given by
\[\Psi^{-1} :\left( SC_\HG^* (B) \right)^{\HG}\to SC^*(X) \qquad a \otimes\chi\mapsto \sum_{g\in G}\chi(g^{-1})\WT{a}_g.\]
Hence, it suffices to show that $\Psi$ intertwines algebraic structures on both sides. $G$-equivariance of $\Psi$ follows directly from $h \cdot \tilde{a}_g = \tilde{a}_{hg}$.

To see that it is a chain map, consider a Floer cylinder $\WT{u}:(-\infty,\infty)\times S^1 \to X$ asymptotic to Hamiltonian orbits $\WT{a}_g$ and $\WT{b}_h$ at $\pm \infty$. Write $a:=p\circ \WT{a}_g$, $b:=p\circ \WT{b}_h$, and $u:=p\circ \WT{u}$. Recall that we have chosen a path from an input marking $(-\infty,0)$ to the output marking $(\infty,0)$ (simply $\{t=0\}$ in the case of cylinders), which determines a path $\gamma_{u, 1}$ joining starting points $a(0)$ and $b(0)$. Let $\tilde{\gamma}$ be its lift starting from the $\tilde{a}_g(0)$, which should end at $\tilde{b}_h(0)$ by the uniqueness of the lifting(we lifted $\tilde{u}$).
On the other hand, by Lemma \ref{lem:pathliftendpts}, we have
$$ h=l(\tilde{\gamma} (1)) = l(\tilde{\gamma} (0))\, g_{\gamma_{u,1}} = g\, g_{\gamma_{u,1}} .$$
%which implies $\chi(\gamma_{u,1}) = \chi (g)^{-1} \chi(h)$. 
This establishes the correspondence between $\mathcal{M}(\tilde{a}_g,\tilde{b}_h)$ and $\{ u \in \mathcal{M}(a,b): g_{\gamma_{u,1}} = g^{-1}h\}$.
Therefore, the coefficient $\langle d_{\hat{G}} (\Psi(\WT{a}_g)) , b \otimes \chi_0 \rangle$ of $b \otimes \chi_0 $ in $d_{\hat{G}} (\Psi(\WT{a}_g))$ is given by
\begin{align*}
\langle d_{\hat{G}} (\Psi(\WT{a}_g)) , b \otimes \chi_0 \rangle & = \left\langle d_{\hat{G}} \left(\sum_{\chi\in \HG} \frac{\chi(g)a\otimes \chi}{|\HG|}\right) , b \otimes \chi_0\right\rangle\\
&= \left\langle  \frac{\chi_0(g) d_{\hat{G}}(a\otimes \chi_0)}{|\HG|}, b\otimes \chi_0 \right\rangle\\
&=  \sum_{k \in G} \sum_{ \substack{u \in \mathcal{M}(a,b) \\  g_{\gamma_{u, 1}} = k    } } \frac{\chi_0 (g)\chi_0 (\gamma_{u,1} )}{|\HG|}\\
&=  \sum_{k \in G} \sum_{ \tilde{u} \in \mathcal{M} (\tilde{a}_g, \tilde{b}_{  gk})  } \frac{\chi_0 (gk)} {|\HG|} =  \sum_{h \in G} \sum_{ \tilde{u} \in \mathcal{M} (\tilde{a}_g, \tilde{b}_h)  } \frac{\chi_0 (h)} {|\HG|}
\end{align*}

On the other hand,
\begin{align*}
\langle \Psi(d (\WT{a}_g)) , b \otimes \chi_0 \rangle & = \left\langle  \Psi \left(\sum_{h \in G} \sum_{\tilde{u} \in \mathcal{M}(\tilde{a}_g,\tilde{b}_h)} \tilde{b}_h \right), b \otimes \chi_0 \right\rangle\\
&=\left\langle    \sum_{h \in G} \sum_{\tilde{u} \in \mathcal{M}(\tilde{a}_g,\tilde{b}_h)} \sum_{\chi \in \hat{G}} \frac{\chi(h) (b \otimes \chi)}{|\hat{G}|}, b \otimes \chi_0 \right\rangle \\
&=   \sum_{h \in G} \sum_{ \tilde{u} \in \mathcal{M} (\tilde{a}_g, \tilde{b}_h)  } \frac{\chi_0 (h)} {|\HG|},
%&=\sum_{\chi\in \HG} \frac{\chi(h)b\otimes \chi}{|\HG|}\\
%&=\Psi(\WT{b}_h)\\
%&=\Psi(m_{1,\WT{u}}(\WT{a}_g)).
\end{align*}
which proves $d_{\HG} \circ \Psi = \Psi \circ d$.

We next show that $\Psi$  intertwines the products.
Suppose $\tilde{u}$
%$\mathbb{CP}^1\setminus\{0,1,\infty\}\to\WT{\Sigma}$ 
is a Floer solution for the product Hamiltonian orbits $\tilde{a}_g$ and $\tilde{b}_h$ contributing to the coefficient of $\tilde{c}_k$ in the product, $g,h,k \in G$. Let $a:=p\circ\tilde{a}_g$, $b:=p\circ\tilde{b}_h$, $c:=p\circ\tilde{c}_k$, and $u:=p\circ\tilde{u}$, as before. Lemma \ref{lem:pathliftendpts} implies
$$g_{\gamma_{u,1}} = g^{-1}k ,\quad g_{\gamma_{u,2}} = h^{-1}k,$$
and hence the one-to-one correspondence 
$$\mathcal{M}(\tilde{c}_k;\tilde{a}_g,\tilde{b}_h) \leftrightarrow \{ u \in \mathcal{M}(c;a,b) : k= g\,g_{\gamma_{u,1}} =  h\,g_{\gamma_{u,2}}\}.$$
We again compute the matrix coefficient for the multiplication:
\begin{align*}
\langle \Psi(\tilde{a}_g) \cdot \Psi(\tilde{b}_h ) , c \otimes \chi_0 \rangle &= \left\langle   \left(\sum_{\chi\in \HG} \frac{\chi(g)a\otimes \chi}{|\HG|} \right) \cdot  \left( \sum_{\chi\in \HG} \frac{\chi(h)b\otimes \chi}{|\HG|} \right), c \otimes \chi_0\right\rangle\\
&= \sum_{\chi_1 \chi_2 = \chi_0 } \sum_{u \in \mathcal{M}(c;a,b)} \frac{\chi_1(g) \chi_2(h) \chi_1(\gamma_{u,1}) \chi_2 (\gamma_{u,2})  }{|\hat{G}|^2}
 \\
&=  \sum_{g',h' \in G} \sum_{\chi_1 \chi_2 = \chi_0 } \sum_{\substack{u \in \mathcal{M}(c;a,b) \\ g_{\gamma_{u,1}}=g', g_{\gamma_{u,2}}=h' } } \frac{\chi_1(g) \chi_2(h) \chi_1(g') \chi_2 (h')  }{|\hat{G}|^2}
 \\
 &=  \sum_{k,g'' \in G} \sum_{\substack{u \in \mathcal{M}(c;a,b) \\ g_{\gamma_{u,1}}=g^{-1}kg'', g_{\gamma_{u,2}}=h^{-1}k } } \sum_{\chi_1 \in \hat{G}} \frac{\chi_0(k) \chi_1(g'')  }{|\hat{G}|^2}
 \end{align*}  
 where in the last line, we made substitutions $g' = g^{-1} k g''$ and $h' = h^{-1}k$ for the obvious later purpose.  Here, the issue is that there is no guarantee if one can make a lift of $u$ having two ends exactly at the given $\tilde{a}_g$ and $\tilde{b}_h$. Notice however that $\sum_{\chi_1 \in \hat{G}} \chi_1 (g'') = 0$ unless $g''$ is the identity. Therefore it is no harm to set $g''=1$ in the above, and obtain
 \begin{align*} 
\langle \Psi(\tilde{a}_g) \cdot \Psi(\tilde{b}_h ) , c \otimes \chi_0 \rangle  &=  
\sum_{k   \in G}  \sum_{\substack{u \in \mathcal{M}(c;a,b) \\ g_{\gamma_{u,1}}=g^{-1}k , g_{\gamma_{u,2}}=h^{-1}k } } \frac{\chi_0(k)  }{|\hat{G}|}    \\
  &=  \sum_{k \in G} \sum_{ u \in \mathcal{M}(\tilde{c}_k;\tilde{a}_g,\tilde{b}_h) } \frac{\chi_0 (k)  }{|\hat{G}|},  
\end{align*}
and it is straightforward to check that this equals $\langle \Psi (\tilde{a}_g \cdot \tilde{b}_h), c \otimes \chi_0 \rangle$.
\end{proof}

\subsection{$\hat{G}$-equivariant Lagrangian Floer theory}\label{subsec:hatgequivlft}

Given a spin, unobstructed compact connected Lagrangian $\bL$ in $B$, one can similarly form a $\hat{G}$-\emph{equivariant} extension of $CF(\bL,\bL)$, analogously to $SC_\HG^* (B)$. 
We work over $\Bbbk=\Lambda$ in the rest of the section for generality, but in the actual application (Section \ref{sec:KSRS}), we will use $\mathbb{C}$-coefficients which perfectly serves our purpose thanks to Lemma \ref{lem:seidelLagoverc}.
As the domain of contributing pseudo-holomorphic maps (which are disks) are homotopically trivial, it is relatively easy to fix the extra data on the moduli space. We will use these date to twist $CF(\bL,\bL)$ with respect to the $\hat{G}$-action on it constructed in the following way. Throughout, we assume $\bL$ is liftable, or equivalently, $\iota_* \pi_1 (L)$ lies inside $p_* \pi_1 (X)$ where we express $\bL$ as an (image of) Lagrangian immersion $\iota: L \to B$. Hence we have a lifting $\tilde{\iota} : L \to X$ of $\iota$.
 
For any Floer generator $Z$, we take a path $\gamma_Z$ from $Z$ to itself, which is an image of a smooth path in $L$. More precisely, we require $(\gamma_Z [0,\epsilon), \gamma_Z (1-\epsilon,1])$ to represent the branch jump associated with $Z$ when $Z$ is an immersed generator supported at a transversal self-intersection point. In particular, $\gamma_Z$ is a loop in $B$, and one can assign an element of $G$ by taking $g_Z:=g_{\gamma_Z}$. Note that the ambiguity of $\gamma_Z$ lies in $\iota_* \pi_1 (L)$ and hence in $p_* \pi_1 (X)$. Therefore $g_{\gamma_Z}$ is well-defined.

If $Z$ is a non-immersed generator, that is, a critical point of a Morse function on $L$, the path $\gamma_Z$ is an image of a loop in $L$ (it can be simply the constant loop), and hence it lifts to a loop in $X$ due to our assumption on $\pi_1 (L)$. If we use other models to define $CF(\bL,\bL)$, then the definition of $\gamma_Z$ should be modified accordingly. For example, if $Z$ is a Hamiltonian chord then, one first takes a path from $Z(0)$ to $Z(1)$ that is an image of a smooth path in $L$, and glue this with the chord $Z$ itself. This produces a loop in $B$ which is homotopic (in $B$) to an image of a loop in $L$ (by choosing a small Hamiltonian if necessary), and hence the lift of $\gamma_Z$ is again a loop in $X$. Therefore it is natural to define $g_Z:=1$ in this case. In summary,

\begin{defn}
For a compact connected Lagrangian immersed $\bL$ given by $\iota: L \to B$, the $G$-labelling on $CF(\bL,\bL)$ is a map $CF(\bL,\bL) \to G$, which is trivial on non-immersed generators, and which  assigns $g_Z:=g_{\gamma_Z}$ to each immersed generator $Z$ where $\gamma_Z$ is a loop with $(\gamma_Z [0,\epsilon), \gamma_Z (1-\epsilon,1])$ representing the branch jump associated with $Z$.
\end{defn}

Once we have the $G$-labelling, the $\hat{G}$-action on $CF(\bL,\bL)$ is given by 
\begin{equation}\label{eqn:hatgonbl}
\chi\cdot Z:=  \chi(g_Z) Z.
\end{equation}

\begin{lemma}
The $\hat{G}$-action on $CF(\bL,\bL)$ defined in \eqref{eqn:hatgonbl} agrees with the one in \cite[Chapter 7]{CHL}.
\end{lemma}

\begin{proof}
In both cases, the action of $\hat{G}$ is set to be trivial for non-immersed generators.
Recall that the $\hat{G}$-action in \cite{CHL} is simply by
$$\chi\cdot Z = \chi(g) Z$$
for an immersed generator $Z$ if an inverse image of $Z$ lies in $h  \cdot \tilde{\iota}(L) \cap gh \cdot \tilde{\iota} (L)$ for some $g \in G$ where we fix one arbitrary lifting $\tilde{\iota}(L)$ of $\bL$ to $X$. 
On the other hand, the lifting of $\gamma_Z$ starting from $h \cdot \tilde{\iota}(L)$ ends at $g_Z h \cdot \tilde{\iota}(L)$, which means the corresponding lifting $\WT{Z}$ lies in the intersection of the two. Since $g_Z$ is uniquely determined, it follows that $g= g_Z$.
\end{proof}

We next equip each element $(u:(D^2,\partial D^2) \to (X,\bL),\{z_i\})$ in the moduli space for $CF(\bL,\bL)$ with paths from the input marked point $z_i$ to the output $z_0$, analogously to Lemma \ref{lem:equatorsh}. This can be easily achieved since $D^2$ is contractible. Denote the images of these paths under $u$ by $\gamma_{u,1},\cdots, \gamma_{u,k}$. If the marked point carries a strip-like end, then we additionally assume that the corresponding $\gamma_{u,i}$ follows $\{t=0\}$ while staying in the strip-like end 
%(this will never appear in our actual application as we mostly employ the Morse model). 
Notice that the homotopy types of these paths only depend on the class $\beta$ of the disk, so we may write 
$$\gamma_{\beta,i}:=\gamma_{u,i}$$ 
for $u$ in the class $\beta \in \pi_2 (X, \bL)$.

%By our choice, homotopy types are preserved through degeneration of holomorphic disks (i.e., disk bubbling). More specifically, one may use the popsicle moduli as in \cite{...}, for which geodesics between marked points can serve as $\gamma_{u_i}$.

The rest of the construction is parallel to the previous one. The underlying $\Bbbk$-module is given as
\begin{equation}\label{eqn:cfhatg}
CF_{\hat{G}} (\bL,\bL):=CF(\bL,\bL) \otimes \Bbbk [\hat{G}]
\end{equation}
on which $\hat{G}$ acts as $\chi: Z \otimes \chi'  \mapsto (\chi\cdot Z) \otimes \chi' $. (It also admits an action of $g\in G$ which multiplies $\chi(g)$ to any element in $CF(\bL,\bL) \otimes \chi$.) The $A_\infty$-operations are defined by
\begin{equation}\label{eqn:mkequiv}
m_{k,\HG} (Z_1\otimes \chi_1,\cdots,Z_k \otimes \chi_k)= \sum_{\beta}  \chi_1(\gamma_{\beta,1}) \cdots \chi_k(\gamma_{\beta,k}) m_{k,\beta} (Z_1,\cdots, Z_k) \otimes \chi_1 \cdots \chi_k
\end{equation}
where $\chi_i (\gamma_{\beta,i})=\chi_i (g_{\gamma_{\beta,i}})$, and  $g_{\gamma_{\beta,i}}$ is the $G$-labelling of the path $\gamma_{\beta,i}$. This $m_{k,\HG}$ is obviously $\hat{G}$-equivariant. Note that $CF_{\hat{G}} (\bL,\bL)$ coincides with $CF(\bL,\bL;\hat{G})$ in \ref{subsec:BmLF} as a module, but not as an $A_\infty$-algebra a priori.

\begin{prop}\label{prop:lagcfupdown}
$CF_{\HG}(\bL,\bL)$ with $m_{k,\HG}$ \eqref{eqn:mkequiv} is an $A_\infty$-algebra, which is strictly $\hat{G}$-equivariant. Moreover, we have a $G$-equivariant isomorphism 
$$ f=\{f_1,f_2=f_3=\cdots=0\}: CF(\WT{\bL},\WT{\bL}) \to \left(CF_{\hat{G}} (\bL,\bL) \right)^{\hat{G}}.$$
\end{prop}

\begin{proof}
The first half of the statement directly follows from homotopy invariance of the paths $\gamma_{\beta,i}$ through disk degeneration.  To find an isomorphism $f$, we label generators of $CF(\WT{\bL},\WT{\bL})$ using the $G$-equivariant trivialization $p^{-1}(B \setminus \Theta) \cong B\setminus\Theta \times G$ as in the proof of Proposition \ref{prop:isomupdownsc}. One needs to take $\bL$ in a generic position for  Floer generators to be away from $\Theta$. Thus we have $\WT{Z}_g$ in $CF(\WT{\bL},\WT{\bL})$ for each $g \in G$ projecting to the same $Z$ in $CF(\bL,\bL)$. 

Let us then define $ \Psi: CF(\WT{\bL}, \WT{\bL}) \to \left(CF_{\hat{G}} (\bL,\bL) \right)^{\hat{G}}$ by
$$ \Psi (Z_g) := \frac{1}{|\hat{G}|} \sum \chi(g) Z \otimes \chi$$
which is a module isomorphism by exactly the same reason as in the proof of Proposition \ref{prop:isomupdownsc}. We show that $f=\{f_1:=\Psi,f_2=f_3=\cdots=0\}$ gives an $A_\infty$-algebra homomorphism, which amounts to verifying
$$ m_{k,\HG} \left(\Psi(Z_{1,g_1}), \cdots, \Psi(Z_{k,g_k}) \right) = \Psi \left(m_k (Z_{1,g_1},\cdots, Z_{k,g_k})\right).$$
We first compute the coefficient of $Z_0 \otimes \chi_0$ of the left hand side.
\begin{equation}\label{eqn:psiformk}
\begin{array}{l}
\langle m_{k,\HG} \left(\Psi(Z_{1,g_1}), \cdots, \Psi(Z_{k,g_k}) \right), Z_0\otimes \chi_0 \rangle \\ =  \frac{1}{|\hat{G}|^k}{\displaystyle\sum_\beta}  {\displaystyle\sum_{\chi_1\cdots\chi_k=\chi_0}} \left( \prod_{j=1}^n \chi_j (g_j) \chi_j (\gamma_{\beta,j}) \right) \, m_{k,\beta} (Z_1,\cdots, Z_k) T^{\omega(\beta)} \\
=   {\displaystyle\sum_\beta} \, \chi_0(g_0) m_{k,\beta} (Z_1,\cdots, Z_k) T^{\omega(\beta)}{\displaystyle\sum_{h_1,\cdots,h_{k-1}}} \left( \Pi_{j=1}^{k-1} ( \sum_{\chi \in \hat{G}} \chi (h_j) ) \right)  \end{array}
\end{equation}
where in the last line we arrange the sum in terms of $h_1,\cdots,h_{k-1},g_0$ such that
$$g_{\gamma_{\beta,1}} = g_1^{-1} g_0  h_1, \quad \cdots, \quad g_{\gamma_{\beta,k-1}} = g_{k-1}^{-1} g_0  h_{k-1}, \quad g_{\gamma_{\beta,k}} = g_k^{-1} g_0 ,$$
i.e., we set the new parameter $g_0$ to be $g_0:=  g_k g_{\gamma_{\beta,k}}$, which determines the rest $h_1,\cdots, h_{k-1}$. Since $\sum_{\chi \in \hat{G}} \chi(h_j)=0$ unless $h_j=1$, the only contribution to the last sum in \eqref{eqn:psiformk} comes from $u$ whose class $\beta$ satisfies $g_{\gamma_{\beta,i}}= g_i^{-1} g_0$ for all $i$ and some $g_0 \in G$. Again, as in the proof of Proposition \ref{prop:isomupdownsc}, Lemma \ref{lem:pathliftendpts} tells us that these disks lift to those in $X$ with inputs $Z_{1,g_1},\cdots,Z_{k,g_k}$ and the output $Z_{0,g_0}$. Using this, it is easy to see that the last line of \eqref{eqn:psiformk} equals $\langle \Psi(m_k (Z_{1,g_1},\cdots, Z_{k,g_k}), Z_0 \otimes \chi_0 \rangle$, which finishes the proof.
\end{proof}

A similar argument gives an isomorphism $CF((\WT{\bL},\tilde{b}),(\WT{\bL},\tilde{b})) \to \left(CF_{\hat{G}} ((\bL,b),(\bL,b)) \right)^{\hat{G}}$ where we set $\tilde{b} = \sum_i x_i \sum_{p( \tilde{X}_{i,j} ) = X_i} \tilde{X}_{i,j}$.

\subsection{Example: Seidel Lagrangian in the pair-of-pants}\label{subsec:expopsei} When $X$ is a principal $G$-bundle of $B=P$ a pair-of-pants, and $\bL$ the Seidel Lagrangian (see \ref{subsec:popFloer}), there is a special choice of $\Theta$ whose associated $CF_{\hat{G}} (\bL,\bL)$ agrees with the construction in \cite{CLe} (see \ref{subsec:BmLF}). We present here one particular example of such $\Theta$ which is given as a trivalent graph on the pair-of-pants, but it is also possible to choose $\Theta$ as a union of embedded arcs. 

We take $\Theta$ as follows. It consists of 6 half-infinite edges, whose finite ends are incident to either of two trivalent vertices. See Figure \ref{fig:G-graph}. It is clear that any covering $p:X \to P$ trivializes over the complement of $\Theta$, i.e. $p^{-1} (P \setminus \Theta) \cong P \setminus \Theta \times G$. Each edge is assigned with an element of $G$, as marked in the picture. Notice that around each vertex, the (ordered) product of the group elements attached to incident edges is trivial, and hence one can assign without ambiguity a group element to any generic path in $P$ as in Lemma \ref{lem:pathliftendpts}.

\begin{figure}[h]
\includegraphics[scale=0.5]{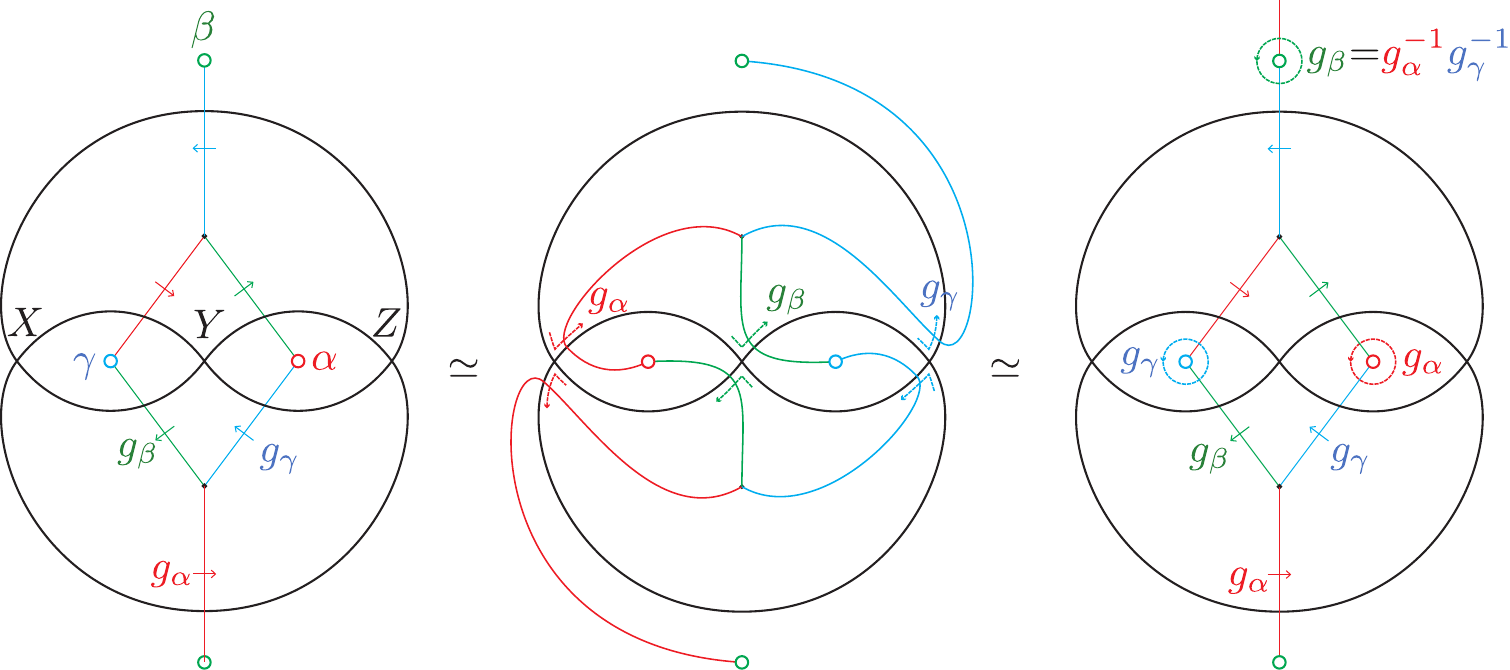}
\caption{$\Theta$ whose associated $CF_{\hat{G}} (\bL,\bL)$ coincides with $CF(\bL,\bL;\hat{G})$ as algebras. (This is valid only for an abelian $G$. One of labelings in the bottom triangle should be replaced by its conjugate when $G$ is not abelian.)}
\label{fig:G-graph}
\end{figure}

Throughout, we use the presentation of $\pi_1 (P)$ given as 
$$\pi_1 (P)=\langle g_\alpha,g_\beta, g_\gamma : g_\alpha g_\beta g_\gamma =1 \rangle.$$ 
One of generators is redundant, but we prefer this symmetric expression which has a clear geometric origin as follows. In fact, the abelian group $G$ (the deck transformation of the abelian covering $X \to P$ in this context) is naturally a quotient of $\pi_1 (P)$, and one can identify $g_\alpha$, $g_\beta$ and  $g_\gamma$ with loops homotopic to $S_\alpha$, $S_\beta$ and $S_\gamma$ in \ref{subsec:SHP} without ambiguity of the choice of a base point. Here, we use the same letters $g_\alpha, g_\beta, g_\gamma$ to denote the projections of the corresponding generators to $G$ (and we will do so from now on, when there is a no danger of confusion).

Furthermore, the Seidel Lagrangian $\bL$ is liftable since, as a loop, its corresponding group element is conjugate to $g_\gamma g_\beta g_\alpha = g_\beta^{-1} g_\alpha^{-1} g_\beta g_\alpha$, and hence should be in the image of $p_* \pi_1 (X)$ since $p$ is an abelian covering.
Therefore our construction produces $CF_{\hat{G}} (\bL,\bL)$ whose $\hat{G}$-invariant part recovers $CF(\WT{\bL},\WT{\bL})$.

\begin{lemma}\label{lem:idbetntwo}
The obvious map 
$$CF_{\hat{G}} (\bL,\bL) \to CF  (\bL,\bL;\hat{G})$$ 
sending $Z \otimes \chi$ to the same expression on the right hand side precisely identifies the $A_\infty$-structure for $\Theta$ we have chosen above.
\end{lemma}

\begin{proof}
Proof directly follows the picture. The middle diagram in Figure \ref{fig:G-graph} shows an isotope of $\Theta$, and it has the following feature. Each time a path $\tau$ in $\bL$ makes the turn at the odd/even-degree corner of $X$ (resp. $Y$ and $Z$), $g_\tau$ takes $g_\alpha$/$g_\alpha^{-1}$ (resp, $g_\beta^{\pm}$,$g_\gamma^{\pm}$) as its factor. Otherwise $g_\tau$ remains trivial since, for instance, if we pass through $X$ without making a turn, then $g_\alpha$ and $g_\alpha^{-1}$ cancel out. Therefore $g_\tau$ is obtained by multiplying of $g_\alpha,g_\beta,g_\gamma$ for each odd-degree $X$,$Y$,$Z$ in $\tau$ and multiplying their inverses for corresponding even corners in $\tau$. Notice its consistency with the $\hat{G}$-action on $CF(\bL,\bL)$:
$$\chi(X) = \chi(g_\alpha), \quad \chi(Y) = \chi(g_\beta), \quad \chi(Z) = \chi(g_\gamma),$$
$$\chi(\bar{X}) = \chi(g_\alpha^{-1}), \quad \chi(\bar{Y}) = \chi(g_\beta^{-1}),\quad \chi(\bar{Z}) = \chi(g_\gamma^{-1}).$$

On the other hand, observe that the path $\gamma_{\beta,i}$ for a disk $u:(D^2,\partial D^2) \to (P,\bL)$ (in the $\beta \in \pi_2(P,\bL)$) is homotopic to the image under $u$ of the arc in $\partial D^2$ running from the $i$-th marked point to the $0$-th marked point. Therefore the factor $\chi_i(\gamma_{\beta,i})$ in $m_k(Z_1 \otimes \chi_1,\cdots, Z_k \otimes \chi_k)$ on the left hand side $CF_{\hat{G}} (\bL,\bL)$ is nothing but the product $\chi_i (Z_1) \cdots \chi_i (Z_i)$, and hence $\prod_i \chi_i (\gamma_{\beta,i})$ recovers the twisting factor on $CF (\bL,\bL;\hat{G})$ appearing in \eqref{twistedainfty}.
\end{proof}

%{\color{red}
%\subsection{Nonabelian G}
%A large part of construction here generalizes to the case of nonabelian group $G$. For this, one needs to reformulate the above in terms of $k[G]$ instead of $k[\hat{G}]$. Notice that there is an obvious loss of information when passing to $\hat{G}=\hom(G,U(1)) = \hom(G/[G,G], U(1))$ for nonabelian $G$. In fact, one can view elements of $\hat{G}$ as characters of $G$ associated with $1$-dimensional representations, and every irreducible representation of $G$ is $1$-dimensional when $G$ is abelian. The group operation on $\hat{G}$ agrees with the convolution product between characters (WRONG).
%
%
%Now suppose we are given an nonabelian group $G$. According to the above discussion, it is natural to consider the algebra generated by irreducible characters of $G$, which now involve higher dimensional ones. Classical representation theory of a finite group tells us:
%
%\begin{prop}
%\end{prop}
%
%See for e.g., \cite{} for the proof.
%}

\section{Orbifold mirror LG models for some punctured Riemann surfaces}\label{sec:KSRS}

In this section, we apply our earlier construction to the class of punctured Riemann surfaces $X$ whose quotient by a finite abelian group is isomorphic to the pair-of-pants $P$. 
 For given such $X$, we prove that the symplectic cohomology of $X$ is naturally isomorphic to the orbifold LG B-model invariants defined in Section \ref{subsec:orbLGB}. The coincidence of the two closed-string invariants can also be shown by calculating both sides directly, for e.g. \cite{GP20} for $SH^*(X)$.
Here, we take more TQFT-oriented approach, as we will relate the two by an enumerative closed-open map, called the Kodaira-Spencer map in the spirit of \cite{FOOO10b}. 

Our strategy is first to establish the mirror symmetry for the pair-of-pants, and work equivariantly to deduce the analogous result for its abelian coverings. We begin with reviewing the construction of Kodaira-Spencer map, which is essentially a closed-open map applying to the single Lagrangian 
boundary.

There is another mirror construction of a general punctured Riemann surfaces by gluing local pieces isomorphic to $W=xyz$ according to some toric data determined by pair-of-pants decomposition of the surface (and its tropicalization). The resulting global mirror takes the form of a LG model on a toric Calabi-Yau 3-fold, which is actually a resolution of the orbifold LG mirror that we have here (for the case of abelian covers of the pair-of-pants.)

\begin{remark}
Any finite abelian group $G= \mathbb{Z}/ m  \times  \mathbb{Z} / n $ can serve as the deck transformation group of $X \to P$. By elementary covering theory, the genus $g$ of $X$ and its number of punctures $N$ in this case are related by
\begin{equation}\label{eqn:abelcovergN}
g= \frac{(m-1)(n-1)-\gcd(m,n) + 1}{2} = \frac{mn - N}{2}, \quad N=m+n+\gcd(m,n).
\end{equation}  
%These two numbers are indeed the orders of deck transformations corresponding to two generators of $\pi_1 (\Sigma))$. 
In particular, there are finitely many possible $N$ for a fixed $g$ if $g \geq 1$.
%, and in fact, one can easily find a lower bound of $N$ for a given $g$ from \eqref{eqn:abelcovergN}. 
%We do not know if there exists an abelian cover of the pair-of-pants for any given $(g,N)$ in \eqref{eqn:abelcovergN}.
\end{remark}

\subsection{Kodaira-Spencer map}\label{subsec:constksmap}
Let $X$ be a punctured Riemann surface with an abelian covering $\pi: X \to P$, and $\tilde{\bL}$ the inverse image of the Seidel Lagrangian $\bL$ in $X$ under $\pi$.
We define one of main objects in our construction, the Kodaira-Spencer map, essentially as the closed-open map in \cite{RS} composed with the projection $HH^* (Fuk(X)) \to HH^* (CF(\tilde{\bL},\tilde{\bL}))$,
\begin{equation}\label{eqn:OCKS}
SC^*(X) \to HH^*(Fuk(X))  \to HH^* (CF(\tilde{\bL},\tilde{\bL}))
\end{equation}
with some modification and reinterpretation of $HH^* (CF(\tilde{\bL},\tilde{\bL}))$ in the sense of Remark \ref{rmk:lblbhoch}. On technical side, one must also address the issue that the our model for Floer theory of $\bL$ does not involve Hamiltonians with higher slopes whereas $SC^\ast(X)$ does. We spell out shortly how to handle this.
The resulting map is also a noncompact analogue of the ring homomorphism appearing in \cite{FOOO10b} with the same name, although our technical setup is slightly different. 

Instead of working directly on $X$, we first pass to the quotient $P$ of $X$ and construct \eqref{eqn:OCKS} for $P$, and recover that of $X$ using the equivariant construction in Section \ref{sec:equivconst}.
For this purpose, let us first spell out the construction of the Kodaira-Spencer map $\mathfrak{ks} : SC^* (P) \to CF((\mathbb{L},b),(\mathbb{L},b))$ for $P$. 
%(Of course, the construction can easily generalize to a more general situation as long as one can fulfill necessary technicalities.) 
As before, we first fix a Hamiltonian $H:P\to\mathbb{R}$ which is a generic $C^2$-small Hamiltonian in $P^{in}$, and is  linear  with a positive slope in the conical end that is smaller than any of positive periods of Reeb orbits occuring in the contact boundary $P$.
Near each orbit, we add a small compactly supported perturbation that breaks $S^1$-symmetry of Hamiltonian orbits.
%, a  perturbation  $H_{\mathrm{pert},\bL}$ that is compactly supported near a small neighborhood of $\bL$, on which it is equal to $-H|_{P^{in}}$.
%% in the complement of $(1-\epsilon, \infty)\times\Sigma$ where $\epsilon>0$ is so chosen that $\bL$ is properly contained in the support of $H_{\mathrm{pert},\bL}$.
%Later, this will force our Floer solutions (for closed-open maps) $J$-holomorphic near the output (from $\bL$). This is compatible with the fact that $CF(\bL,\bL)$ is taken to be a Morse model.

To incorporate the quantum cohomology into the symplectic cohomology (respectively, the compact Fukaya category into the wrapped Fukaya category), \cite[Section 7.1]{RS} uses the following strategy. Let us consider the enlarged chain complex
\begin{equation}\label{eqn:hybridsc}
SC_\Diamond^*(P)=QC^*(P)[\mathbf{q}] \oplus \bigoplus_{w\geq1}CF(wH)[\mathbf{q}]
\end{equation}
The first component $QC^*(P)$ of \eqref{eqn:hybridsc} is the quantum cochains on $P$ generated by (Poincar\'{e} duals of) locally finite cycles. In our situation, we choose its generators to be the unstable manifolds of a proper Morse function on $P$ so that $QC^*(P)$ is simply a Morse complex of this function. The additional piece of differential on $SC_\Diamond^*(P)$ is the PSS map, which identifies $\mathbf{q} QC^\ast (P)$ with contractible orbits in the second component of \eqref{eqn:hybridsc}. This can be thought of as an extension of the continuation map $CF(wH) \to CF((w+1)H)$ when $w=0$.
Therefore its cohomology still computes the same symplectic cohomology $SH^\ast (P)$.

Similarly, one can promote the wrapped Fukaya category $\mathcal{W} (P)$ to some hybrid-type category $\mathcal{W}_\Diamond(P)$ which shares the same objects with $\mathcal{W} (P)$, but when both of objects $L$ and $L'$ are compact, their morphism is enlarged to
\begin{equation}\label{eqn:hybridcf}
\hom_{\mathcal{W}_\Diamond(P)}(L,L') = CF(L,L')[\mathbf{q}] \oplus \bigoplus_{w\geq1}CW(L,L';wH)[\mathbf{q}].
\end{equation}
The first component $CF(L,L')$ of \eqref{eqn:hybridcf} uses a compactly supported Hamiltonian only, and for this reason, we will often call it the \emph{slope 0 component} of $\hom_{\mathcal{W}_\Diamond(P)}(L,L') $. 
The higher slope component $CW(L,L';wH)$ is generated by time-$1$ Hamiltonian chords for $wH$ from $L$ to $L'$ as usual. In our case of interest where $L=L'=\bL$, the slope $0$ component will be taken to be the Morse complex on $\bL$, i.e., $CF(\bL,\bL):=CM(h_\mathbb{L})\oplus \mathbb{C} \langle X,Y,Z,\bar{X},\bar{Y},\bar{Z} \rangle $ for a fixed generic perfect Morse function $h_\bL$ on $\bL$. To be more precise, one imposes Floer datum to each pair $(L,L')$ (as well as a perturbation datum to each domain of pseudo-holomorphic disks that agrees with the Floer datum for $(L,L')$ on each strip-like end), and the above simply means we use the trivial Hamiltonian perturbation datum for the pair $(\bL,\bL)$ when sitting in the slope-0 component of $\hom_{\mathcal{W}_\Diamond(P)}(\bL,\bL)$. We will explain how to address the related transversality issue in Remark \ref{rmk:transv}.

The $A_\infty$-structure on \eqref{eqn:hybridcf} is followed mostly from the standard construction in the telescope model \cite[Section 3]{AS_w}, which counts pseudo-holomorphic disks equipped with a geodesic (called a flavor) joining each input to output. Any geodesic can carry a preferred point that indicates the continuation from a lower slope to higher, and the corresponding input in this case is taken with the factor $\mathbf{q}$. 
In addition, one has a continuation map from $\mathbf{q} CF(L,L')$ into the second component of \eqref{eqn:hybridcf} as a part of the differential. We refer readers to \cite[Section 7.1]{RS} for more details. 

Since continuation maps are quasi-isomorphisms for $L$ and $L'$ compact, the inclusion $ CF(L,L') \to \hom_{\mathcal{W}_\Diamond(P)}(L,L')$
(with no higher maps) defines a quasi-isomorphism, and in particular, we have an $A_\infty$ quasi-isomorphism
$$ CF(\bL,\bL) \to \hom_{\mathcal{W}_\Diamond(P)}(\bL,\bL).$$
Since there are no higher components, it takes our weak bounding cochain $b$ on the left hand side to the same element now living in the slope 0 component on the right hand side. By abuse of notation, we still call this element $b$, and consider the boundary-deformed complex 
$$\hom_{\mathcal{W}_\Diamond(P)}((\bL,b),(\bL,b))$$
whose $m_k$ operations are $b$-deformed as in \eqref{eqn:bdeformedmk}. This is clearly quasi-isomorphic to $CF((\bL,b),(\bL,b))$.

Following \cite[Theorem 7.2]{RS}, we now have a ring homomorphism (after taking  cohomologies)
$$\tilde{\mathfrak{ks}}:SC_\Diamond^*(P) \to \hom_{\mathcal{W}_\Diamond(P)}((\bL,b),(\bL,b)).$$
This can be thought of as the restriction of the closed-open map $CO: SC_\Diamond^*(P) \to HC^\ast (\mathcal{W}_\Diamond(P))$, but uses a different algebraic formalism that is along the same line as \cite{FOOO10b}. Detailed description of associated moduli spaces and their contributions to $\tilde{\mathfrak{ks}}$ will be given shortly.
Composing quasi-isomorphisms, we obtain
$$ \mathfrak{ks}: SC^\ast (P) \to CF((\bL,b),(\bL,b)).$$
Namely, we have 
\begin{equation*}
\xymatrix{SC^\ast (P) \ar[r] \ar[d]_{\mathfrak{ks}}& SC_\Diamond^*(P) \ar[d]^{\tilde{\mathfrak{ks}}} \\
CF((\bL,b), (\bL,b)) \ar[r]& \hom_{\mathcal{W}_\Diamond(P)}((\bL,b),(\bL,b))}
\end{equation*}
where horizontal arrows are quasi-isomorphisms, and the bottom one can be inverted (up to homotopy). In general, a quasi-inverse is given as $\{g_k\}_{k \geq 0}$ possibly with higher components, but we only need $g_1$ since we do not care about the full $A_\infty$-structure, but the ring structures on cohomologies.

Let us give a concrete description on the pseudo-holomorphic disk counts involved in $\tilde{\mathfrak{ks}}$ and $\mathfrak{ks}$. Note that $\tilde{\mathfrak{ks}}$ takes its input from $CF(wH)[\mathbf{q}]$ for some $w \geq 0$ (see Figure \ref{fig:tildeks}). The output will belong to either the slope $w$ or $w+1$ component of $\hom_{\mathcal{W}_\Diamond(P)}((\bL,b),(\bL,b))$ depending on the existence of the continuation.
%
%========
%
%, a set of moduli spaces for defining differentials in this enlargement is described in detail. As the auxiliary summand corresponds to $w=0$ part and our Lagrangian $\bL$ lives in $P^{in}$, we find that the inclusion $CF(\bL,\bL) \hookrightarrow \hom_{\mathcal{W}_\Diamond(P)}(\bL,\bL)$ is a quasi-isomorphism. We invert this to find an $A_\infty$-equivalence $\{g_k\}:\hom_{\mathcal{W}_\Diamond(P)}(\bL,\bL) \to CF(\bL,\bL)$. In this way, the closed-open map sits inside the complex $CF(\bL,\bL)$.
%We then further deform the closed-open map to define the Kodaira-Spencer homomorphism.
For a fixed weight $w \in \mathbb{Z}_{\geq 0}$, we consider the following moduli space of punctured disks $u:D\to P$ consisting of the data:
\begin{enumerate}
\item[(a)] the domain $D$ is a disk with one negative boundary puncture, one positive interior puncture and any number of positive boundary punctures.
\item[(b)] the negative boundary puncture and the interior puncture carry an appropriate end parametrization; biholomorphisms to $(-\infty,0] \times [0,1]$ from a neighborhood of the negative boundary puncture and,  to $[0,\infty) \times S^1$ from a neighborhood of the interior puncture.
% and to $[0,\infty) \times [0,1]$ from a neighborhood of each positive boundary puncture. 
%\item the weight an integer $w\geq1$;
%\item[(c)] On the strip-like end near each positive boundary puncture, we impose an additional perturbation $H_{pert,\bL}$. When the output is taken from the slope $0$ component $CF((\bL,b),(\bL,b))$, we do the same on the strip-like end near the negative boundary puncture. By abuse of notation, we will write $H$ for this additionally perturbed Hamiltonian.
\item[(c)] a fixed choice of a geodesic $\gamma$ from the interior punctures to the negative boundary puncture (see Lemma \ref{lem:equatorsh} and subsequent discussions) where  $\gamma$  may or may not carry a preferred point.
\item[(d)] if $\gamma$ carries a preferred point, set $\beta$ to be a subclosed 1-form that is equal to $0$ near the interior puncture and  positive boundary punctures, but equal to $dt$ near the negative boundary puncture where $t$ is a coordinate along $[0,1]$ on $(-\infty,0]\times[0,1]$. Otherwise, set $\beta \equiv 0$. 
\item[(e)] a closed 1-form $\delta$ on $D$ that is equal to $dt$ near the interior puncture and near the negative boundary puncture.
\item[(f)] $u$ satisfies the perturbed Cauchy-Riemann equation $(du - X_{H} \otimes (w \delta+ \beta ))^{0,1} =0$, and sends boundary to $\bL$.
\item[(g)] on each positive boundary punctures, $u$ asymptotes to one of $X$, $Y$ or $Z$ (taken from the slope 0 component $CF(\bL,\bL)$).
%as $s \to \infty$ where $s$ is a coordinate for $[0,\infty)$ on $[0,\infty)\times[0,1]$. 
This is consistent with the fact that these generators are taken from the slope 0 component $CF(\bL,\bL)$ for which we have $wH=0$ and hence, Floer solutions are $J$-holomorphic near these inputs. 
\item[(h)] $u$ asymptotes to a Hamiltonian orbit $\alpha$ of $wH$ at the interior punctures. When $w=0$, we impose the intersection condition at the interior puncture with the locally finite cycle in $QC^\ast(P)$.
\item[(i)] For $w > 0$, $u$ asymptotes to a Hamiltonian chord of $wH$ at the negative boundary punctures when $\gamma$ carries a preferred point, to that of $(w+1)H$ when $\gamma$ does not. If $w=0$, we require $u$ at the negative puncture to pass through one of unstable manifolds of $h_\bL$. In this case, the corresponding Morse trajectory of $h_\bL$ can limit to the unique critical point of degree 1 (that is, $f_\bL$), or can flow to a constant disk at one of $X,Y$ and $Z$.
%(This is again consistent with the fact about the slope 0 component $CF(\bL,\bL)$ mentioned in (c).)
\end{enumerate}
%These data makes us to define a moduli space $\mathcal{M}^{k}_{CO}(\alpha,Z)$ for $\alpha$ a generator of $SC^*(P)$ and $Z$ a (immersed) Floer generator of $\bL$. 

\begin{figure}[h]
\includegraphics[scale=0.3]{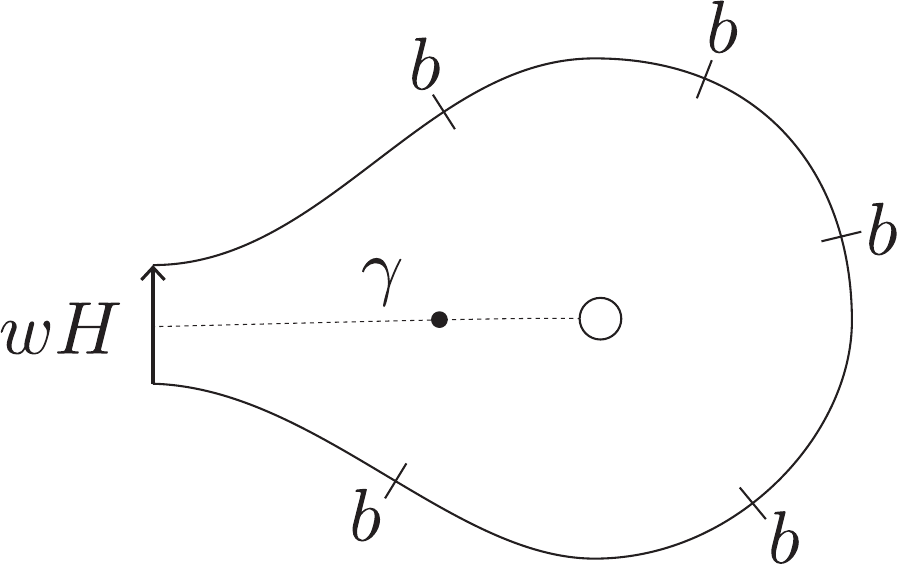}
\caption{a punctured disk counted for $\tilde{\mathfrak{ks}}$}
\label{fig:tildeks}
\end{figure}

For an orbit $\alpha$ in $CF(wH)[\mathbf{q}]$, $\tilde{\mathfrak{ks}}(\alpha)$ is precisely the counting of the isolated element in the above moduli for which the geodesic $\gamma$ does not carry a preferred point. Suppose $w >0$, and there exists a punctured disk in the above moduli which asymptotes to $\alpha$ at the interior puncture, to a chord $Z$ (for $wH$) at the negative boundary puncture, and passes through $X_{i_1},\cdots, X_{i_d}$ at positive boundary punctures in clockwise order (where $X_{i_l} \in \{X,Y,Z\}$). Then it contributes $x_{i_1} x_{i_2} \cdots x_{i_d} Z$ to $\tilde{\mathfrak{ks}}(\alpha)$, and in general $\tilde{\mathfrak{ks}}(\alpha)$ is a linear combination of such.
In particular, the output belongs to the slope $w$ component of $\hom_{\mathcal{W}_\Diamond(P)}((\bL,b),(\bL,b))$). 

 If $w=0$ for the output and the negative boundary puncture is incident to the unstable manifold of $h_\bL$, we do the same with the only difference that the output is either the critical point or taken from the pearl attached to the corresponding trajectory of $h_\bL$ (see (i)). When $w=0$ happens at the interior puncture, we take the input from locally finite cycles in $P$ (see (h)).
For $\mathbf{q} \alpha$, one makes an analogous count for $\tilde{\mathfrak{ks}}( \mathbf{q}  \alpha)$, but with a preferred point on the geodesics (and hence the Floer equation involves nontrivial $\beta$), and the output lives in the slope $(w+1)$ component.
Finally, One can uniquely extend this to the entire complex by requiring $\partial_\mathbf{q} $-equivariance.

\begin{remark}\label{rmk:transv}
Notice that the above moduli is an adaptation of the one in \cite{RS} with the only difference being the usage of the Morse chain model for $CF(\bL,\bL)$ sitting at the slope $0$ component. Observe that the immersed generators are already given as transversal intersections (of two local branches of $\bL$) which do not require a further perturbation (see {\rm(g)} above). Therefore in order to guarantee the transversality of the moduli (for generic perturbation data), it suffices to additionally consider the transversality between the evaluation map from the boundary of a punctured disk and stable/unstable manifolds of $h_\bL$. This can be easily achieved by choosing generic perfect Morse function $h_\bL$ since in our situation $\dim \bL =1$, and hencestable/unstable manifolds are either a point or the entire $\bL$.
\end{remark}

\begin{remark}
Since $b$ is taken from the slope $0$ component for which the Morse model is employed, a priori one may need to also include more complicated pearls (other than those in (i)) that are disks with boundaries on $\bL$ and arbitrary number of $b$ inputs. These can be attached to the punctured disk through gradient trajectories of $h_\bL$. While it is not difficult to enlarge the definition of the above moduli to include pearls as transversal intersection between boundaries of pearls and unstable/stable manifold of $h_\bL$ can be easily achieved for $dim \bL=1$ as mentioned in Remark \ref{rmk:transv}, we do not necessarily need to consider these for $\mathfrak{ks}$ due to the following reason. If the pearl component is nonconstant, the gradient trajectory of $h_\bL$ can be attached to any point in the boundary, so the associated moduli cannot be isolated. On the other hand if it is constant, it must lie at one of $X$, $Y$ or $Z$, and hence involve inputs other than $b$ (such as $\bar{X}, \bar{Y}$ or $\bar{Z}$), since the gradient flow of $h_\bL$ should start away from the branch jump.
\end{remark}

\begin{figure}[h]
\includegraphics[scale=0.55]{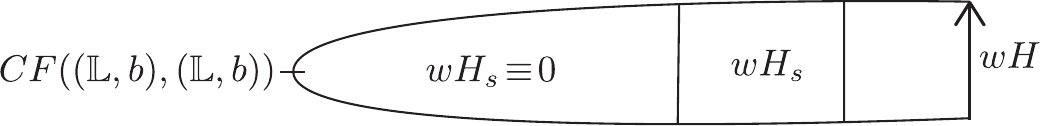}
\caption{The continuation map from nontrivial slope to $CF((\bL,b),(\bL,b))$; the boundary can have arbitrary number of $b$-insertions which are omitted in the picture.}
\label{fig:contiback}
\end{figure}

In order to obtain $\mathfrak{ks}$ from  $\tilde{\mathfrak{ks}}$, one may take any quasi-inverse of the inclusion 
$$CF((\bL,b), (\bL,b)) \to \hom_{\mathcal{W}_\Diamond(P)}((\bL,b),(\bL,b)),$$
which will produce the same map after taking cohomologies. In practice, one can also calculate the image $\mathfrak{ks} (\alpha)$ for an orbit $\alpha \in CF(wH)$ in the following more geometric way. By construction, $\tilde{\mathfrak{ks}} (\alpha)$ lies in the slope $w$ component of $\hom_{\mathcal{W}_\Diamond(P)}((\bL,b),(\bL,b))$, which is quasi-isomorphic to the slope 0 component $CF((\bL,b),(\bL,b)$ through a canonically given (up to homotopy) continuation map as $\bL$ is compact. As we are using the Morse-Bott type definition for $CF((\bL,b),(\bL,b)$, the continuation map should interpolate $wH$ and some Hamiltonian which vanishes near $\bL$. For this,  we introduce the compactly supported perturbation term $H_{pert,\bL}$  that is compactly supported near a small neighborhood of $\bL$, on which it is equal to $-H|_{P^{in}}$. With this, one can define the continuation map from this slope $w$ component to the slope $0$ component $CF((\bL,b),(\bL,b))$ by choosing an interpolation to be $w H_s:=w H + s w  H_{pert,\bL}$ ($-1 \leq s \leq 0$) and imposing it in on some fixed compact region, say $[-1,0] \times [0,1]$, in the Floer strip $(-\infty,\infty) \times [0,1]$ (note that the isotopy occurs on a compact subset of $P$ only). See Figure \ref{fig:contiback} for the illustration of such strips where in the picture we omit $b$-boundary insertions. 

\begin{figure}[h]
\includegraphics[scale=0.3]{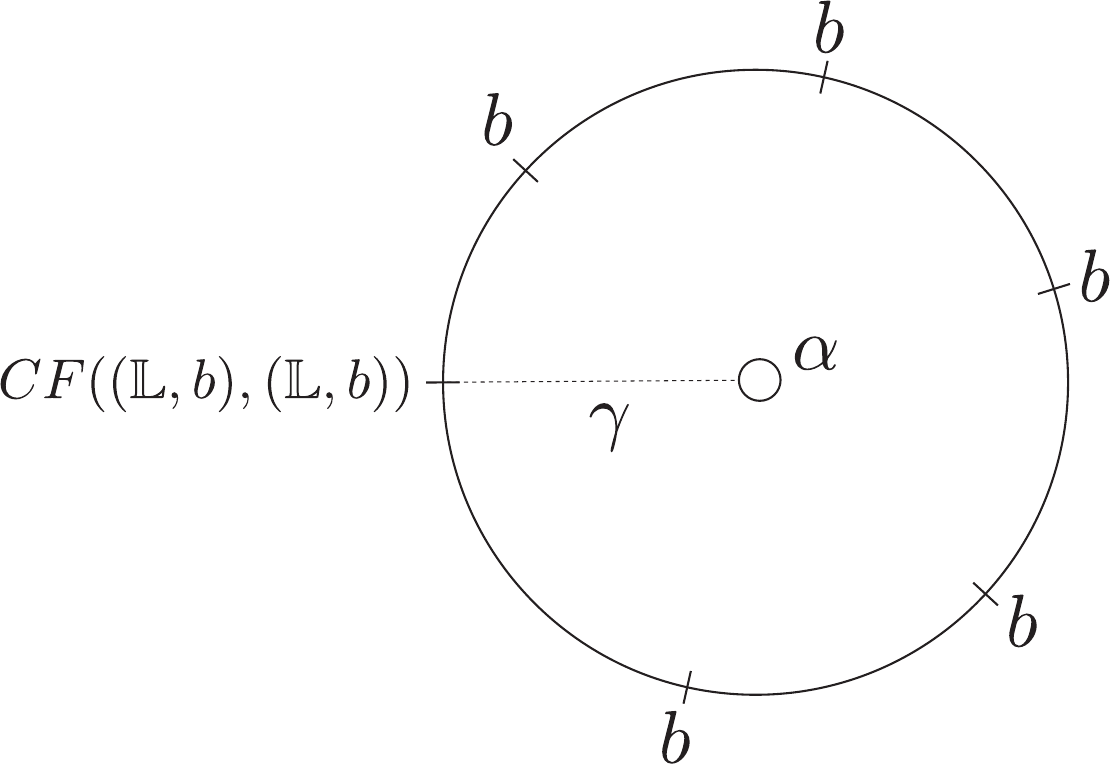}
\caption{The image $\mathfrak{ks} (\alpha) \in CF((\bL,b),(\bL,b))$ of a weight $w$ orbit $\alpha$}
\label{fig:aftergl}
\end{figure}

The standard gluing argument applying to the puncture disks (Figure \ref{fig:tildeks}) and these continuation strips (Figure \ref{fig:contiback}) tells us that in order to determine the image $\mathfrak{ks} (\alpha)$, we can instead count the same punctured disks for $\tilde{\mathfrak{ks}}$ except that the Hamiltonian $wH$ is replaced by $w(H+H_{pert,\bL})$, and hence the negative boundary puncture (the output) is incident to one of generators in $CF(\bL,\bL)= CM(h_\bL) \oplus \mathbb{C} \langle X,Y,Z, \bar{X},\bar{Y},\bar{Z} \rangle$. See Figure \ref{fig:aftergl}. We do not put a preferred point on $\gamma$ in this case, since $\alpha$ is an honest Hamiltonian orbit (without $\mathbf{q}$ attached) in $SC^\ast (P)$. We denote moduli of such punctured disks with the output $Z \in CF((\bL,b),(\bL,b))$ by $\mathcal{M}_{CO} (\alpha, Z)$.

Suppose we are given an abelian covering $X \to P$. Following \ref{subsec:Thetagraph}, we can choose a 1-dimensional submanifold $\Theta$ of $P$ away from which the covering $X \to P$ trivializes. 
(More generally, $\Theta$ can be chosen to be a graph on $P$ as in Remark \ref{rmk:Thetagraph1}.) Recall that each component (or each edge when $\Theta$ is a graph) is assigned with an element of $G$. Using these data, we have constructed $SC_\HG^* (P)$ and $CF_{\hat{G}} ((\bL,b),(\bL,b))$ in \ref{subsec:constequivsc} and \ref{subsec:hatgequivlft} which admit natural $\hat{G}$-actions with associated quotients recovering the corresponding invariants upstairs in $X$.

Having the homomorphism $\mathfrak{ks}:SC^*(P)\to CF((\bL,b),(\bL,b))$,  we want to twist this analogously to operations on  $SC_\HG^* (P)$ and $CF_{\hat{G}} ((\bL,b),(\bL,b))$ (see \eqref{eqn:equivdiff}, \eqref{eqn:equivprod}, \eqref{eqn:mkequiv}) so that it promotes to
\[\mathfrak{ks}_{\hat{G}}:SC_\HG^* (P)\to CF_{\hat{G}} ((\bL,b),(\bL,b)).\]
This can be simply done by assigning the weight $g(\gamma_u)$ to each $u:D \to P$ in the moduli space where $\gamma_u$ is the image of $\gamma$ in $D$ under $u$. Recall that generically $\gamma_u$ intersects finitely many components (or edges) of $\Theta$ that determines the element $g_\gamma$ as in Definition \ref{def:ggamma1}. As before, $g(\gamma_u)$ only depends on the homotopy class of $u$. For an Hamiltonian orbit $\alpha$, we define
\begin{equation}\label{eqn:equiksdef}
\mathfrak{ks}_{\hat{G}} (\alpha \otimes \chi) := \sum_{u \in \mathcal{M}_{CO} (\alpha, Z)}  \chi (g(\gamma_u)) \, Z \otimes \chi.
\end{equation}

%
%\begin{defn}
%We define
%\[\mathfrak{ks}:SC^\ast(\Sigma)\to CF((\bL,b),(\bL,b))\]
%by counting the moduli space $\mathcal{R}^1(a)$ of once punctured pseudo-holomorphic discs.
%\end{defn}

\begin{lemma}\label{lem:ksalgebraic}
The map $\mathfrak{ks}_{\hat{G}}$ is a $G \times \hat{G}$-equivariant chain map that preserves the products, inducing a ring homomorphism $KS_{\HG}:SH_{\HG}^*(P) \to HF_{\HG}((\bL,b),(\bL,b))$.
\end{lemma}

\begin{proof}
%The cobordism argument in \cite{RS} still works with additional consideration on the breaking of the Morse flow line which is already included as a part of $m_1$ on $CF((\bL,b),(\bL,b))$. Hence, 
The proof is the same as that for $\mathfrak{ks}$ (or $\tilde{\mathfrak{ks}}$), which uses standard  cobordism argument as in \cite{RS}. However, we additionally need to check if $\mathfrak{ks}_{\hat{G}}$ is compatible with the twisted structure coming from the paths $\gamma$ from inputs to the output. This is purely a topological aspect, keeping track of (the homotopy classes) these paths under the degeneration. For example, to see $\mathfrak{ks}_{\hat{G}}$ is a chain map, we inspect the boundary behavior of elements in the 1-dimensional moduli $\mathcal{M}_{CO} (\alpha, Z)$. Possible degenerations are described in Figure \ref{fig:ksgbreak}.

\begin{figure}[h]
\includegraphics[scale=0.6]{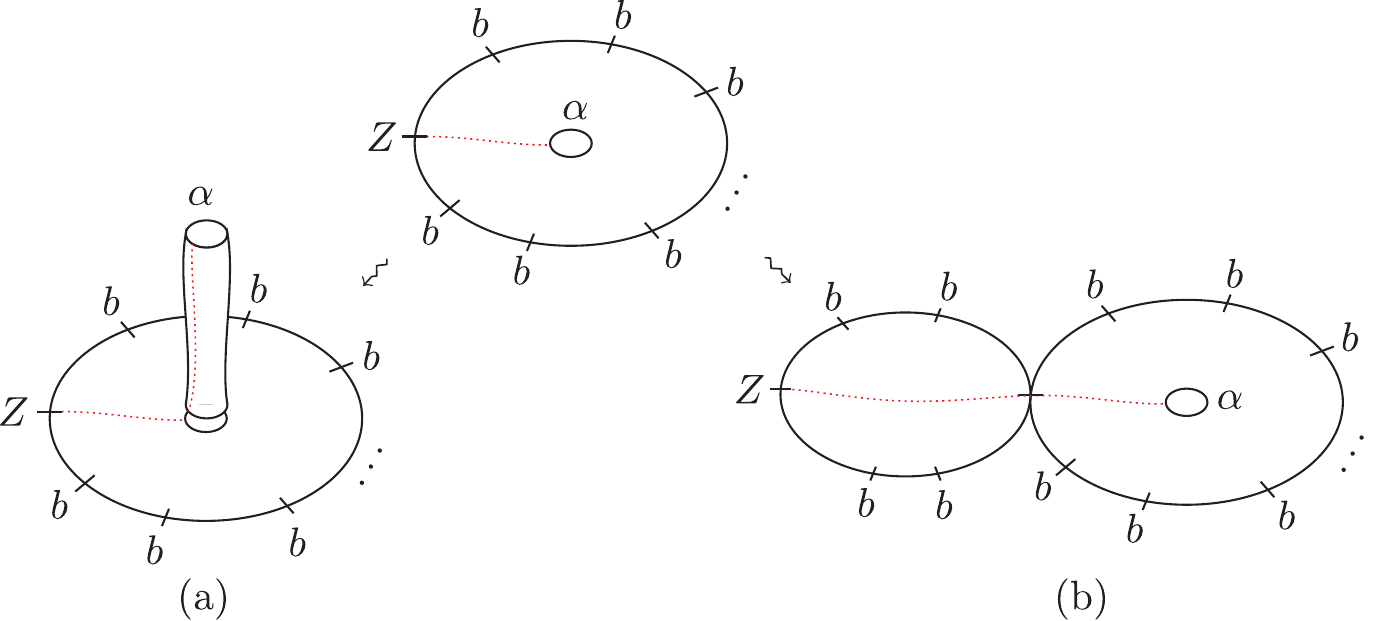}
\caption{Cobordism between $ d_{\hat{G}}  \circ \frak{ks}_{\hat{G}} $ and $\frak{ks}_{\hat{G}} \circ d_{\hat{G}} $ (when counted with weights coming from intersection between dotted paths and $\Theta$)}
\label{fig:ksgbreak}
\end{figure}

As in the proof of Lemma \ref{lem:equatorsh}, one can make a coherent choice of $\gamma_u$ so that it is consistent with other choices of paths $\gamma$ from inputs to output for $SC_\HG^* (P)$ and $CF_{\hat{G}} ((\bL,b),(\bL,b))$. Therefore the count of configurations in (a) of Figure \ref{fig:ksgbreak} weighted by $\xi$ gives 
\begin{equation}\label{eqn:ksd}
\frak{ks}_{\hat{G}} \circ d_{\hat{G}} (\alpha \otimes \chi)
\end{equation}
whereas (b) of Figure \ref{fig:ksgbreak} corresponds to
\begin{equation}\label{eqn:dks}
m_{1,\hat{G}}^{b\otimes1,b\otimes1}  \circ \frak{ks}_{\hat{G}} (\alpha \otimes \chi).
\end{equation}
Notice that if a Floer cylinder $u$ attached with $v_1 \in \mathcal{M}_{CO} (\alpha', Z)$ is cobordant to $v_2  \in \mathcal{M}_{CO} (\alpha, Z)$ with the bubbled-off disk $w$ through the cobordism in Figure \ref{fig:ksgbreak}, then the concatenation $\gamma_u \sharp \gamma_{v_1}$ is homotopic to $\gamma_{v_2} \sharp \gamma_{w}$. Therefore the weights $\xi(g_u)$ in the counts \eqref{eqn:ksd} and \eqref{eqn:dks} must agree.
In general, the $1$-dimensional moduli can have more boundary components due to disk bubbles which do not contribute in our case due to weakly unobstructedness. Similar topological argument proves that $\mathfrak{ks}_{\hat{G}}$ preserves the product structure. 

Recall that the $\hat{G}$ actions on both $SC_\HG^* (P)$ and $CF_{\hat{G}} ((\bL,b),(\bL,b))$ are determined by the liftings of loops associated with generators.
Then the $\hat{G}$-equvariance of $\mathfrak{ks}_{\hat{G}}$ is an easy consequence of the fact that for any $u \in \mathcal{M}_{CO} (\alpha, Z)$, the image (limit) of the interior puncture is homotopic to the image of the boundary. Obviously, $\mathfrak{ks}_{\hat{G}}$ is also $G$-equivariant with respect the $G$-action $ g( \alpha \otimes \chi) = \chi (g) \alpha \otimes \chi$ and $ g( Z \otimes \chi) = \chi (g) Z \otimes \chi$.
\end{proof}

Combined with Proposition \ref{prop:isomupdownsc} and \ref{prop:lagcfupdown}, we obtain a natural map between $SH^*(X)$ and its closed-string $B$-invariant of the orbifold LG mirror, that is, the Koszul cohomology of $W$ relative to $\hat{G}$-action.

 \begin{cor} 
Let $X$ be an abelian cover of $P$, and $\tilde{\bL}$ the inverse image of the Seidel Lagrangian in $X$. Then there exists a natural ring homomorphism $SH^* (X) \to HF((\tilde{\bL},\tilde{b}),(\tilde{\bL},\tilde{b}))$ induced by $KS_{\HG}$.
\end{cor}

The rest of the paper is devoted to show that they are actually isomorphic. Recall that the right hand side can be identified with the orbifold $B$-model invariant of $(W,\hat{G})$ introduced in \ref{subsec:orbLGB} (Lemma \ref{lem:downtoup} and Proposition \ref{prop:floerkos}).

\subsection{Pair-of-pants and its mirror potential $W(x,y,z) = xyz$}\label{subsec:popms}
Let us first look into the pair-of-pants $P$ more closely, which will serve as a building block for subsequent cases. Recall from \ref{subsec:SHP} that the symplectic cohomology ring $SH^* (P)$ admits the following representation
$$ SH^* (P) = \Bbbk \langle e \rangle \oplus \Bbbk \langle f_1 \rangle \oplus \Bbbk \langle f_2 \rangle \oplus  \Bbbk\langle e_\alpha,f_\alpha \rangle [t_\alpha] \oplus \Bbbk \langle e_\beta,f_\beta \rangle [t_\beta] \oplus \Bbbk \langle e_\gamma,f_\gamma \rangle [t_\gamma]$$
with the product structure determined by
$$ f_1 \star e_\alpha= f_\alpha, \quad  e_\alpha \star e_\alpha =e_\alpha, \quad e_\alpha \star f_\alpha = f_\alpha $$
$$ f_1 \star e_\beta= f_2 \star e_\beta = f_\beta, \quad  e_\beta \star e_\beta =e_\beta, \quad e_\beta \star f_\beta = f_\beta $$
$$ f_2 \star e_\gamma= f_\gamma, \quad  e_\gamma \star e_\gamma =e_\gamma, \quad e_\gamma \star f_\gamma = f_\gamma $$
and the products between different $t_a$'s vanish. Geometrically, $e$(serving as the unit), $f_i$ are Morse generators appearing in $P^{in}$ and $e_\alpha t_\alpha^l$ and $f_\alpha t_\alpha^l$ are Hamiltonian oribits around one of the punctures with winding number $l$ (and similar for $\beta$ and $\gamma$).

Let us now consider the LG mirror of $P$.
As before, we take $\bL$ to be the Seidel Lagrangian, whose Maurer-Cartan deformation induces the mirror LG-model $W : \Bbbk^3 \to \Bbbk, W(x,y,z) = xyz$. The coefficient ring $\Bbbk$ will be $\mathbb{C}$ in the mirror symmetry consideration.
%Recall from Lemma \ref{lem:seidelLagoverc} that one may take $\Bbbk = \mathbb{C}$
We take the Koszul complex $(K^* (\partial W), d_W)$ defined in \ref{subsec:orbLGB} (with trivial $H$) as the associated closed-string B-model, in this case, consisting of the following data:
\begin{align*}
K^* (\partial W)&=S_R^\bullet\left(\bigoplus_{t\in\{x,y,z\}}R\cdot\theta_t\right),\\
d_W&= yz\partial_{\theta_x}+zx\partial_{\theta_y}+xy\partial_{\theta_z}.
\end{align*}
where $R=\Bbbk[x,y,z]$ is the function ring of $\Bbbk^3$ and $\theta_a$'s for $a \in\{x,y,z\}$ are formal variables of degree $-1$ as before. 
%$M:=\bigoplus_{t\in\{x,y,z\}}R\cdot\theta_a$ is then a graded $R$-module supported at degree $-1$ so that the symmetric tensor algebra $S_R^\bullet(M)$ is supported on degree $\leq0$. $\partial_a$ for $a \in\{x,y,z\}$ the obvious $R$-derivation of degree $1$ that kills $\theta_a$.
(The exterior degree here is for the convenience of exposition, and we will only use its reduced $\mathbb{Z}/2$-grading, later.)
%
%Because of graded commutativity, we have $\theta_a^2=0$ in $K^\ast (\partial W)$, and hence degree $-2$ component is freely generated by $\theta_x \theta_y, \theta_y \theta_z, \theta_z \theta_x$. 

The degree $-3$ component is $R\cdot\theta_x\theta_y\theta_z$, and it injects into the $-2$-degree component under $d_W$ since $R$ is an integral domain. Therefore its cohomology at degree $-3$ vanishes.
For the cohomology at degree $-2$, suppose $\omega=f_x\theta_y\theta_z+f_y\theta_z\theta_x+f_z\theta_x\theta_y$ is a cocycle, i.e., 
\[d_W(\omega)=x(yf_y-zf_z)\theta_x+y(zf_z-xf_x)\theta_y+z(xf_x-yf_y)\theta_z=0.\]
It leads to $xf_x=yf_y=zf_z$ which we denote by $f$. Since $xyz$ divides $f$, we may write $f=xyz g$.
Then we have $f_x=yz g$, $f_y=zx g$, and $f_z=xy g$ so that 
\[\omega = g(yz\theta_y\theta_z+zx\theta_z\theta_x+xy\theta_x\theta_y)=d_W(g\theta_x\theta_y\theta_z).\]
Thus the $-2$nd cohomology is also zero.

In order to compute the cohomology in degree $-1$, it is elementary to check that $\lambda = g_x\theta_x+g_y\theta_y+g_z\theta_z$ is cocycle if and only if it is a $R$-linear combination of  $\lambda_x = y\theta_y-z\theta_z$, $\lambda_y = z\theta_z-x\theta_x,$ and $\lambda_z = x\theta_x-y\theta_y$. These generators satisfy the single linear relation $\lambda_x+\lambda_y+\lambda_z=0$ (on the chain level). 
%Since $K^{\ast} (\partial W)$ in degree -2 is a free $R$-module generated by $\theta_y\theta_z$, $\theta_z\theta_x$, and $\theta_x\theta_y$, 
On the other hand, the set of coboundaries is spanned by $x\lambda_x$, $y\lambda_y$, and $z\lambda_z$. Therefore the $(-1)$-cohomology is given as
$$H^{-1}(K^* (\partial W))=\dfrac{\{h_x\lambda_x+h_y\lambda_y+h_z\lambda_z : h_x,\,h_y,\,h_z\in R\} }{ \langle x\lambda_x=y\lambda_y=z\lambda_z=0\rangle} = \dfrac{\{h_x\lambda_x+h_y\lambda_y : h_x, h_y \in R \} }{ \langle x\lambda_x=y\lambda_y =0 \rangle}$$
where we used $\lambda_x + \lambda_y + \lambda_z =0$ at the end. 

On the cohomology, we have the following induced relations
\[
\begin{array}{c}
yz\lambda_x=0,\quad zx\lambda_y=0,\quad xy\lambda_z=0,\\[2pt]
x^i\lambda_y=-x^i\lambda_z,\quad y^i\lambda_z=-y^i\lambda_x\quad z^i\lambda_x=-z^i\lambda_y
\end{array}
\]
for $i\geq1$ that completely determine the $R$-module structure on $H^{-1}(K^*(\partial W))$, which descends to $R/\langle yz, zx, xy \rangle$-module structure. We also have the relations
\begin{equation}
\begin{array}{c}
\lambda_x^2=\lambda_y^2=\lambda_z^2=0, \quad
\lambda_y\lambda_z=\lambda_z\lambda_x=\lambda_x\lambda_y=d_W(\theta_x\theta_y\theta_z).
\end{array}
\end{equation}

Finally, because $K^{-1} (\partial W)=\bigoplus_{a\in\{x,y,z\}}R\cdot\theta_a$ whereas $K^0(\partial W)$ is simply $R$, we find that $H^0 (K^* (\partial W)) = R / \langle yz, zx, xy\rangle$.
Consequently, 
\begin{eqnarray*}
Kos(W) = H^* (K^*(\partial W)) & =&H^{0}(K^* (\partial W))\oplus H^{-1}(K^*(\partial W)) \\
& \cong&  R / \langle yz, zx, xy\rangle \oplus \dfrac{\{h_x\lambda_x+h_y\lambda_y : h_x, h_y \in R \} }{ \langle x\lambda_x=y\lambda_y =0 \rangle},
\end{eqnarray*}
and it has a structure of $H^{0}(K^* (\partial W))$-module which is generated by $1_R \in H^0 (K^* (\partial W))$ and by $\lambda_x, \lambda_z \in H^{-1} (K^* (\partial W))$ subject to the algebra relations
\[
\begin{array}{c}
x \lambda_x=z \lambda_z=0, \quad
y \lambda_x=-y \lambda_z, \quad
\lambda_x^2=\lambda_z^2 = \lambda_x\lambda_z = \lambda_z \lambda_x = 0.
\end{array}
\]
%for all $i\geq 1$.
Under the identification $\Theta$ in Proposition \ref{prop:CFKOS}, the above generators correspond to 
$$ 1_R,\quad  \lambda_x, \quad  \lambda_z \qquad \longleftrightarrow  \qquad e_\bL, \quad yY - zZ, \quad xX - yY$$
in $CF((\bL,b),(\bL,b))$.

From the above explicit descriptions, we see that there is an obvious ring isomorphism between $SH^*(P)$ and $Kos(W)$,
%$\simeq CF((\bL,b),(\bL,b))$
 which matches generators of $SH^*(P)$ and those of $Kos(W)=H^{*}(K^* (\partial W))$ via
\[
\begin{array}{c}
e\mapsto 1_R,\quad e_\alpha t_\alpha^i \mapsto x^i,\quad e_\beta t_\beta^i \mapsto y^i,\quad e_\gamma t_\gamma^i\mapsto z^i,\\[3pt]
f_1\mapsto \lambda_z,\quad f_2\mapsto \lambda_x
\end{array}
\]
for all $i\geq1$. Observe that we have chosen generators on both sides in such a way that relations among them precisely match.
The only remaining question is now if the Kodaira-Spencer map $\mathfrak{ks}:SC^*(P) \to CF((\bL,b),(\bL,b))$ composed with the qausi-isomorphism $CF((\bL,b),(\bL,b)) \to K^* (\partial W)$ geometrically realizes this. 
We have $\mathfrak{ks}(e) = e_\bL$ since the closed-open map is unital, and $e_\bL$ again maps to the unit in $K^* (\partial W)$.
Observing that $SH^*(P)$ is (multiplicatively) generated by $e_\alpha t_\alpha$, $e_\beta t_\beta$, $e_\gamma t_\gamma$, $f_1$ and $f_2$, it now suffices to determine their images under $\mathfrak{ks}$.

\begin{lemma}\label{lem:KSpoppop}
 The Kodaira-Spencer map $\mathfrak{ks} : SC^*(P) \to CF((\bL,b),(\bL,b))$ computes
\[
\begin{array}{c}
\mathfrak{ks} (e_\alpha t_\alpha)= x \,e_\bL, \quad \mathfrak{ks} (e_\beta t_\beta)= y \,e_\bL, \quad \mathfrak{ks} (e_\gamma t_\gamma)= z \,e_\bL\\[3pt]
\mathfrak{ks} (f_1)= -xX + yY, \quad \mathfrak{ks} (f_2)= - yY + zZ.
\end{array}
\]
%Therefore, {\color{red}by ,} $\mathfrak{ks}$ is a quasi-isomorphism of differential graded algebras.
\end{lemma}

\begin{proof}
By obvious symmetry, it suffices to deal with the case for $\alpha$.
We first show $\mathfrak{ks} (e_\alpha t_\alpha)= x e_\bL$. Note that the geometric representative (the Hamiltonian orbit) for $e_\alpha t_\alpha$ is topologically a simple closed curve round the puncture $\alpha$. It is easy to find a pair of topological (punctured) disks $u$ and $u'$ whose punctured boundary in the interior wrap $\alpha$ exactly once and whose boundary circles hit the corner $X$. (Figure \ref{fig:cancelpairks})

\begin{figure}[h]
\includegraphics[scale=0.5]{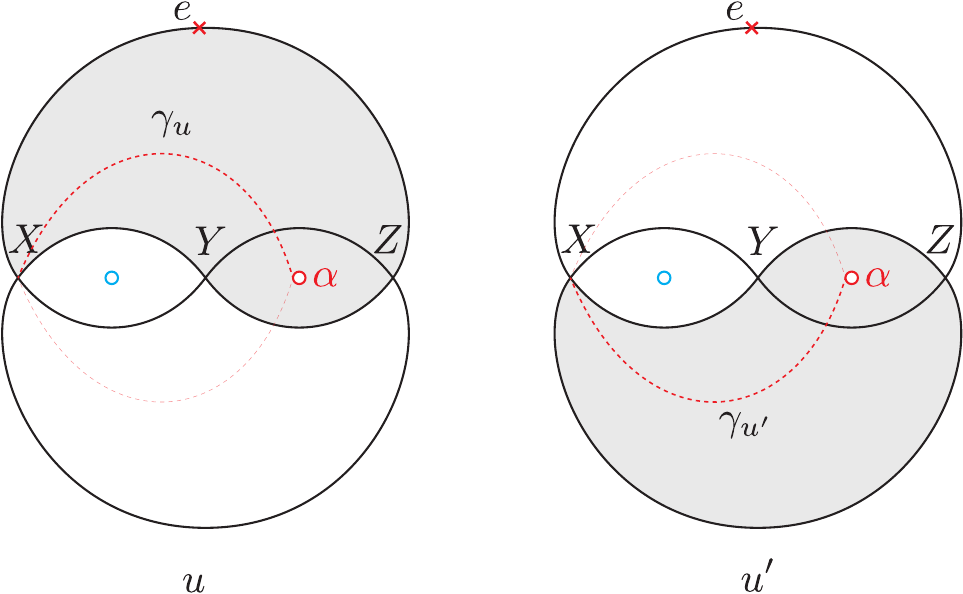}
\caption{Topological punctured disks with boundaries on $\bL$ and asymptotic to $e_\alpha t_\alpha$ at the puncture.}
\label{fig:cancelpairks}
\end{figure}

In order to see that these regions $u$ and $u'$ are indeed parametrized as solutions to our perturbed Cauchy-Riemann equation, we use the domain-stretching argument in \cite[5.2]{Ton}. Specifically, in our setting, the domain-stretching technique is applied to $(\mathbb{C}P^1,D=p_1+p_2+p_3)$ where the points $p_1,p_2,p_3$ are such that $P=\mathbb{C}P^1\setminus\{p_1,p_2,p_3\}$. If $p_1$ denotes the point that fills up the puncture $\alpha$, there are two $J$-holomorphic disks $\tilde{u},\tilde{u}':(\mathbb{D}^2,\partial\mathbb{D}^2)\to(\mathbb{C}P^1,L)$ of Maslov index 2 intersecting with $D$ only once at $p_1$ each of whose corner is incident to $X$. Topologically, $\tilde{u}$ and $\tilde{u}'$ compactify $u$ and $u'$ respectively. 

Let us now consider a sequence of domain disks
$$D_n = A_n \cup B, \quad n \in \mathbb{Z}_{>0}$$
where $A_n = S^1 \times [0,n]$ and $B$ is a unit disk capping the boundary component $S^1 \times \{n\}$ of $A_n$. As $n$ increases, $D_n$ stretches the domain $D$ (the unit disk) of $\tilde{u}$. One imposes the $s$-shaped Hamiltonian $H$ on $D_n$ which is linear with a fixed slope $w(\geq 1)$ near $S^1 \times \{n\}$ and vanishes near $S^1 \times \{0\}$ and $0 \in B$. By definition (see \cite[4.6]{Ton}), the $s$-shape $H$ becomes constant again for large $R$ in contrast to the Hamiltonian in our setup which is linear in $R$ with a nontrivial slope for all $R \gg 1$. Therefore the $s$-shape $H$ can have additional nonconstant orbits on the region where $H"(R) <0$, which are called the type III orbit in \cite{Ton}. Nonconstant orbits occurring on $H"(R)>0$ are called type II, and these are precisely the orbits considered in our geometric setup.

For each finite $n$, consider the solutions to the Floer equation with the above choice of a $s$-shaped Hamiltonian that pass through $X$ and through $p_1$ once at the interior. The count of such solutions remain unchanged as $n$ varies since one can homotope $H$ back to $0$ (note  that we work on $\mathbb{C}P^1$ which is compact). 
On the other hand, as $n \to \infty$, $D_n$ breaks into two pieces,
\begin{itemize}
\item[-] (a)-curve: a punctured disk with boundary condition identical to that of $\tilde{u}$ and the asymptotic condition to an Hamiltonian orbit $\gamma$ at the puncture
\item[-] (b)-curve: a disk with boundary equal to $\gamma$ and passing through $p_1$ once in the interior.
\end{itemize}
See the left of Figure \ref{fig:neckstr}.
We want to prove that $\gamma$ is the Hamiltonian orbit $e_\alpha t_\alpha$ that winds around $p_1$ exactly once. According to the discussion above, it suffices to show that $\gamma$ is of type II (since, once this is the case, its winding number is determined by a topological consideration).

Tonkonog classified the possible types of the orbit $\gamma$ into I, II, III, IVa, IVb. Orbits other than type II and III are constants, and if they existed, we would have a sphere satisfying the (perturbed) Floer equation in $\mathbb{C}P^1$ and not intersecting $p_2, p_3$ which is a contradiction. Finally, type III orbits can be excluded using some generalized version of maximum principle \cite[5.7]{Ton} and the fact that type III orbits have the opposite orientations compared to type II. Therefore we are only left with the type II orbit appearing as $\gamma$ which is  desired. Since the count of (b)-curves is $1$ in this case, we conclude that the number of punctured disks of topological type $u$ is precisely equals to the number of Maslov $2$ disk passing through $p_1$.

 We learned the above argument from \cite[Theorem 5.4]{CCJ}.
Alternatively, one can also apply more traditional neck-stretching argument (that degenerates the almost complex structure near some suitably chosen circle $C$ enclosing $\alpha$) to split the disk $u$ into $u_1$ and $u_2$, and compactify each using the finiteness of the energy (see the right of Figure \ref{fig:neckstr}). $u_1$ becomes a Maslov $2$ disk bounding $\bL$ in the compactification of $P$ whereas  $u_2$ is a pseudo-holomorphic cylinder appearing in the calculation of $d$ on $SC^*(\mathbb{C})$ (the degree of $\alpha$ as an orbit in $\mathbb{C}$ is odd). It is well-known that their counts are both $1$.

\begin{figure}[h]
\includegraphics[scale=0.5]{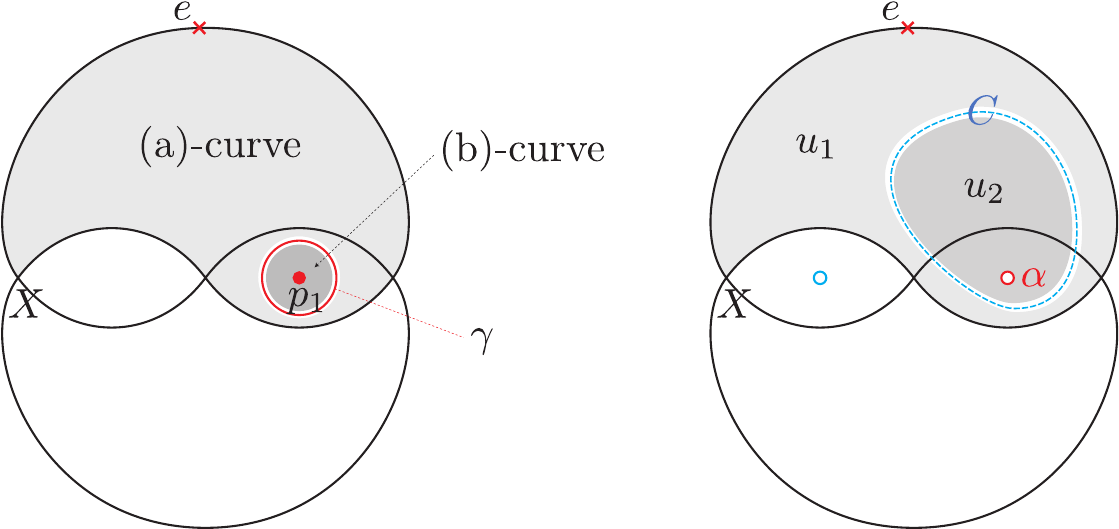}
\caption{Domain stretching and SFT-stretching}
\label{fig:neckstr}
\end{figure}

Clearly $u$ contributes $x e_\bL$ to $\mathfrak{ks} (e_\alpha t_\alpha)$ with coefficient precisely $1$, since so is the corresponding Malsov $2$ disk in $\mathbb{C} P^1$. On the other hand, both $u$ and $u'$ can be viewed as punctured disks with the output $\bar{X}$. Calculating signs carefully, these contributions cancel each other. Factors contributing to their signs appear in \cite[Lemma C.4]{Ab-gen} in details, and in fact, the nontrivial spin structure put on $\bL$ is solely responsible for their opposite signs in our case. This is no longer the case if we twist operations with weights induced from $\gamma_u$ and $\gamma_{u'}$ in Figure \ref{fig:cancelpairks}, in the presence of a nontrivial $\hat{G}$-action.

We next compute $\mathfrak{ks} (f_1)$. By degree reason (in terms of fractional grading) and by $G$-equivariance of $\mathfrak{ks}$,  $\mathfrak{ks} (f_1)$ must be a linear combination of $xX$, $yY$, $zZ$ and $f_{\bL}$. Comparing the actions, we see that $\mathfrak{ks} (f_1)$ can have contributions only from Morse trajectories and constant disks. In terms of the punctured pseudo-holomorphic disks given in \ref{subsec:constksmap}, this is simply a constant disk which intersects a locally finite cycle (the unstable manifold for $f_1$) at the interior puncture, and also passes through the unstable manifold of the Morse function $h_\bL$  at the negative boundary puncture, where the corresponding trajectory $h_\bL$ can flow to a critical point or a constant disk.
Suppose $\mathfrak{ks} (f_1)$ is given as
\[ \mathfrak{ks} (f_1) = c_x xX + c_y yY + c_z zZ + c_{f_\bL} f_{\bL}\]
for some coefficients $c_x,c_y,c_z$ and $c_{f_\bL}$ in $\mathbb{C}$. Firstly, it is immediate to see there are exactly two \emph{cascades} from $f_1$ to $f_\bL$. They are the concatenated paths consisting of Morse trajectories of the Morse function on $\bL$ and those of the ambient Morse function joined at $r$ and $s$ respectively (or the constant disks intersecting locally finite cycles and unstable manifolds of $h_\bL$ at the interior and the boundary, respectively). See (b) of Figure \ref{fig:morsetrajectories}. 
They have the opposite signs due to the intersection parities of two trajectories at $r$ and $s$. Thus we have $c_{f_{\bL}} = 0$.

% By counting broken Morse trajectories, we find that $c_{f_{\bL}} = 0$. 
To find the remaining coefficients $c_x,c_y,c_z$, we argue as follows. 
 Choose an abelian quotient $G$ of $\pi_1 (P)$ of the shape $G= \langle g_\alpha, g_\beta,g_\gamma : g_\alpha^n = g_\alpha g_\beta = g_\gamma =1 \rangle \cong \mathbb{Z}/n$, and consider the corresponding principal $G$-bundle $X \to P$. With respect to $\Theta$ chosen in Figure \ref{fig:morsetrajectories} (a) and a nontrivial $\chi \in \hat{G}$, we compute the differential on $SC^* (P) \otimes \chi (\subset SC^*_\HG(P))$ and $CF((\bL,b),(\bL,b)) \otimes \chi (\subset CF_{\HG}((\bL,b),(\bL,b)))$. The coefficients $c_x,c_y,c_z,$ will be determined by comparison. 
Observe first that  
\[d_{\HG} (e \otimes \chi) = \chi(g_\beta) (1- \chi(g_\alpha )) f_1 \otimes \chi = \left(\chi(g_\alpha^{-1})-1\right) f_1 \otimes \chi\]
by looking at the intersections of two isolated Morse trajectory from $e$ to $f_1$ (for the ambient Morse function) with $\Theta$. Taking $\mathfrak{ks}_{\HG}$ on both sides yields
\begin{equation}\label{eqn:ksdsc}
\mathfrak{ks}_{\HG} \left( d_{\HG} (e \otimes \chi) \right) = \left(\chi(g_\alpha^{-1})-1\right)\mathfrak{ks}_{\HG} ( f_1 \otimes \chi ).
\end{equation}
Since $\mathfrak{ks}_{\hat{G}}$ is a chain map, the left hand side of \eqref{eqn:ksdsc} becomes
\[
\begin{split}
\mathfrak{ks}_{\HG} \left( d_{\HG} (e \otimes \chi) \right) &= m_{1,\HG}^{b \otimes 1,b \otimes 1} \left( \mathfrak{ks}_{\HG} (e \otimes \chi) \right)\\
&=m_{1,\HG}^{b \otimes 1,b \otimes 1} \left( \chi(g_\alpha^{-1}) e_\bL \otimes \chi \right)\\
&= \chi(g_\alpha^{-1}) \left( \left(1-\chi(g_\alpha^{-1})\right) xX \otimes \chi + \left(1-\chi(g_\alpha)\right) yY \otimes \chi \right)\\
&= \left(1-\chi(g_\alpha^{-1})\right)\left( \chi(g_\alpha^{-1})xX \otimes \chi -  yY \otimes \chi\right)
\end{split}
\]
where we used $\mathfrak{ks}_{\HG} (e \otimes \chi)=  \chi(g_\alpha^{-1}) e_\bL \otimes \chi$ in the middle, which can be easily seen from the intersection of the ambient Morse trajectory from $e$ to $e_\bL$ with $\Theta$.

We now compute the right hand side of \eqref{eqn:ksdsc}, $\mathfrak{ks}_{\HG} ( f_1 \otimes \chi )$. By definition, it counts exactly the same pearl trajectories as those for $\mathfrak{ks} ( f_1 )$, but with weights coming from the intersection with $\Theta$.
%
% $\mathfrak{ks}_\chi ( f_1 \otimes \chi)$ admits energy zero contributions. 
Notice that in either cases, contributing trajectories must lie inside the colored Morse trajectories in Figure \ref{fig:morsetrajectories} (b), considering the locations of input and outputs. Keeping track of intersection patterns of these two with $\Theta$, we can compare easily the difference between $\mathfrak{ks}  ( f_1)$ and $\mathfrak{ks}_{\HG} ( f_1 \otimes \chi)$, and it results in
\[\mathfrak{ks}_{\HG} ( f_1 \otimes \chi) = \chi(g_\alpha^{-1})c_x xX \otimes \chi+ c_y yY \otimes \chi +c_z zZ \otimes \chi.\]
Equating this with $-\chi(g_\alpha^{-1})xX \otimes \chi +  yY \otimes \chi$ from the left hand side, we have $c_x =-1$, $c_y = 1$ and $c_z = 0$.
\end{proof}

Consequently, we see that $\mathfrak{ks}$ is a quasi-isomorphism:
%This implies the ring isomorphism between 
%$$SH(P) \to HF((\bL,b),(\bL,b)):= H^\ast ( CF((\bL,b),(\bL,b)), m_1^{b,b}),$$ 
%and 
\begin{cor}\label{cor:basecase}
The map $\mathfrak{ks}$ induces a ring isomorphism  $KS:SH^*(P) \to HF((\bL,b),(\bL,b))$ on the cohomology-level, and hence that between $SH^*(P)$ and $Kos(W)$ due to Proposition \ref{prop:floerkos} or Corollary \ref{cor:cfkostau} (with $H=1$).
\end{cor}

\begin{proof}
Recall that both $SH^0(P)$ and $HF^0 ((\bL,b),(\bL,b))$ are isomorphic to the function ring of the union of three coordinate axes in $\mathbb{C}^3$. Both $SH^*(P)$ and $HF((\bL,b),(\bL,b))$ can be regarded as finitely generated modules over $\mathbb{C}[x,y,z]/\langle xy,yz,zx\rangle$, and Lemma \ref{lem:KSpoppop} identifies the module generators (see \eqref{eqn:logcohommorse} and \eqref{eqn:addgen}). From \eqref{eqn:logcohommorse1} and \eqref{eqn:subjectoll}, product relations among these generators also match.
\end{proof}
 
It is worthwhile to mention that $\mathfrak{ks}$ intertwines algebraic and geometric relations in $SH^*(P)$ and $HF((\bL,b),(\bL,b))$. For example, we have $f_1 \cdot f_2 = 0$ on the chain level (in $SC^* (P)$) since there is no contributing holomorphic curves, but there is a geometric relation $m_2 (xX - yY, zZ - xX) = m_1^{b,b} (f_\bL)$ from disk counting on the other side. Thus $m_2 (\mathfrak{ks}(f_1),\mathfrak{ks}(f_2))=0$ holds only on the cohomology.

\subsection{Abelian covers of pair-of-pants and orbifold LG mirror}
We now establish the closed string mirror symmetry for an abelian cover $X$ of $P=X/G$ based on the quasi-isomorphism $\mathfrak{ks}$, or more precisely, by \emph{orbifolding} $\mathfrak{ks}$ with respect to $\hat{G}$-action. Throughout, we work with the geometric setup given in (a) of Figure \ref{fig:morsetrajectories}. In particular, relative positions of $\Theta$ and Floer generators are fixed. Notice that $\Theta$ here is a Borel-Moore cycle and is different from our earlier choice in \ref{subsec:expopsei}. Figure \ref{fig:comparegraph} shows the isotopy between the two not crossing self-intersections of $\bL$, and hence they result in the same calculation. 
We then have the following.
%find the extent to which we can generalize this. As Floer invariants of abelian covers of a Liouville manifold $X$ can be transcribed to fit in the context of equivariant geome	try, it is reasonable to suspect the mirror relation to hold for general abelian covers of $P$ as well.

\begin{figure}[h]
\includegraphics[scale=0.4]{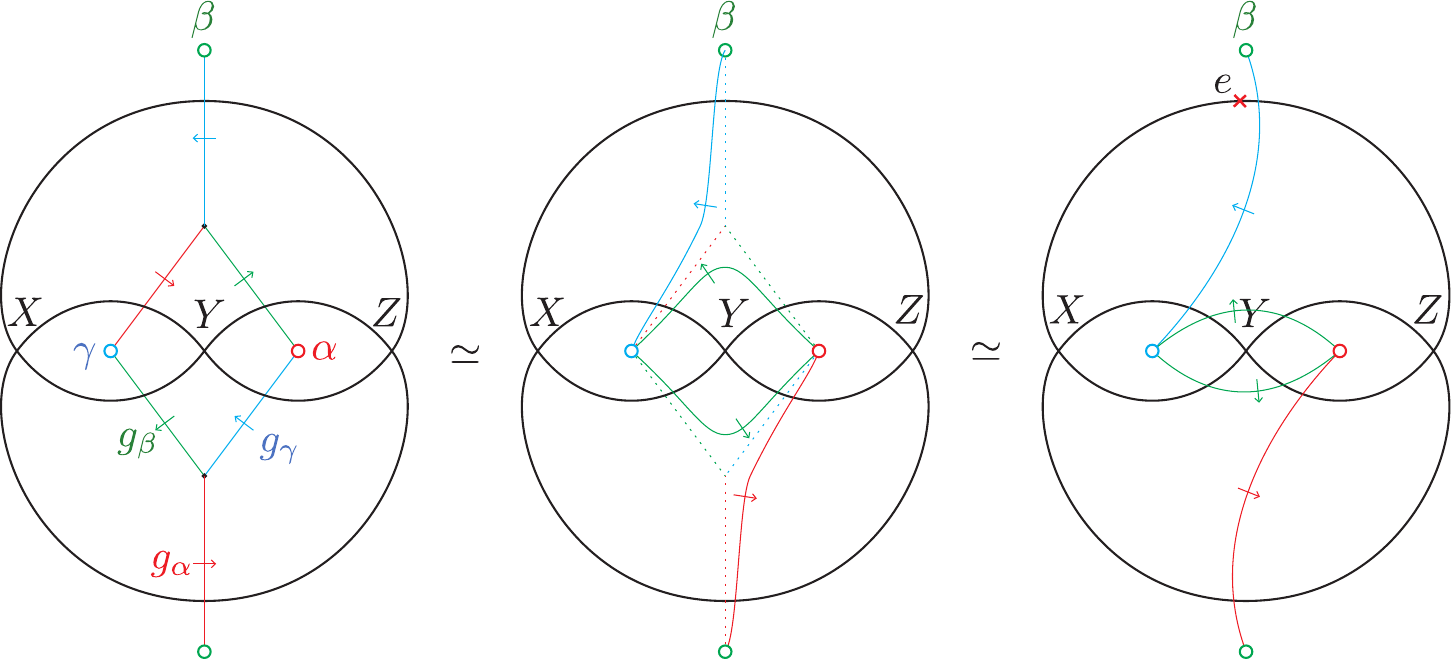}
\caption{Deforming $\Theta$ on $P$ given in Figure \ref{fig:G-graph}}
\label{fig:comparegraph}
\end{figure}

\begin{thm}\label{thm:mainmain}
%The map
%\[\mathfrak{ks}_{\HG}:SC_\HG^\ast(P)\to CF_{\HG}((\bL,b),(\bL,b))\]
%defined in \eqref{eqn:equiksdef} is a $G\times\HG$-equivariant quasi-isomorphism.
The map
\[\mathfrak{ks}_{\HG}:SC_\HG^*(P)\to CF_{\HG}((\bL,b),(\bL,b))\]
 is a quasi-isomorphism. Hence, combined with Lemma \ref{lem:ksalgebraic}, its cohomology-level map
 $$ KS_{\HG} : SH_\HG^*(P)\to HF_{\HG}((\bL,b),(\bL,b))$$
 is a ring isomorphism.
\end{thm}

\begin{figure}[h]
\includegraphics[scale=0.6]{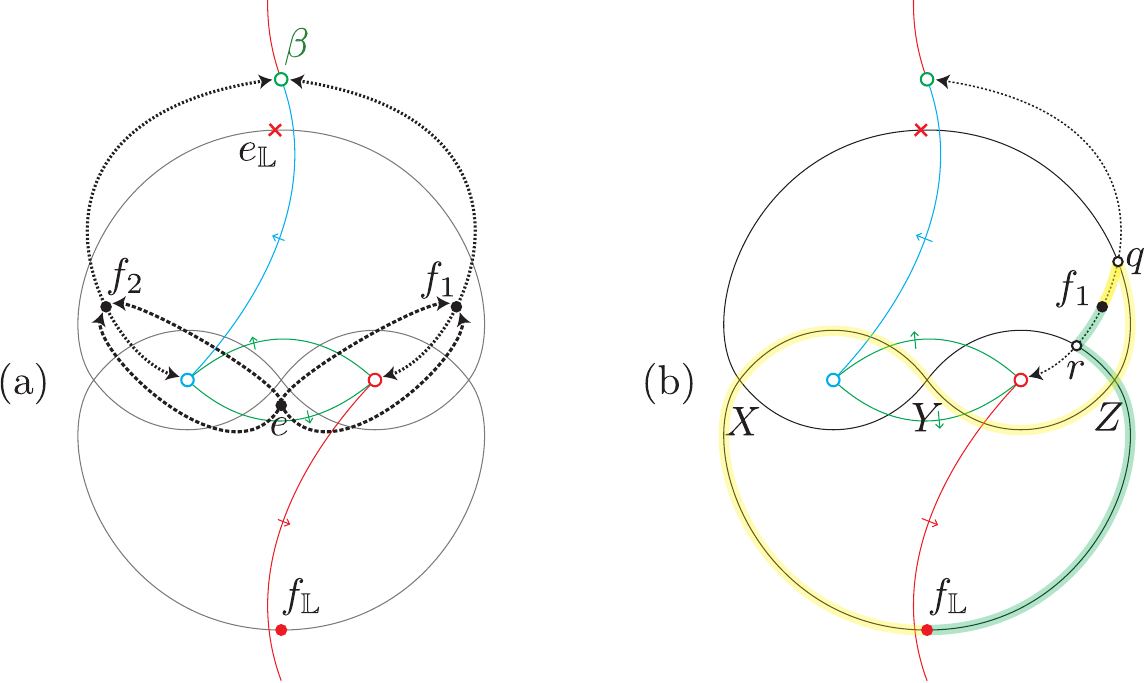}
\caption{Choice of Morse functions and their trajectories contributing to $\mathfrak{ks}_{\hat{G}}$}
\label{fig:morsetrajectories}
\end{figure}

\begin{proof}
%Instead of directly proving that $\mathfrak{ks}_{\hat{G}}$ is a quasi-isomorphism, 
Since we know $\mathfrak{ks}_{\HG}$ preserves the product structure, it suffice to prove that the composition below is a quasi-isomorphism of complexes
\[
\begin{tikzcd}[column sep=small]
SC_\HG^*(P) \ar[r, "\mathfrak{ks}_{\HG}", shorten <= -0.2em, shorten >= -0.2em]
	& CF_{\HG}((\bL,b),(\bL,b)) \ar[r, "="]
		& CF((\bL,b),(\bL,b); \HG) \ar[r, "\sim", "\tau"']
			& K^*(\partial W,\HG).
\end{tikzcd}
\]
in each $\chi$-sector $SC^*(P) \otimes \chi$.
We denote this composition by $\mathcal{F} \otimes \chi$. The case where $\chi=1$ corresponds to the pair-of-pants $P$ itself, which is shown in Corollary \ref{cor:basecase}. Assume $\chi\neq1$ for the rest of the proof.

We first claim that the part of $SC^*(P) \otimes \chi$ generated by Morse critical points $e \otimes \chi$, $f_1 \otimes \chi$ and $f_2 \otimes \chi$ gives rise to an 1-dimensional cohomology in the odd degree. To see this, we count the gradient trajectories $\gamma_i$ and $\delta_i$ in Figure \ref{fig:comparegraph} between these generators $e$ and $f_i$, but with the weights determined by their intersections with $\Theta$ (due to the effect of twisting in the $\chi$-sector). By direct calculations 
%Note that $SC^\ast(P) \otimes \chi$ admits a direct sum decomposition over the free loop classes in $P$. Thus if we fix a free loop class $\ell$, we get a subcomplex $SC^{\ast}_{\ell}(P;\chi)$. Let $SC^{\ast}_{l=0}(P;\chi)$ be the subcomplex corresponding to the contractible loop class. The structure of $SC^{\ast}_{l=0}(P;\chi)$ is the same as the Morse complex $CM^\ast(h;P)$ but equipped with a $\chi$-twisted differential. By examining how the gradient flows intersect with $\Theta$, we find that
\[d(e\otimes\chi)= \pm\left(\chi(\gamma_1) -  \chi(\delta_1)\right) f_1 \otimes \chi \pm \left(\chi(\gamma_2) - \chi(\delta_2)\right) f_2 \otimes\chi. \] 
This is always nonzero for $\chi \neq 1$ since $\gamma_1* \delta_1^{-}$, $\gamma_2 * \delta_2^{-}$ generate the first homology of $P$. Hence $e \otimes \chi$ does not support a cocycle in this case, and $f_1, f_2$ are colinear in the cohomology.  
%The three cocycles in $SC^{1}_{c}(P;\chi)$ has two independent linear relations coming from the differential of two cochains in $SC^{0}_{c}(P;\chi)$, leaving 

Let us next consider non-contractible loops of $SC^*(P) \otimes \chi$. We only have possibly nontrivial Floer differential between the elements of the form $e_\alpha t_\alpha^n \otimes \chi$ and $f_\alpha t_\alpha^n \otimes \chi$ (and similar for $\beta$ and $\gamma$) since otherwise there does not exist any contributing Floer cylinders.
%
%When $\ell$ is non-contractible, $SC^{\ast}_{\ell}(P;\chi)$ is either zero or the span of two cochains, one for each degree. Specifically, $SC^{\ast}_{\ell}(P;\chi)$ is the span of two cochains iff $\ell=$ where $n$-folded class of one of the boundary circles $S_\alpha,$ $S_\beta,$ and $S_\gamma$ of $P^{in}$.
%Let $\ell$ be a non-contractible loop class such that $SC^{\ast}_{\ell}(P;\chi)\neq0$ and let $S$ be the boundary circle corresponding to $\ell$. Let $e_\ell\otimes\chi$ and $f_\ell\otimes\chi$ the be the generators of $SC^{\ast}_{\ell}(P;\chi)$, even and odd. 
There are exactly two holomorphic cylinders $u,v$ whose two ends asymptote $e_\alpha t_\alpha^n \otimes \chi$ (input) and $f_\alpha t_\alpha^n \otimes \chi$ (output) or their underlying orbits. The contributions of the two cylinders $u,v$ cancel in the absence of $\chi$, so both $e_\alpha t_\alpha^n$ and $f_\alpha t_\alpha^n$ survive in $SH^*(P)$. 
%This results that both $e_\ell$ and $f_\ell$ represents a nontrivial cohomology class in $SC^\ast(P)$, whereas 
However, the twisted differential of $SC^*(P)\otimes \chi$ applied to $e_\alpha t_\alpha^n\otimes\chi$ gives $\pm(\chi(\gamma_{u})-\chi(\gamma_{v}))f_\alpha t_\alpha^n\otimes\chi$, and $\chi(\gamma_{u})-\chi(\gamma_{v})=0$ if and only if  $\chi ( \gamma_u * \gamma_v^{-} ) = \chi (\alpha) = 1$ (since $\gamma_u * \gamma_v^{-}$ or its reverse is homotopic to $\alpha$, see  \cite[Section 2]{CFHW} for more details).

To summarize, $SH^* (P) \otimes \chi$ takes a single generator in degree $1$ from the Morse part for $\chi \neq 1$, and  takes two generators $e_\alpha t_\alpha^n \otimes \chi$ and $f_\alpha t_\alpha^n \otimes \chi$ ($n \geq 1$) if and only if $\chi(g_\alpha)   = 1$, and similar for $\beta$ and $\gamma$. Let us now look into the behavior of $\mathfrak{ks}_{\HG}$ on these generators. We divide our proof  into two cases depending on the action of $\chi$ on three variables $x$, $y$ and $z$.
\newline

\noindent\textit{Case 1.} When $\chi$ fixes at least one of $x$, $y$, and $z$:\\ 
Without loss of generality, we assume $\chi\cdot x = x$. As seen above, there are exactly two nonzero cohomology classes $e_\alpha t_\alpha^n \otimes \chi$ and $f_\alpha t_\alpha^n \otimes \chi$ (of degrees 0 and 1, respectively) belonging to the free loop class $\alpha^n$ for each $n \geq 1$. Summing up over all $n \geq 1$ and adjoining $f_1 \otimes \chi$, they form a free $\mathbb{C}e \oplus e_\alpha t_\alpha \C [ t_\alpha] $-module of rank $2$, generated by $e_\alpha t_\alpha \otimes \chi$ and $f_1   \otimes \chi$ in degree $0$ and $1$. Here, $\mathbb{C}e \oplus e_\alpha t_\alpha \C [ t_\alpha]$ means $(\mathbb{C}e \oplus e_\alpha t_\alpha \C [ t_\alpha] ) \otimes 1 (\subset SH^*(P;\hat{G})$), but we omit $-\otimes 1$ for notational simplicity.

The module structure is induced from the product on $SH^*(P;\hat{G})$, and more specifically, it is given as
$$(e_\alpha t_\alpha^{n}\otimes1) \cdot (e_\alpha t_\alpha \otimes\chi)= e_\alpha t_\alpha^{n+1} \otimes\chi, \quad  (e_\alpha t_\alpha^n\otimes1)\cdot(f_1\otimes\chi) =(f_\alpha t_\alpha^n)\otimes\chi$$
where $e_\alpha t_\alpha^{n-1}\otimes1$ and $e_\alpha t_\alpha^{n-1}$ in $\mathbb{C}e \oplus e_\alpha t_\alpha \C [ t_\alpha]$ are identified. Observe that the ring $\mathbb{C}e \oplus e_\alpha t_\alpha \C [ t_\alpha]$ is nothing but a polynomial ring in $1$-variable, and in fact, it maps to $\mathbb{C}[x]$ under $\mathcal{F} \otimes 1$ .
%
%The  $e_\alpha t_\alpha \C [ t_\alpha](\otimes 1)$-module structure on the component generated by the union of these classes over $n=1,2,\cdots$ is given as 
%$$e_\alpha t_\alpha^n\otimes\chi=(e_\alpha t_\alpha^{n-1}\otimes1) \cdot (e_\alpha t_\alpha \otimes\chi), \quad (f_\alpha t_\alpha^n)\otimes\chi=(e_\alpha t_\alpha^n\otimes1)\cdot(f_2\otimes\chi).$$ 
%Note that $SH^{\ast}(P;\chi)$ can only be supported for free loop classes $c$ and $\alpha^n\,(n\geq1)$.

On the other hand, $\C[x,y,z]^\chi=\C[x]$ in this case, so the Koszul complex $K^*(\partial W^\chi)$ is a $\C[x]$-module in the following form:
\[
\begin{tikzcd}[column sep=small]
0 \rar & \C[x]\theta_x\theta_y\theta_z \rar{0} & \C[x]\theta_y\theta_z \rar & 0.
\end{tikzcd}
\]
Hence its cohomology $Kos(W,\chi)$ is a free $\C[x]$-module of rank $1$ in each degree ($0$ or $1$). It is freely generated by any nonzero multiples of $\theta_x\theta_y\theta_z $ and $\theta_y\theta_z $ in degree $1$ and $0$, respectively. 

Pseudo-holomorphic curves contributing to $\mathfrak{ks}_{\hat{G}} ( e_\alpha t_\alpha \otimes \chi )$ are already classified in the proof of Lemma \ref{lem:KSpoppop}. Additional factor here is the nontrivial weight coming from the intersection between $\Theta$ and $\gamma_u$ or $\gamma_{u'}$. Take these into account, we obtain
\begin{equation}\label{eqn:kshatge1}
\mathfrak{ks}_{\hat{G}} ( e_\alpha t_\alpha \otimes \chi )= \pm(\chi(g_\beta)-1)\bar{X}\otimes \chi \pm  x\,e_\bL\otimes\chi .
\end{equation}
After mapped under $\tau$ in \eqref{eqn:tautautau}, only the first term in \eqref{eqn:kshatge1} survives, and becomes a nonzero multiple of $\theta_y \theta_z$ since $\chi(g_\beta) \neq 1$.

Let us now compute $\mathfrak{ks}_{\hat{G}} (f_1 \otimes \chi )$. Considering the behavior of  $\tau$ in \eqref{eqn:tautautau},  it suffice to prove that the coefficient of $f_\bL \otimes \chi$ in $\mathfrak{ks}_{\hat{G}} (f_1 \otimes \chi )$ is nonzero (so that $f_1 \otimes \chi $ maps to a nonzero multiple of $\theta_x \theta_y \theta_z $ under $\mathcal{F}\otimes \chi$.) Contributions are from two (concatenated) morse trajectories marked in \ref{fig:morsetrajectories} (b). Tracking the intersections of these trajectories with  $\Theta$, we have
\begin{equation}\label{eqn:ksf1}
\mathfrak{ks}_{\hat{G}} (f_1 \otimes \chi ) = \pm (1 - \chi(g_\gamma))f_\bL \otimes\chi\pm A yY\otimes\chi\pm B xX\otimes\chi ,
%\mathfrak{ks}_{\hat{G}} (f_1 \otimes \chi ) = \pm (1 - \chi(g_\gamma))f_\bL \otimes\chi\pm \chi(g_\gamma)yY\otimes\chi\pm\chi(g_\alpha^{-1})\chi(g_\gamma)xX\otimes\chi.
\end{equation}
and we see that the coefficient of $f_\bL \otimes \chi$ is nonzero as desired. (They cancel each other when $\chi=1$ as seen in the proof of Lemma \ref{lem:KSpoppop}.)
%
%{\color{red} Fill out the details
%Trick:
%$$\mathfrak{ks}_\chi ( d_\chi (e \otimes \chi) ) = \mathfrak{ks}_\chi ((...) f_1 \otimes \chi ) $$
%$$ d_\chi ( \mathfrak{ks}_\chi (e \otimes \chi) )= (...) ( \mathfrak{ks}_\chi  f_1 \otimes \chi ) $$
%$$ d_\chi ( (....) e_\bL \otimes \chi )= (...) ( \mathfrak{ks}_\chi  f_1 \otimes \chi ) $$
%$$  (....) ( (....) xX \otimes \chi + (....) yY \otimes \chi )= (...) ( \mathfrak{ks}_\chi  f_1 \otimes \chi ).$$
%}
%
%...
%Therefore, combined with \ref{}, $e_\alpha t_\alpha \otimes \chi$ and $f_1 \otimes \chi$ are mapped under $\mathcal{F}  \otimes \chi$ to $( ..... )\theta_x\theta_y$ and $(....) \theta_x\theta_y\theta_z$, respectively, which are free generators of $Kos(W,\chi)$ over $\mathbb{C}[x]$. 
Consequently, $\mathcal{F} \otimes \chi$ matches sets of free generators of two free modules over $\mathbb{C}e \oplus e_\alpha t_\alpha \mathbb{C}[t_\alpha]$ and $\mathbb{C} [x]$ which are identified as rings by $\mathcal{F} \otimes 1$, and hence, is an isomorphism.
 %for some $C\in\C$ and $D \in \C$. 
%Note that $C\neq0$ and $D\neq0$ {\color{red} since the vanishing of any of these numbers implies $\chi \cdot y=y$ or $\chi \cdot z = z$ so that $\chi=1$.}

%
%As these cohomology classes pass through the map $\mathcal{F}$, we see that $e_\alpha t_\alpha \otimes\chi\in SH^{0}_{\alpha}(P;\chi)$ maps to $C\cdot\theta_x\theta_y$ for some $C\in\C$ and $f_1\otimes\chi\in SH^{1}_{c}(P;\chi)$ maps to $D\cdot\theta_x\theta_y\theta_z$ for some $D \in \C$. Note that $C\neq0$ and $D\neq0$ {\color{red} since the vanishing of any of these numbers implies $\chi \cdot y=y$ or $\chi \cdot z = z$ so that $\chi=1$.}

%Note that the cohomology of $K^\ast(\partial W^\chi)$ is a free $\C[x]$-module on which the multiplication by $x$ is obvious. The action of $e_\alpha t_\alpha\otimes1\in e_\alpha t_\alpha\C[t_\alpha]$ on the corresponding cocycles of $SC^\ast(P)\otimes\chi$ is exactly the same, establishing the desired quasi-isomorphism.

\vspace{0.4cm}

\noindent\textit{Case 2.} When $\chi$ fixes none of $x$, $y$, and $z$:\\
We have $\C[x,y,z]^\chi=\C$ and the Koszul complex is the free module generated by $\theta_x\theta_y\theta_z$ over $\C$. Also, non-contractible orbits cannot become a cocycle in $SC^*(P)\otimes\chi$ as we have examined above. 
%To be more precise, we have that $\chi(\alpha) = 1$ if and only if $\chi \cdot x = x$, etc. This contradicts the hypothesis. 
Thus the only cohomology generator of $SH^*(P)\otimes\chi$ is of degree 1, represented by the class $f_1\otimes\chi$ or $f_2\otimes\chi$. As   in $\eqref{eqn:ksf1}$, $\mathfrak{ks}_{\HG} (f_1 \otimes \chi)$ has a nonzero coefficient for $f_\bL\otimes\chi$ term (since $\chi(g_\gamma) \neq 1$), and maps to a nonzero multiple of $\theta_x \theta_y \theta_z$ under $\mathcal{F} \otimes \chi$, which completes the proof.
\end{proof}

The following is an immediate consequence of $G\times\HG$-equivariance of $\mathfrak{ks}_{\HG}$.

\begin{cor}\label{cor:final}
The Kodaira-Spencer map restricts to a $G$-equivariant quasi-isomorphism
\[\mathfrak{ks}_{\HG}:\left(SC_\HG^*(P) \right)^{\HG}\to \left(CF_{\HG}((\bL,b),(\bL,b))\right)^{\HG}\]
on $\HG$-invariant parts of both sides. Therefore we obtain an isomorphism
$$SH^*(X) \cong Kos (W,\hat{G})$$
as modules. \end{cor}

We can also equip $Kos (W,\hat{G})$ with the product structure by transferring $m_{2,\hat{G}}^{b\otimes 1,b\otimes 1,b\otimes 1 }$ on $ H^* \left(CF_{\HG}((\bL,b),(\bL,b))\right)^{\HG}$ in this case. Corollary \ref{cor:final} justifies the new product structure in the mirror symmetry context. More general treatment will be provided elsewhere. The explicit formula for the differential on $CF_{\HG}((\bL,b),(\bL,b))$ is presented in Appendix \ref{app:doncfg} as well as one convenient choice of generators for its cohomology (which is an infinite dimensional vector space over $\mathbb{C}$).

\subsection{Examples}\label{subsec:exexex}

We end our discussion with simple examples that manifest the feature of the effect of orbifolding. The following are the examples whose mirrors are heuristically explained in Introduction (see Figure \ref{fig:4punc} and Figure \ref{fig:3puncell}). Here, we give a precise formulation of their orbifold LG mirrors.

\subsubsection{Sphere with four punctures}\label{subsubsec:4punc}
Suppose $X$ is a sphere with four punctures. It admits the rotation action of $G=\mathbb{Z}/2$ whose quotient is the pair-of-pants $P$. As the deck transformation group, $G$ can be identified as the quotient of $\pi_1 (P)$ as
$$ \langle g_\alpha,g_\beta, g_\gamma : g_\alpha^2 = g_\alpha g_\beta  = g_\gamma = 1 \rangle.$$
%The covering $X \to P$ corresponds to 
%The order 2 group homomorphism $\tilde{\chi}:\pi_1(P) \to U(1)$ given by
%\[\tilde{\chi}(\alpha)=1, \quad \tilde{\chi}(\beta) = \tilde{\chi}(\gamma) = -1\]
%results in an order 2 group $G = \pi_1(P)/\ker\,\chi$ and a $G$-covering $X \to P$ which is a 4-punctured sphere. 
In the character group $\HG$, we have a unique non-unit $\chi$ whose value on the generator $g_\alpha=g_\beta$ is equal to $-1$. This specializes the situation in ``\emph{case} 1" in the proof of Theorem \ref{thm:mainmain}.
%The nontrivial twisted sector $SH(P) \otimes \chi$ in $SH(X)=SH(P;\hat{G})^{\hat{G}}$ is generated by
%$$ ....$$
%in degree $0$,
%$$....$$
%in degree $1$.
%
On the mirror side, the nontrivial twisted sector $Kos(W,\chi)$ is given by
$$ \mathbb{C}[z] \theta_x \theta_y \oplus \mathbb{C}[z] \theta_x \theta_y \theta_z.$$
Note that this is precisely the (algebraic) de Rham complex of $\mathbb{C}$ (with coordinate $z$) up to degree shift which appear in Figure \ref{fig:3puncell} (b). 
In terms of $CF((\bL,b),(\bL,b)) \otimes \chi$, $ \theta_x \theta_y$ and $\theta_x \theta_y \theta_z$ correspond to
\begin{equation*}
 \bar{Z}\otimes\chi + \frac{1}{2} z\, e_\bL\otimes\chi , \quad \,f_\bL\otimes\chi + \frac{1}{2} 
(yY\otimes\chi + zZ\otimes\chi).
\end{equation*}
One can compute the new product structure on $Kos(W,\hat{G})$ using these representatives in the Floer complex or using their corresponding elements in $SH^*(X)$ (see Appendix \ref{AppendixB}).

\subsubsection{Torus with three punctures}
Let us now take $X$ to be a torus with three punctures. The deck transformation group of $X \to P$ can be presented as
$$G=  \langle g_\alpha,g_\beta, g_\gamma : g_\alpha^3 = g_\beta^3 g_\gamma^3  = g_\alpha g_\beta g_\gamma = 1 \rangle,$$
and choose $\chi \in G$ such that $\chi(g_\alpha)=\chi(g_\beta) = \chi (g_\gamma) = \rho$
where $\rho = e^{2 \pi i / 3} \in \C$. This specializes ``\emph{case} 2" in the proof of Theorem \ref{thm:mainmain}.
% satisfies $\rho^2+\rho+1=0$ gives us an orbifold model of the 3-punctured elliptic curve.

In this case, we have a twist sector for each of $\chi$ and $\chi^2$. Since there are no invariant variables, twisted sectors are given as
$$Kos(W,\chi)  = \mathbb{C} \, \theta_x \theta_y \theta_z, \quad Kos(W,\chi^2) = \mathbb{C} \, \theta_x \theta_y \theta_z. $$
These correspond to the two additional points in Figure \ref{fig:3puncell} (b).

In terms of $CF((\bL,b),(\bL,b) ; \hat{G})^{\hat{G}}$, these two sectors can be more clearly distinguished. They are generated by 
\begin{equation*}
(1-\rho)f_\bL\otimes\chi - \rho^2yY\otimes\chi + zZ\otimes\chi,\quad 
(1-\rho^2)f_\bL\otimes\chi^2-  \rho yY\otimes\chi^2+ zZ\otimes\chi^2,
\end{equation*}
respectively for $\chi$ and $\chi^2$.
%\section{speculations}
%{\color{red}
%
%\begin{enumerate}
%\item For $\phi: \pi_1 (\Sigma) = \langle \alpha,\beta,\gamma : \alpha \beta \gamma =1 \rangle \to \mathbb{Z} / n \times \mathbb{Z} /m$ with $\phi (\alpha) = (1,0)$ and $\phi (\beta) = (0,1)$, the corresponding (abelian) covering of the pair-of-pants is a genus $g=\frac{(m-1)(n-1) - d+1}{2}$ surface with $(m+n+d)$ punctures where $d = \gcd (m,n)$. Does this family cover all the punctured Riemann surface having at least $3$-punctures?
%\item If a given punctured Riemann surface has two different abelian quotient (into the pair-of-pants), we have different equivariant structures on the mirror LG model. Is there some natural (intrinsic) reason to expect that the corresponding orbifold LG models are equivalent? (E.g., 6-punctured sphere can be quotiented by $\mathbb{Z}/4$ or $\mathbb{Z} / 2 \times \mathbb{Z} /2$ to give a pair-of-pants.)
%\item $SH$ can be $\mathbb{Z}$-graded after fixing a spin structure on the punctured Riemann surface. What $\mathbb{Z}$-grading on LG side correspond to this?
%\item How to find mirrors for surfaces with less than 3 punctures. E.g. Once-punctured torus looks doable.
%\item What are ``maximal" quotients of a given Riemann surface?
%
%
%\end{enumerate}
%}

\appendix 

\section{Calculation of $CF_{\HG}((\bL,b),(\bL,b))$ for the Seidel Lagrangian $\bL$}\label{app:doncfg}
In the following, we give an explicit calculation of the differential $d = m_{1,\HG}^{b\otimes1,b\otimes1}$ on $CF_{\HG}((\bL,b),(\bL,b))=CF((\bL,b),(\bL,b)) \otimes \Bbbk [\hat{G}]$ for the Seidel Lagrangian $\bL$ in $P$ with the choice of $\Theta$ in Figure \ref{fig:comparegraph}. 
The proof is similar to the argument in \ref{subsec:popFloer} (Proposition \ref{prop:CFKOS}, in particular)) which can be viewed as the computation for the component $CF((\bL,b),(\bL,b)) \otimes \chi$ with $\chi=1 \in \hat{G}$. Alternatively, one can use the identification in Lemma \ref{lem:idbetntwo} and compute the differential on $CF  (\bL,\bL;\hat{G})$. We leave details as an exercise for readers.

\begin{lemma}
For $\chi\in\HG$, we have
\begin{equation}\label{eqn:cfevendiff}\arraycolsep=1.4pt
\begin{array}{rrrr}
d(e_\bL\otimes\chi) = &\left(1-\chi(g_\alpha^{-1})\right)xX& + \left(1-\chi(g_\beta^{-1})\right)yY & + \left(1-\chi(g_\gamma^{-1})\right)zZ,\\

d(\bar{X}\otimes\chi)= & \left(\chi(g_\alpha)-1\right)x\,f_\bL &+\chi(g_\beta^{-1})xyY &-xzZ,\\

d(\bar{Y}\otimes\chi)=&\left(\chi(g_\beta)-1\right)y\,f_\bL&-\chi(g_\alpha^{-1})xyX&+\chi(g_\beta)yzZ,\\

d(\bar{Z}\otimes\chi)=&\left(\chi(g_\gamma)-1\right)z\,f_\bL&+\chi(g_\gamma) \chi(g_\alpha^{-1}) xzX&-\chi(g_\gamma)yzY\; ,
\end{array}
\end{equation}
and
\begin{equation*}\label{eqn:cfodddiff}\arraycolsep=1.4pt
\begin{array}{rrrr}
d(f_\bL\otimes\chi) = &\chi(g_\alpha^{-1})yz\bar{X}&+\chi(g_\gamma)xz\bar{Y}&+xy\bar{Z},\\

d(X\otimes\chi) = &yz\,e_\bL& +  \left(\chi(g_\alpha)-\chi(g_\beta^{-1})\right)z\bar{Y} & + \left(\chi(g_\gamma^{-1})-\chi(g_\alpha)\right)y\bar{Z},\\

d(Y\otimes\chi) = &\chi(g_\alpha^{-1})xz\,e_\bL& +  \left(\chi(g_\alpha^{-1})-\chi(g_\beta)\right)z\bar{X} &+ \left(\chi(g_\beta)-\chi(g_\gamma^{-1})\right)x\bar{Z},\\

d(Z\otimes\chi) = &\chi(g_\gamma)xy\,e_\bL&+\left(\chi(g_\gamma)-\chi(g_\alpha^{-1})\right)y\bar{X}&+ \left(\chi(g_\beta^{-1})-\chi(g_\gamma)\right)x\bar{Y}\; .
\end{array}
\end{equation*}
All the outputs belong to the $\chi$-sector by definition of $d$ ($-\otimes \chi$ omitted on the right hand side).
%confined ourselves to denoting the sector (which is $\chi$) from the output.
\end{lemma}

%\begin{proof}
%This is straightforward from $m_1^{b,b'}$ calculation in the proof of Proposition \ref{prop:CFKOS} together with Figure \ref{fig:comparegraph} since we can track the group element passed by $\gamma_{u,1}$ of each $u$ of the contributing discs.
%\end{proof}

Note that $CF_{\HG}((\bL,b),(\bL,b))$ and its cohomology are infinite dimensional over $\mathbb{C}$, but are naturally $\mathbb{C}[x,y,z]$-modules.
Below, we provide one particular choice of (independent) generators of the cohomology over $\mathbb{C}[x,y,z]$. 

Given the full expression of $d$, the following is simply an exercise in homological algebra.

%of $CF_{\HG}((\bL,b),(\bL,b))$
%
%The list of independent $d$-cocycles generating the cohomology of $CF_{\HG}((\bL,b),(\bL,b))$ will be provided also.

%As there is a full list of coboundaries, it is the turn to list the cohomology generators.
\begin{lemma}
For $\chi \neq 1$, let $P_\chi$, $Q_\chi$ and $R_\chi$ be odd degree cocycles in $CF((\bL,b),(\bL,b)) \otimes \chi (\subset CF_{\hat{G}} ((\bL,b),(\bL,b))  )$ given by
\begin{equation*}
\arraycolsep=1.4pt
\begin{array}{rrrrrrrr}
P_\chi:= & (1-\chi(g_\alpha))f_\bL\otimes\chi&-& \chi(g_\beta^{-1})yY\otimes\chi&+&zZ\otimes\chi,\\
Q_\chi:= & (1-\chi(g_\beta))f_\bL\otimes\chi&+& \chi(g_\alpha^{-1})xX\otimes\chi&-&\chi(g_\beta)zZ\otimes\chi,\\
R_\chi:= & (1-\chi(g_\gamma))f_\bL\otimes\chi&-& \chi(g_\gamma)\chi(g_\alpha^{-1})xX\otimes\chi &+& \chi(g_\gamma)yY\otimes\chi.
\end{array}
\end{equation*}
\begin{enumerate}
\item[(i)]
Suppose $\chi (g_\alpha)=1$, but $\chi (g_\beta) = \chi (g_\gamma)^{-1} \neq 1$ (\emph{cf}. ``\emph{case} 1" in the proof of Theorem \ref{thm:mainmain}). Then 
$$H^{odd} (CF ((\bL,b),(\bL,b)) \otimes \chi) \cong \C[x]$$ 
as $\C[x,y,z]$-modules, and it is generated by $[Q_\chi]$ or $[R_\chi]$ (in this case $[P_\chi]=0$). Analogous statements hold when $\chi(g_\beta)=1$ or $\chi(g_\gamma)=1$.
\item[(ii)]
If none of $\chi(g_\alpha),\chi(g_\beta),\chi (g_\gamma)$ equals to $1$ (\emph{cf}. ``\emph{case} 2" in the proof of Theorem \ref{thm:mainmain}), then 
$$H^{odd} (CF ((\bL,b),(\bL,b)) \otimes \chi) \cong \C$$ 
as $\C[x,y,z]$-modules, and it is generated by any of $[P_\chi]$ or $[Q_\chi]$ or $[R_\chi]$.

\end{enumerate}
%
%for all $\chi \in \HG$ generates the odd degree cohomology of $(CF_{\HG}((\bL,b),(\bL,b)),d)$ over $\C[x,y,z]$.
\end{lemma}

We next describe the even degree cohomology.
For $\chi \neq 1$, define even cochains $U_\chi$, $V_\chi$ and $W_\chi$ by
\begin{equation*}\label{eqn:UVWcochain}\arraycolsep=1.4pt
\begin{array}{rrrrrr}
U_\chi:=&\chi^{-1}(g_\alpha)x\,e_\bL\otimes\chi&+&(\chi^{-1}(g_\alpha)-\chi(g_\beta))\bar{X}\otimes\chi,\\
V_\chi:=&\chi(g_\gamma)y\,e_\bL\otimes\chi&+&(\chi^{-1}(g_\beta)-\chi(g_\gamma))\bar{Y}\otimes\chi,\\
W_\chi:=&z\,e_\bL\otimes\chi&+&(\chi^{-1}(g_\gamma)-\chi(g_\alpha))\bar{Z}\otimes\chi.
\end{array}
\end{equation*}

\begin{lemma}
The even degree cohomology of $CF_{\HG}((\bL,b),(\bL,b))$ is given as follows.

\item[(i)]
Suppose $\chi (g_\alpha)=1$, but $\chi (g_\beta) = \chi (g_\gamma)^{-1} \neq 1$. Then 
$$H^{even} (CF ((\bL,b),(\bL,b)) \otimes \chi) \cong \C[x]$$ 
as $\C[x,y,z]$-modules, and it is generated by $[U_\chi]$. Analogous statements hold when $\chi(g_\beta)=1$ or $\chi(g_\gamma)=1$.
\item[(ii)]
If none of $\chi(g_\alpha),\chi(g_\beta),\chi (g_\gamma)$ equals to $1$, then 
%none of $U_\chi$, $V_\chi$, and $W_\chi$ is a cocycle and therefore
$$H^{even} (CF ((\bL,b),(\bL,b)) \otimes \chi) \cong 0$$ 
as $\C[x,y,z]$-modules.
\end{lemma}
In fact, direct calculation from $\eqref{eqn:cfevendiff}$ shows that $U_\chi$ is a cocycle if and only if $\chi(g_\alpha) = 1$, $V_\chi$ is a cocycle if and only if $\chi(g_\beta) = 1$, and $W_\chi$ is a cocycle if and only if $\chi(g_\gamma) = 1$.

%\begin{proof}
%Suppose $\chi(g_\alpha)=1$. Thus, $U_\chi$, among the cochains in $\eqref{eqn:UVWcochain}$, is a cocycle of even degree. We see again by direct calculation from $\eqref{eqn:cfevendiff}$ that
%\begin{equation*}
%d(Y\otimes\chi)  = zU_\chi,\quad d(Z\otimes\chi) = \chi^{-1}(g_\beta)yU_\chi
%\end{equation*}
%which shows the annihilation of $U_\chi$ under the multiplication by elements in the ideal $\langle y,z \rangle$, yielding a free $\C[x]$-module on even degree cohomology, generated by $U_\chi$. When $\chi(g_\beta)$ or $\chi(g_\gamma)$ is equal to $1$, then the even degree cohomology is generated by either $V_\chi$ or $W_\chi$ freely over $\C[y]$ or $\C[z]$, accordingly.
%\end{proof}

\begin{remark}
One can check that $P_\chi$ and $R_\chi$ are images of $f_2 \otimes \chi $ and $f_1 \otimes \chi$ in $SH^*(P) \otimes \chi (\subset SH^*_\HG(P) )$ under $\mathfrak{ks}_\HG$. Likewise, $U_\chi$, $V_\chi$ and $W_\chi$  are the images of $e_\alpha t_\alpha\otimes\chi$, $e_\beta t_\beta\otimes\chi$, and $e_\gamma t_\gamma\otimes\chi$.
%under the orbifold Kodaira-Spencer map $\mathfrak{ks}_{\HG}$, listing all the possible candidates of cohomology generators of $CF_{\HG}((\bL,b),(\bL,b))$ in the even degree.
\end{remark}

\section{Products in orbifold Koszul algebras}\label{AppendixB}
In this section we write down the product of $Kos(xyz,\Z/2)$ explicitly, mentioned briefly in Section \ref{subsubsec:4punc}. Recall that $\Z/2=\{1,\chi\}$ acts on $\C[x,y,z]$ by 
 \[ \chi\cdot x=-x,\; \chi\cdot y=-y, \chi\cdot z=z.\]
 In Section \ref{sec:KSRS} we adopted a Lagrangian Floer cohomology as a model for $Kos(W,G)$. Consequently the product structure arises naturally. By Lemma \ref{lem:equivmfHH} and Proposition \ref{prop:mfkos} we have more algebraic way of understanding the ring structure on $Kos(W,G)$. Namely we can construct $Kos(W,G)$ as an endomorphism space of a matrix factorization. In this section we will manifest the algebraic method using matrix factorizations rather than Floer theory. We use the same notation below as in Section \ref{subsec:orbLGB}.

Let $W=xyz$ and $W \boxminus W=x'y'z'-xyz \in \C[x,y,z,x',y',z']$. Recall that there is an isomorphism which preserves corresponding direct summands
\begin{align}\label{eq:kosmf}
\begin{split}
Kos(W,\Z/2)&=Kos(W,1) \oplus Kos(W,\chi) \\
&\cong \big(\Hom_{MF(W\boxminus W)}(\Delta_W^1,\Delta_W^1)\big)^{\Z/2} \oplus 
\big(\Hom_{MF(W\boxminus W)}(\Delta_W^1,\Delta_W^\chi)\big)^{\Z/2}.
\end{split}
\end{align}
%By previous computations we have
%\[ Kos(W,\chi)\cong \C[z]\theta_x \theta_y \oplus \C[z]\theta_x \theta_y\theta_z. \]
%Let us denote the product on $Kos(W,\Z/2)$ by $\cup$. 
We are particularly interested in the multiplication on the nontrivial twisted sector
\[ \cup_{kos}: Kos(W,\chi) \otimes Kos(W,\chi) \to Kos(W,\chi^2)=Kos(W,1).\]
The product $\cup_{kos}$ can be described clearly if we move to $\Hom_{MF(W\boxminus W)}(\Delta_W^1,\Delta_W^\chi)$. A priori it is not easy to recover a morphism of matrix factorizations from a Koszul cocycle, but a recipe is now given in \cite{Leee}. Here, we just give the result of the recipe in our example. Recall that 
\[ Kos(W,\chi)= \big(H^*(K^*(\partial W^\chi)\cdot \theta_x \theta_y)\big)^{\Z/2} \cong \C[z]\theta_x \theta_y \oplus \C[z]\theta_x \theta_y \theta_z.\]
\begin{prop}\label{prop:mfkosqis}
\begin{itemize}
\item Let $h\in \{1,\chi\}$. The following is a quasi-isomorphism of modules: 
\begin{align*}
    (-)_{kos}^h: hom_{MF(W\boxminus W)}(\Delta_W^1,\Delta_W^h) &\to K^*(\partial W^h)\cdot \theta_{I_h}, \\
    \sum_{I,J} f_{I,J}(x,y,z,x',y',z')\theta_I\partial_J &\mapsto \sum_{I: I_h\subset I}f_{I,\emptyset}(x^h,y^h,z^h,x^h,y^h,z^h)\theta_I.
\end{align*}
\item The following is a quasi-inverse of $(-)_{kos}^\chi$:
    \[ e^{\eta_\chi}:K^*(\partial W^\chi)\cdot \theta_x \theta_y \to  hom_{MF(W\boxminus W)}(\Delta_W^1,\Delta_W^\chi)\]
    such that
    \begin{align*} 
    e^{\eta_\chi}(\theta_x \theta_y)=&\; \theta_x\theta_y - z\theta_x\partial_{\theta_x}+\frac{z}{2},\\
    e^{\eta_\chi}(\theta_x\theta_y\theta_z)=&\; \theta_x\theta_y\theta_z + z\theta_x\theta_z\partial_{\theta_x} -y\theta_x\theta_y \partial_{\theta_x}-x'\theta_x\theta_y\partial_{\theta_y}+\frac{z}{2}\theta_z-\frac{y}{2}\theta_y+x'z\theta_x\partial_{\theta_x}\partial_{\theta_y}-\frac{x'z}{2}\partial_{\theta_y}.
    \end{align*}
\end{itemize}
\end{prop}
The cautious readers can readily check (through tedious computation) that right hand sides of the above are indeed closed elements of $hom_{MF(W\boxminus W)}(\Delta_W^1,\Delta_W^\chi)$. However, it does not make sense to "compose" two of above morphisms yet. Given two morphisms which we want to compose, we should translate the latter morphism to $hom_{MF(W\boxminus W)}(\Delta_W^\chi,\Delta_W^{\chi^2})$ so that the composition is in $hom_{MF(W\boxminus W)}(\Delta_W^1,\Delta_W^{\chi^2})=hom_{MF(W\boxminus W)}(\Delta_W^1,\Delta_W^{1}).$ It was shown in \cite{CLe} that the translation is given by
\begin{align*} \chi_*: hom_{MF(W\boxminus W)}(\Delta_W^1,\Delta_W^\chi) &\to hom_{MF(W\boxminus W)}(\Delta_W^\chi,\Delta_W^{\chi^2}),\\
f(x,y,z,x',y',z')\theta_I\partial_J &\mapsto f(x,y,z,\chi^{-1}\cdot x',\chi^{-1}\cdot y',\chi^{-1}\cdot z')\cdot \rho(\chi^{-1})(\theta_I\partial_J).
\end{align*}
We summarize above discussions as follows:
\begin{align*}
\theta_x\theta_y \cup_{kos} \theta_x\theta_y&= \big(\chi_*(e^{\eta_\chi}(\theta_x\theta_y)) \circ e^{\eta_\chi}(\theta_x \theta_y)\big)_{kos},\\
\theta_x\theta_y\theta_z \cup_{kos} \theta_x\theta_y&=\big(\chi_*(e^{\eta_\chi}(\theta_x\theta_y\theta_z)) \circ e^{\eta_\chi}(\theta_x \theta_y)\big)_{kos},\\
\theta_x\theta_y\theta_z \cup_{kos} \theta_x\theta_y\theta_z&=\chi_*(e^{\eta_\chi}(\theta_x\theta_y\theta_z)) \circ e^{\eta_\chi}(\theta_x \theta_y\theta_z)\big)_{kos}.
\end{align*}
The following computations are now straightforward.
\begin{align*}
    & \chi_*(e^{\eta_\chi}(\theta_x\theta_y)) \circ e^{\eta_\chi}(\theta_x \theta_y)  \\
    =&(\theta_x\theta_y - z\theta_x\partial_{\theta_x}+\frac{z}{2})\cdot (\theta_x\theta_y - z\theta_x\partial_{\theta_x}+\frac{z}{2}) = \frac{z^2}{4}, \\ & \\
    & \chi_*(e^{\eta_\chi}(\theta_x\theta_y\theta_z)) \circ e^{\eta_\chi}(\theta_x \theta_y) \\
    =& (\theta_x\theta_y\theta_z + z\theta_x\theta_z\partial_{\theta_x} +y\theta_x\theta_y \partial_{\theta_x}-x'\theta_x\theta_y\partial_{\theta_y}+\frac{z}{2}\theta_z+\frac{y}{2}\theta_y+x'z\theta_x\partial_{\theta_x}\partial_{\theta_y}-\frac{x'z}{2}\partial_{\theta_y}) \\
    & \cdot (\theta_x\theta_y - z\theta_x\partial_{\theta_x}+\frac{z}{2}) \\
    =& \frac{z^2}{4}\theta_z + \frac{yz}{4}\theta_y - \frac{x'z}{2}\theta_x + c_1  \partial_{\theta_x}+c_2 \partial_{\theta_y}+c_3\partial_{\theta_x}\partial_{\theta_y}, \\ & \\
   & \chi_*(e^{\eta_\chi}(\theta_x\theta_y\theta_z)) \circ e^{\eta_\chi}(\theta_x \theta_y\theta_z) \\
   =& (\theta_x\theta_y\theta_z + z\theta_x\theta_z\partial_{\theta_x} +y\theta_x\theta_y \partial_{\theta_x}-x'\theta_x\theta_y\partial_{\theta_y}+\frac{z}{2}\theta_z+\frac{y}{2}\theta_y+x'z\theta_x\partial_{\theta_x}\partial_{\theta_y}-\frac{x'z}{2}\partial_{\theta_y}) \\
   & \cdot (\theta_x\theta_y\theta_z + z\theta_x\theta_z\partial_{\theta_x} -y\theta_x\theta_y \partial_{\theta_x}-x'\theta_x\theta_y\partial_{\theta_y}+\frac{z}{2}\theta_z-\frac{y}{2}\theta_y+x'z\theta_x\partial_{\theta_x}\partial_{\theta_y}-\frac{x'z}{2}\partial_{\theta_y})\\ 
   =& \frac{x'y}{2}\theta_x\theta_y+\frac{yz}{2}\theta_y\theta_z-\frac{x'z}{2}\theta_x\theta_z+\frac{x'yz}{4}+ d_1  \partial_{\theta_x}+d_2 \partial_{\theta_y}+d_3\partial_{\theta_x}\partial_{\theta_y}
\end{align*}
where $c_1,c_2,c_3,d_1,d_2,d_3 \in S\langle \theta_x,\theta_y,\theta_z\rangle$. 

Applying Proposition \ref{prop:mfkosqis}, we have
\begin{align*} 
\theta_x \theta_y \cup_{kos} \theta_x\theta_y & = \frac{z^2}{4}, \\
\theta_x \theta_y \theta_z \cup_{kos} \theta_x \theta_y &= \frac{z^2}{4}\theta_z + \frac{yz}{4}\theta_y - \frac{xz}{2}\theta_x, \\
\theta_x \theta_y \theta_z \cup_{kos} \theta_x \theta_y \theta_z &= 0. 
\end{align*}

\bibliographystyle{amsalpha}
\bibliography{geometry}
%\begin{thebibliography}{}
%\bibitem{CHL1} C.-H. Cho, H. Hong and S.-C. Lau, {\em Localized mirror functor for Lagrangian immersions, and homological mirror symmetry for $\mathbb{P}^1_{a,b,c}$}, J. Differential Geom. 106 (2017), no. 1, 45-126.
%\end{thebibliography}

\end{document}